\documentclass{amsart}

\usepackage{amssymb}
\usepackage{epsfig}  		\usepackage{epic,eepic}   
\usepackage{amsmath}
\usepackage{amsxtra,amssymb,latexsym,amscd,amsthm}
\usepackage{array}
\usepackage{multirow}
\usepackage{adjustbox}
\usepackage{caption}
\usepackage{mathtools}
\usepackage{ mathrsfs }
\usepackage{fancyhdr}
\usepackage{cases}
\usepackage{textcomp}
\usepackage{tikz}
\usepackage{tikz-cd}
\usepackage{graphicx}
\usepackage{fancybox}
\usepackage{hyperref}
\usepackage[all]{xy}
\usepackage{xcolor}
\usepackage{ tipa }
\usepackage{appendix}
\usepackage[portrait, top=2.5cm, bottom=2.5cm, left=2.5cm, right=2.5cm] {geometry}

\usepackage[backend=biber, style=alphabetic]{biblatex}
\usepackage{float}
\restylefloat{table}
\newtheorem{thm}{Theorem}[section]
\newtheorem{prop}[thm]{Proposition}
\newtheorem{lm}[thm]{Lemma}

\newtheorem{conj}[thm]{Conjecture}

\theoremstyle{definition}
\newtheorem{definition}[thm]{Definition}

\newtheorem{rmk}[thm]{Remark}

\def\Ocal{\mathcal O}

\def\Dcal{\mathcal D}
\def\Mcal{\mathcal M}

\def\Bcal{\mathcal B}

\def\Bscr{\mathscr B}
\def\Rscr{\mathscr R}

\def\Pbb{\mathbb P}
\def\Rbb{\mathbb R}
\def\Cbb{\mathbb C}

\def\Zbb{\mathbb Z}

\def\Tbb{\mathbb T}
\def\Qbb{\mathbb Q}
\def\Gbb{\mathbb G}

\def\m{\mathrm m}
\def\Spec{\mathrm{Spec\ }}
\def\reg{\mathrm{reg}}
\def\Res{\mathrm{Res}}
\def\div{\mathrm{div}}
\def\Corr{\mathrm{Corr}}
\def\CHM{\mathrm{CHM}}
\def\CH{\mathrm{CH}}
\def\Hom{\mathrm{Hom}}
\def\id{\mathrm{id}}
\def\Gal{\mathrm{Gal}}

\def\dR{\mathrm{dR}}
\def\SmProj{\mathrm{SmProj}}
\def\pr{\mathrm{pr}}

\newcommand{\et}{\textnormal{\'et}}
\numberwithin{equation}{section}

\begin{document}

\title{The Mahler measure of exact polynomials in three variables}

\author{Trieu Thu Ha}
\address{UMPA, ÉNS de Lyon, 46, allée d’Italie, 69007, France}
\email{thu-ha.trieu@ens-lyon.fr}
\thanks{This work was performed within the framework of the LABEX MILYON (ANR-10-LABX-0070) of Université de Lyon, within the program ``Investissements d'Avenir'' (ANR-11-IDEX-0007) operated by the French National Research Agency
(ANR)}
\keywords{Mahler's measure, regulators, polylogarithms, motivic cohomology, residues.}
\subjclass[2020]{Primary:  19F27; Secondary: 11G55, 11R06, 19E15}

\begin{abstract}
We prove that under certain explicit conditions, the Mahler measure of a three-variable polynomial can be expressed in terms of elliptic curve $L$-values and Bloch-Wigner dilogarithmic values, conditionally on Beilinson's conjecture. In some cases, these dilogarithmic values simplify to Dirichlet $L$-values. The proof involves a construction of an element in $K_4^{(3)}$ of a smooth projective curve over a number field. This generalizes a result of Lalín \cite{Lal15} for the polynomial $z + (x+1)(y+1)$. We apply our method to several other Mahler measure identities conjectured by Boyd and Brunault.
\end{abstract}

\maketitle

\section*{Introduction}
Let $P(x_1, \dots, x_n) \in \Cbb[x_1^{\pm 1}, \dots, x_n^{\pm 1}]$ be a nonzero Laurent polynomial. The (logarithmic) Mahler measure of $P$ is defined by 
\begin{equation}\label{19}
\begin{aligned}
 \m(P) &= \dfrac{1}{(2\pi i)^n} \int_{\Tbb^n} \log |P(x_1, \dots, x_n)| \ \dfrac{d x_1}{x_1} \wedge \cdots \wedge \dfrac{d x_n}{x_n},
\end{aligned}
\end{equation}
where $\Tbb^n : |x_1|=\dots = |x_n|=1$ is the $n$-dimensional torus. This quantity was firstly introduced by Mahler \cite{Mah62} in 1962. 

In 1997, Deninger \cite{Den97} linked the Mahler measure of polynomials $P(x_1, \dots, x_n)$ under certain conditions to the motivic cohomology of $V_P$, where $V_P$ is the zero locus of $P$ in $\Cbb^n$. This allowed him to place the Mahler measure in the very general framework of Beilinson's conjectures on special values of $L$-functions. More precisely, Deninger defined the following chain
$$\Gamma = \{(x_1, \dots, x_n) \in \Cbb^n : P(x_1, \dots, x_n)=0, |x_1|=\dots = |x_{n-1}|=1, |x_n|\ge 1\}.$$
He showed that if $\Gamma$ is contained in the regular locus $V_P^\reg$ of $V_P$, then there is a differential $(n-1)$-form $\eta(x_1, \dots, x_n)$ on $\Gbb_m^n$ such that its restriction to $V_P$ represents the regulator of the Milnor symbol $\{x_1, \dots, x_n\}$, and we have
\begin{equation}\label{1}
    \m(P) = \m(\tilde{P}) + \dfrac{(-1)^{n-1}}{(2\pi i)^{n-1}} \int_\Gamma \eta(x_1,\dots, x_n),
\end{equation}
where $\tilde{P}$ is the leading coefficient of $P$ seen as a polynomial in $x_n$. 

From now on we assume that $P$ has rational coefficients and $\Gamma$ is contained in $V_P^\reg$. If $\partial \Gamma  = \emptyset$, then $\Gamma$ is a cycle. Then Deninger found out that in certain situations, identity \eqref{1} together with Beilinson's conjecture imply that $\m(P)$ can be expressed in terms of the $L$-function of the motive $H^{n-1}(\overline{V_P})$, where $\overline{V_P}$ is a smooth compactification of $V_P$.  As an example, he showed that under the Beilinson conjecture, 
\begin{equation}\label{20}
    \m\left(x+\dfrac{1}{x}+y+\dfrac{1}{y}+1\right) \stackrel{?}{=} L'(E_{15}, 0),
\end{equation}
where $E_{15}$ is the elliptic curve (of conductor 15) defined by $x+1/x+y+1/y+1 = 0$. In this example, $\partial \Gamma \neq \emptyset$ but a symmetry argument reduces this to the case $\partial \Gamma = \emptyset$. It was completely shown (without assuming the Beilinson conjecture) by Rogers and Zudilin \cite{RZ14} in 2014.

Boyd \cite{Boy98} conjectured, based on numerical evidence, that
\begin{equation}\label{21}
    \m \left(x + \dfrac{1}{x} + y + \dfrac{1}{y} + k\right)  \stackrel{?}{=} r_k \ L'(E_{N(k)}, 0),
\end{equation}
where $k \in \Zbb \setminus \{0, \pm 4\}$, $r_k \in \Qbb^\times$ and $E_{N(k)}$ is the elliptic curve (of conductor $N(k)$) obtained as a smooth compactification of the zero set of $P_k = x + 1/x + y + 1/y + k$. Until now, identity (\ref{21}) is only proved for a finite number of $k$:
$$k \in \{-4\sqrt{2}, -2\sqrt{2}, 1, 2, 3, 2\sqrt{2},  3\sqrt{2}, 5, 8, 12, 16, i, 2i, 3i, 4i, \sqrt{2}i\},$$
by the works of Brunault, Lal\'{i}n, Rodriguez-Villegas, Rogers, Samart, and Zudilin (see \cite{Bru16}, \cite{Lal10}, \cite{LSZ16}, \cite{LR07}, \cite{Rod99}, \cite{RZ12}), \cite{RZ14}, \cite{Zud14}.

The case $\partial \Gamma \neq \emptyset$ is more difficult, Maillot \cite{Mai03} suggested we should look at the variety 
\begin{equation}\label{maillot}
    W_P := V_P \cap V_{P^*},
\end{equation} 
where $P^* (x_1,\dots ,
x_n) = \bar P(x_1^{-1}, \dots, x_n^{-1})$. We call $W_P$ the \textit{Maillot variety}. If $P$ is an \textit{exact} polynomial, i.e.,  $\eta = d \omega$, where $\omega$ is a differential form on $V_{P}^{\reg}$, then Stokes' theorem gives
$$\m(P) = \m(\tilde{P}) + \dfrac{(-1)^{n-1}}{(2\pi i)^{n-1}} \int_{\partial \Gamma} \omega.$$
Moreover, $\partial \Gamma$ is contained in $W_P$, hence we can hope that $\m(P)$ is related to the cohomology of $W_P$. In this direction, Lalín showed the following result. 
\begin{thm}[{\cite[Theorem 2]{Lal15}}]\label{lalinthm}
  Assume that  $P\in \Qbb[x, y, z]$ is irreducible and satisfies the following conditions:
\begin{enumerate}
    \item[(i)] $W_P$ is birationally equivalent to an elliptic curve $E$ over $\Qbb$,
    \item[(ii)]  $\partial \Gamma$ defines  an element of $H_1(E(\Cbb), \Zbb)^+$, the invariant part of $H_1(E(\Cbb), \Zbb)$ under the action induced by the complex conjugation on $E(\Cbb)$,
    \item[(iii)] $ x \wedge y \wedge z = \sum_j f_j \wedge (1-f_j) \wedge g_j \text{ in } \bigwedge^3 \Qbb(V_P)^\times$, for some functions $f_j, g_j \in  \Qbb(V_P)^\times$,
    \item[(iv)] $x \wedge y \wedge z = 0 \text{ in } \bigwedge^3 \Qbb(E)^\times$,
    \item[(v)] $\sum_j v_p(g_j) \{f_j(p)\}_2 = 0 \text{ in }  \Zbb[\Pbb^1_{\bar \Qbb}]/R_2(\bar \Qbb) \text{ for all } p \in E(\bar \Qbb)$, 
\end{enumerate}
where $R_2(\bar \Qbb)$ is the subgroup generated by the five-term relations (see equation (\ref{130})), and $v_p(g_j)$ is the vanishing order at $p$ of $g_j$ seen as a function on $E$. Then under Beilinson's conjecture, 
\begin{equation}\label{lalin}
    \m(P) = m(\tilde P) + a\cdot L'(E, -1), \  a \in \Qbb.
\end{equation}  
\end{thm}

Note that the condition (3) implies that $P$ is exact (see Remark \ref{6.3}). In this article, we relax Lalín's conditions in order to deal with Mahler measure identities which are more general than (\ref{lalin}), for example, containing also Dirichlet $L$-values. We only assume that $W_P$ is of genus 1 and we do not require the conditions (iv)-(v) above. Recall that the Bloch-Wigner dilogarithm function $D : \Pbb^1(\Cbb) \to \Rbb$ is defined by
\begin{equation}\label{dilog}
        D(z) = \begin{cases}
            \mathrm{Im}\left(\sum_{k=1}^\infty \frac{z^k}{k^2}\right) + \arg (1-z) \log |z| &(|z| \le 1),\\
            -D(1/z) &(|z| \ge 1).
        \end{cases}
\end{equation}
For any field $F$, we denote by $\Bcal(F)$ the Bloch group of $F$ tensored with $\Qbb$ (see \cite[Section 2]{Zag91}). Let $\tau$ be the involution of $\Gbb^3_m$  given by $(x, y, z) \mapsto (1/x, 1/y, 1/z)$. Since $P$ has rational coefficients, $\tau$ induces an involution of $W_P$. For $A$ is a abelian group, we denote by $A_\Qbb := A \otimes \Qbb$. Let us state our main theorem here.

\begin{thm}\label{0}
Assume $P \in \Qbb[x, y, z]$ is irreducible and that $W_P$ is a curve of genus 1. Let $C$ be the normalization of $W_P$. Suppose that 
\begin{equation}
     x \wedge y \wedge z = \sum_j f_j \wedge (1-f_j) \wedge g_j \text{ in } \bigwedge^3 \Qbb(V_P)^\times_\Qbb,
\end{equation}
for some functions $f_j, g_j$ on $V_P$. Let $S$ be the closed subscheme of $C$ consisting of the zeros and poles of the functions $g_j$ and $ g_j \circ \tau$  on $C$ for all $j$. Then for $p \in S$, 
\begin{equation*}
    u_p := \sum_j v_p(g_j) \{f_j(p)\}_2 + v_p(g_j \circ \tau) \{f_j \circ \tau (p)\}_2
\end{equation*}
define elements in the Bloch group $\Bcal(\Qbb(p))$, where $\Qbb(p)$ is the residue field of $C$ at $p$.

Assume that the Deninger chain $\Gamma$ is contained in $V_P^\reg$ and that its boundary $\partial \Gamma$ is contained in $W_P^\reg$, then $\partial \Gamma$ defines an element in $H_1(C(\Cbb), \Zbb)^+$. If $u_p = 0$ for all $p \in S$, then under Beilinson's conjecture for the curve $C$ (see 
 Conjecture \ref{BeiConj}), we have
\begin{equation*}
    \m(P) = \m(\tilde P) + a \cdot L'(E, -1), \qquad (a \in \Qbb),
\end{equation*}
where $E$ is the Jacobian of $C$. Otherwise, let $S'$ be the subset of  $S$ consisting of the points  $p$ such that $u_p \neq 0$. Let $K$ be the splitting field of $S'$ in $\Cbb$. Let $\Ocal$ be any fixed point in $S'(K)$. Assume that the difference of any two geometric points in $S'(K)$ has finite order dividing $N$ in $E(K)$, then under Beilinson's conjecture for the curve $C$ (see Conjecture \ref{BeiConj}),
\begin{equation}\label{123}
     \m(P) = \m (\tilde P) + a \cdot L'(E,-1) + \dfrac{1}{4N\pi} \sum_{q \in S'(K)\setminus \{\Ocal\}}b_q\cdot D(u_q), \qquad (a \in \Qbb, b_q\in \Zbb),
\end{equation}
where for $q\in S'(K)$ supported on a closed point $p \in S'$, we define $u_q$  to be the image of $u_p$ under the map $\Bcal(\Qbb(p)) \hookrightarrow \Bcal(K)$ induced by the embedding $\Qbb(p) \xhookrightarrow{q} K$.
\end{thm}

The rational number $a$ comes from Beilinson's conjecture, and does not depend on the choice of $\Ocal$, but the integer numbers $b_q$ actually do. However, the $D$-values in identity \eqref{123} are independent of the choice of $\Ocal$. Indeed, when we remove $\Ocal$ from the sum, the complex conjugation of $\Ocal$ maintains the $D$-values in identity \eqref{123}. Now let us use Theorem \ref{0} to investigate several conjectural Mahler measure identities of the following types:

a) \textit{Pure identities}: $\m(P)  \stackrel{?}{=}  a \cdot  L'(E, -1)$ for some $a \in \Qbb^\times$. Table \ref{table:1}  consists of pure identities that we prove up to a rational factor and conditionally on Beilinson's conjecture. Identity (3) in Table \ref{table:1} was studied by Lalín \cite{Lal15}.
\begin{table}[h!]
\centering
\begin{tabular}{ |c||p{9.5cm}|c|c|c| } 
 \hline
 & \multicolumn{1}{c|}{$P$} & $E$ & $a$ & References\\ 
 \hline
 1. & $(1 + x)(1 + y)(x + y) + z$ & $14a4$ & -3 & \cite[p. 81]{BZ20} \\ 
 \hline
 2. & $1 + x + y + z + xy + xz + yz$ & $14a4$ & -5/2 & \cite{Bru20} \\ 
 \hline
 3. & $(x+1)(y+1)+z$  & 15a8 & -2 &  \cite{Boy06} \\ 
 \hline
 4. & $(x+1)^2 + (1-x)(y + z)$ & $20a1$ & -2 & \cite{Boy06}, \cite[p. 81]{BZ20}\\
 \hline
 5. & $1+ (x+1)y +(x-1)z$ & \multirow{4}{*}{$21a1$} & \multirow{4}{*}{-5/4} & \cite{Boy06}\\
 \cline{1-2} \cline{5-5}
 6. & $\frac{1}{2} (x+2) + (x^2+x+1)y + (x^2-1)z $& & & \multirow{3}{*}{\cite{LN23}}\\
 \cline{1-2}
 7. & $ \frac{1}{2} (x^2-2x+2) +(x^4 - x^3 + x^2 -x +1)y + (x^4-x^3+x-1)z$ & & & \\
\cline{1-2}
 8. & $ \frac{1}{2}(x^4 + x+ 2) + (x^5 + x^4 + x + 1)y + (x^5 -1 )z$ & & &\\
 \hline
 9. & $(x+1)^2(y+1) + z$ & $21a4$ & -3/2 & \cite{Boy06}, \cite[p. 81]{BZ20} \\
 \hline
 10. & $(1+x)^2 + y + z$ & $24a4$ & -1 & {\cite{Boy06}}\\
 \hline
 11. & $1 + x + y + z + xy + xz + yz - xyz$ & $36a1$ & -1/2 & \cite{Bru20} \\ 
  \hline
  12. & $1+xy+(1+x+y)z$ & $90b1$ & -1/20 & \cite{Bru20}\\
 \hline
 13. & $(x+1)^2 + (x-1)^2y + z$ & $225c2$ & -1/48 & \cite{Boy06}, \cite{Bru20}\\
 \hline
\end{tabular}
\centering
\caption{Pure identities $\m(P) \stackrel{?}{=} a \cdot L'(E, -1)$.}
\label{table:1}
\end{table}
Notice that Theorem \ref{lalinthm} of Lalín does not apply to  identity (5) in Table \ref{table:1} as  it violates condition (v). Identities (6), (7) and (8) are conjectured by Lalín and Nair in \cite{LN23}, in fact, they showed that by some changes of variables, the Mahler measure of polynomials (5), (6), (7) and (8) are equal. Moreover, from Table \ref{table:1}, we have the following relations (under the Beilinson conjecture)
\begin{equation*}
    \m((1+x)(1+y)(x+y) +z) \stackrel{?}{\sim}_{\Qbb^\times}\m(1+x+y+z+xy+ xz + yz),
\end{equation*}
\begin{equation*}
    \m(1+(x+1)y + (x-1)z)  \stackrel{?}{\sim}_{\Qbb^\times}  \m((x+1)^2(y+1) + z).
\end{equation*}
\noindent Besides, we also give an example that Theorem \ref{0} does not imply (see Section \ref{www}(g)):
\begin{equation}\label{counterex}
\m((1+x)(1+y)+(1-x-y)z) = \dfrac{1}{288}L'(E_{450c1}, -1).
\end{equation}

b) \textit{Identities with Dirichlet $L$-values}:  
    \begin{equation}\label{mixed}
        \m(P(x, y, z)) \stackrel{?}{=} a \cdot L'(E, -1) + \sum_\chi b_\chi \cdot L'(\chi, -1) \qquad (a \in \Qbb, b_\chi \in \Qbb^\times),
    \end{equation}
    where $\chi$ are odd quadratic Dirichlet characters.  In cases where the coefficient $a$ is nonzero, such identities are referred to as \textit{mixed identities}. Table \ref{mix1} consists of mixed identities we prove (up to rational factors) under Beilinson's conjecture for genus 1 curves (see Sections \ref{ly}(a) and \ref{ly}(e)).
\begin{table}[h!]
\centering
\begin{tabular}{ |c||p{7cm}|c|c|c|c|c|c| } 
 \hline
 & \multicolumn{1}{c|}{$P$} & $E$ & $a$ & $b_1$ & $b_2$ & References\\ 
  \hline
1. & $1+ (x^2 - x + 1)y + (x^2 + x + 1)z$ & $45a2$ & -1/6 & 1 & 0 & \multirow{2}{*}{\cite{Bru20}} \\ 
 \cline{1-6}
2. & $x^2+1 +(x+1)^2 y + (x^2 -1)z$ & $48a1$ & -1/10 & 0& 1 &\\  
 \hline
\end{tabular}
\centering
\caption{ Conjectural identities $\m(P) \stackrel{?}{=} a \cdot L'(E, -1) + b_1 \cdot L'(\chi_{-3}, -1) + b_2 \cdot L'(\chi_{-4}, -1)$.}
\label{mix1}
\end{table}
The first polynomial in Table \ref{mix1}  does not satisfy conditions
(iv)-(v) of Theorem \ref{lalinthm} of Lalín. Moreover, the Maillot variety $W_P$ is a curve of genus 1 and does not have any rational point, hence violates condition (i) also.  For the second polynomial in Table \ref{mix1}, the Maillot variety $W_P$ is a union of a line and a nonsingular curve of genus 1.  We also give an example to which our theorem does not apply (see Section \ref{ly}(d))
\begin{equation}\label{mix3}
    \m(x^2 + x + 1 + (x^2 +x +1)y + (x-1)^2 z) = -\dfrac{1}{12} L'(E_{72a1}, -1) + \dfrac{3}{2} L'(\chi_{-3}, -1),
\end{equation}
which is conjectured by Brunault \cite{Bru20}.

By a method of Lalín in \cite[Example 4.2]{Lal15}, we prove without assuming  Beilinson's conjecture the following Mahler measure identities involving only Dirichlet $L$-values (see Section \ref{ly}(b)-(c)):
\begin{table}[H]
\centering
\begin{tabular}{|c||p{9cm}|c|c|c|c| } 
 \hline
 & \multicolumn{1}{c|}{$P$} & $b_1$ & $b_2$ & References\\ 
 \hline
 1. & $1 + (x + 1)(x^2 + x + 1)y + (x+1)^3z$ & 3 & 0& \multirow{8}{*}{\cite{Bru20}}\\ 
 \cline{1-4}
 2. & $x^2+1 + (x^2 + x + 1)y + (x+1)^3z$ & \multirow{2}{*}{7/2} & \multirow{2}{*}{0} & \\ 
 \cline{1-2}
 3. & $x^2+1 + (x+1)(x^2 + x + 1)y + (x+1)^3z$& & &\\
 \cline{1-4}
 4. & $x^2+1 + (x + 1)(x^2 + x + 1)y + (x - 1)(x^2 - x + 1) z$ & \multirow{2}{*}{0} & \multirow{2}{*}{7/3} & \\
  \cline{1-2}
5. & $(x + 1)(x^2 + 1) + (x + 1)(x^2 + x + 1)y + (x - 1)(x^2 - x + 1) z$ &  & & \\ 
 \cline{1-4}
6. & $x^2 + 1 + (x+1)^2y + (x-1)^2z$ & 0 & 2& \\ 
\cline{1-4}
7. & $x^2 + 1 + (x+1)^3y + (x-1)^3z$ & \multirow{2}{*}{0} & \multirow{2}{*}{3}&  \\ 
\cline{1-2}
8. & $(x + 1)(x^2 + 1) + (x+1)^3y + (x-1)^3z$ &  & &\\  
\hline
\end{tabular}
\centering
\caption{$\m(P) = b_1 \cdot L'(\chi_{-3}, -1) + b_2 \cdot L'(\chi_{-4}, -1)$.}
\label{table:3}
\end{table}

The article contains five sections. In the first three sections, we recall some tools that needed for our constructions. In Section \ref{DeligneCoho}-\ref{Beireg}, we recall the definitions of the Deligne cohomology and some facts about motivic cohomology. In particular, we give an explicit isomorphism between Chow motives of a genus 1 smooth projective curve and its Jacobian. We then recall briefly the regulator maps and Beilinson's conjecture for smooth projective curves of genus 1. In Section \ref{Goncharov}, we bring back Goncharov's polylogarithmic complexes and his regulator maps on the cohomology of these complexes. In Section \ref{DeJeu}, we recall De Jeu's polylogarithmic complexes and discuss his maps connecting the cohomology of Goncharov's polylogarithmic complexes to the motivic cohomology. In particular, in Section \ref{regmaps}, given a 2-cocycle in weight 3 Goncharov's polylogarithmic complex, we compare its Goncharov regulator and the Beilinson regulator of its image under the map defined in Section \ref{dJmapsection}.  In Section \ref{Deligne}, given an exact polynomial $P$ in $\Qbb[x, y, z]$, we construct an element in the Deligne cohomology of an open subset of the normalization $C$ of the Maillot variety $W_P$ \eqref{maillot}. We then relate this element to the Mahler measure of $P$ in Section \ref{relate1}. In Section \ref{motivic}, we construct explicitly an element in $K_4^{(3)}(C)$  satisfying that its Beilinson regulator has connection with the Deligne cohomology element constructed in Section \ref{Deligne}. We then prove the main theorem in Section \ref{relate}. We end the article with Section \ref{Examples}, where we study the conjectural Mahler measure identities mentioned above.

\bigskip 

\textbf{Acknowledgement.} I would like to express my gratitude to my supervisor, François Brunault, who brought me the idea to work on this subject as well as spent an immense amount of time to discuss with me. I also would like to thank for his generosity in sharing his conjectural Mahler measure identities. I would like to thank Rob de Jeu for his enlightening explanation in the construction of his maps and in computing regulators integrals in Section \ref{DeJeu}. I also would like to thank Nguyen Xuan Bach for several fruitful discussions. 

\section{The Beilinson regulator map }\label{preliminary}

\subsection{Deligne cohomology}\label{DeligneCoho}  Deligne cohomology of a smooth complex algebraic variety $X$ is firstly introduced by Deligne in 1972, it is given by the hypercohomology of 
\begin{equation}\label{49}
    0 \to \Zbb(n) \to \Ocal_X \to \Omega_X^1\to \Omega_X^2 \to \dots \to \Omega_X^{n-1} \to 0,
\end{equation}
where the constant sheaf $\Zbb(n) := (2\pi i)^n \Zbb$ is placed in degree 0 and $\Omega_X^j$ is the sheaf of holomorphic $j$-forms on $X$ placed in degree $j+1$. Burgos \cite{Bur97} then showed that this hypercohomology can be  the cohomology  of a single complex. Let us recall briefly Burgos' construction (see \cite{Bur97}, \cite{BZ20}). 

Let $X$ be a smooth complex algebraic variety of dimension $d$. Let  $(\bar X, \iota)$ be a good compactification of $X$, i.e., that $\bar X$ is a smooth proper variety and $\iota : X \hookrightarrow \bar X$ is an open immersion such that $D := \bar X - \iota(X)$ is locally given by $z_1 \dots z_m = 0$ for some analytic local coordinates $z_1, \dots, z_d$ on $\bar X$ and $m \le d$. 

\begin{definition}[{\cite[Proposition  1.1]{Bur97}}]\label{logaalongD}
A complex smooth differential form $\omega$ on $X$ \textit{has logarithmic singularities along $D$ } if locally $\omega$ belongs to the algebra generated by the smooth forms on $\bar X$ and $\log |z_i|, \dfrac{dz_i}{z_i}$, $\dfrac{d\bar z_i }{\bar z_i}$, for $1 \le i \le m$, where $z_1 \dots z_m = 0$ is a local equation of $D$. For $\Lambda \in \{\Rbb, \Cbb\}$, $E^n_{X, \Lambda}(\log D)$ denotes the space of such $\Lambda$-valued smooth differential forms of degree $n$ on $X$.
\end{definition}
We have $ E^n_{X, \Cbb} (\log D) = \bigoplus_{p+q = n} E^{p, q}_{X, \Cbb}(\log D)$,
where $E^{p, q}$ is the subspace of type $(p, q)$-forms. We denote by $\bar \partial : E^{p, q} \to E^{p, q+1}$ and $\partial : E^{p, q} \to E^{p+1, q}$ as the usual operators and $d = \partial + \bar \partial$. Burgos defined
\begin{equation*}
    E^*_{\log, \Lambda}(X) = \varinjlim_{(\bar X, \iota) \in \mathcal{I}^{opp}} E^*_{X, \Lambda}(\log D),
\end{equation*}
where $\mathcal{I}$ is the category of good compactification of $X$. He then introduced the following complex.
\begin{definition}\label{54} ({\cite[Theorem 2.6]{Bur97}})
For any integers $j, n \ge 0$,
\begin{equation*}
    E_j(X)^n := \begin{cases}
    (2 \pi i)^{j-1} E^{n-1}_{\log, \Rbb}(X) \cap \left(\bigoplus_{p+q = n-1; p, q < j} E^{p, q}_{\log, \Cbb}(X)\right) &\text{if } n \le 2j-1,\\
    (2\pi i)^j E^n_{\log, \Rbb} (X) \cap \left(\bigoplus_{p+q = n; p, q \ge j} E^{p, q}_{\log, \Cbb}(X)\right) &\text{if } n \ge 2j,
    \end{cases}
\end{equation*}
\begin{equation*}
    d^n \omega :=\begin{cases}
    -\text{pr}_j(d \omega) &\text{if } n < 2j-1,\\
    -2 \partial\bar \partial \omega &\text{if } n = 2j -1,\\
    d \omega &\text{if } n \ge 2j,
    \end{cases}
\end{equation*}
where $\text{pr}_j$ is the projection $\bigoplus_{p, q} \to \bigoplus_{p, q< j}$. Denote by $E_j(X)$ the complex $(E_j(X)^n, d^n)_{n\ge 0}$.
\end{definition}

\begin{definition}[\textbf{Deligne Cohomology}]\label{52}(\cite[Corollary 2.7]{Bur97})
Let $X$ be a smooth complex algebraic variety. The Deligne cohomology of $X$ is the cohomology of the complex $E_j(X)$, that is
$$H^n_\Dcal (X, \Rbb(j)) = H^n (E_j(X))\quad \text{for } j, n\ge 0.$$
\end{definition}
As the canonical map
$E^*_{X, \Cbb}(\log D) \to E^*_{\log, \Cbb} (X)$
is a quasi-isomorphism (cf. \cite[Theorem 1.2]{Bur94}), we can use $E^*_{X, \Lambda}(\log D)$ for some good compacitification of $X$ instead of $E_{\log, \Lambda}^* (X)$  in Definition \ref{54}.

\begin{rmk}\label{1.5} For the case $j > \dim X \ge 1$ or $j > n$, $H^n_\Dcal(X, \Rbb(j))$ is canonically isomorphic to the de Rham cohomology $H^{n-1} (X, (2\pi i)^{j-1} \Rbb)$ by the canonical map sending a Deligne cohomology class to its de Rham cohomology class (see \cite[Section 8.1]{BZ20}).
In particular, for $j > 1$, we have $H^1_\Dcal(\Spec (\Cbb), \Rbb(j)) \simeq \Rbb(j-1)$.
\end{rmk}

Let $X$ be a smooth variety over $\Rbb$. Let $X(\Cbb)$ denote the set of complex points of $X\times_\Rbb \Cbb$. Denote by $c$ the complex conjugation on $X(\Cbb)$. For a complex differential form $\omega$ on $X(\Cbb)$, we define an involution $F_\dR (\omega) := c^*(\bar \omega)$. It acts on the complex $E_j(X(\Cbb))$, hence acts on the Deligne cohomology.

\begin{definition}\label{94}({\cite[Remark 6.5]{Bur97}}) Let $X$ be a smooth variety over $\Rbb$. The Deligne cohomology of $X$ is defined by 
$$H^n_\Dcal (X, \Rbb(j)) := H^n_\Dcal(X \times_\Rbb \Cbb, \Rbb(j))^+,$$
where ``+'' denotes the invariant part under the action of the involution $F_{\text{dR}}$.
\end{definition}

Let $X$ be a smooth real or complex variety, there is a cup-product in Deligne cohomology 
\begin{equation}
\cup : H^n_\Dcal(X, \Rbb(j)) \otimes H^{m}_\Dcal(X, \Rbb(k)) \to H^{n+m}_\Dcal (X, \Rbb(j + k)), 
\end{equation}
(see \cite[Theorem 3.3]{Bur97}). It is graded commutative  (i.e., $\alpha \cup \beta = (-1)^{mn} \beta \cup \alpha$), and associative. In the case $n < 2j$, $m< 2k$, for $\alpha \in  H^n_\Dcal(X, \Rbb(j))$ and $\beta \in H^m_\Dcal(X, \Rbb(k))$, we have that $\alpha \cup \beta$ is represented by
\begin{equation}\label{65}
 (-1)^n r_j(\alpha) \wedge \beta + \alpha \wedge r_k (\beta),
\end{equation}
where $r_\ell(\alpha) :=
\partial (\alpha^{\ell -1, n-\ell}) - \bar \partial (\alpha^{n-\ell, \ell-1})$.

\subsection{Chow motives}\label{Chow}
In this section, we discuss the Chow motives of smooth projective curves. For more details, we refer to \cite{MNP13} or \cite{Scho94}. Recall that the Chow groups of a variety $X$ are $\mathrm{CH}^n(X) := Z^n(X)/\sim$, where $Z^n(X)$ is the free abelian group generated by irreducible subvarieties of $X$ of codimension $n$, and $\sim$ is the rational equivalence (see \cite[Section 1.2]{MNP13}).  If $\phi : X \to Y$ is a morphism of varieties, one has the following homomorphisms $$\phi^* : \CH^n (Y) \to \CH^n(X), \qquad \phi_* : \CH^n(X) \to \CH^{n+ \dim Y - \dim X} (Y).$$
Let $k$ be a field and $r \ge 0$. Let $X, Y \in \mathrm{SmProj}(k)$. If $X$ is of pure dimension $d$, the group of \textit{correspondences of degree $r$} is given by
$\Corr^r (X, Y) := \CH^{d+r}(X\times Y) \otimes \Qbb$. If $X = \coprod X_d$ is a decomposition of subschemes with $X_d$ is of pure dimension $d$, then $\Corr^r(X, Y) := \bigoplus \Corr^r(X_d, Y).$
Let $X, Y, Z \in \SmProj(k)$ and $f \in \Corr^r(X, Y)$, $g \in \Corr^s(Y, Z)$, the composition of correspondences
\begin{equation*}
    \Corr^r(X, Y) \times \Corr^s
    (Y, Z) \to \Corr^{r+s}(X, Z)
\end{equation*}
is defined by
\begin{equation*}
    (f, g) \mapsto g \circ f := \pr_{13*}\ (\pr_{12}^*\  f \cdot \pr_{23}^*\ g) = \pr_{13*} \ (f \times Z \cdot X \times g),
\end{equation*}
where $\pr$ is the canonical projection and $\cdot$ is  the intersection product. Let $\phi : X \to Y$ in $\SmProj(k)$, with $X$ and $Y$ are of pure dimensions $d$ and $e$, respectively. Let $\Gamma_\phi$ denote the image of the closed immersion $(\id_X, \phi) : X \to X\times Y$. We have the following correspondences: 
    \begin{equation*}        \phi_* = [\Gamma_\phi] \in \Corr^{e-d}(X, Y), \qquad 
        \phi^* := [^t \Gamma_\phi] \in \Corr^0(Y, X).
    \end{equation*}

\begin{definition}[\textbf{Chow motives}] Objects of the category of Chow motives $\CHM_\Qbb(k)$ are triples $(X, p, m)$, where $X \in \SmProj(k)$, $p \in \Corr^0(X, X)$ is an idempotent, i.e., $p \circ p = p$, and $m$ is an integer. If $(X, p, m)$ and $(Y, q, n)$ are Chow motives, then
$$\Hom_{\CHM_\Qbb(k)}((X, p, m), (Y, q, n)) = p \circ \Corr^{n-m}(X, Y) \circ q \subset \Corr^{n-m}(X, Y).$$
There is a contravariant functor
\begin{equation}
    h : \SmProj(k) \to \CHM_\Qbb(k), \quad X \mapsto h(
    X) := (X, [\Delta_X], 0),
\end{equation}
where $\Delta_X$ is the graph of the diagonal map. If $\phi : X \to Y$ is a morphism, 
$h(\phi) := \phi^* = [^t\Gamma_\phi] \in \Corr^0(Y, X) = \Hom_{\CHM_\Qbb(k)} (h(Y), h(X)).$
One calls $h(X)$ the Chow motive of $X$.
\end{definition}

Let $C$ be a smooth projective curve over $k$ (not necessarily having a $k$-rational point). Pick any zero cycle $Z$ on $C$ of positive degree $N$, one defines projectors $p_0(C) := \frac{1}{N}[Z \times C]$, $p_2(C) := \frac{1}{N}[C\times Z]$, and Chow motives $h^i(C) := (C, p_i(C), 0) \in \CHM_\Qbb(k)$ for $i = 0,2$. These motives do not depend on the choice of $Z$, in fact, $h^0(C) \simeq h(\Spec k')$ and $h^{2}(C) \simeq h(\Spec k') \otimes \mathbb{L},$ where $k' = \Gamma(C, \Ocal_C)$ and $\mathbb{L} = (\Spec k, \id, -1)$ is the Lefschetz motive. One sets $p_1(C) := \Delta_C -p_0(C)-p_2(C)$ and $h^1(C) := (C, p_1(C), 0) \in \CHM_\Qbb(k)$. We have the following direct sum decomposition
\begin{equation*}
    h(C) = h^0(C) \oplus h^1(C) \oplus h^2(C), 
\end{equation*}
which depends on the choice of $Z$, but $h^1(C)$ is well-defined up to unique isomorphism  (see \cite[Section 2.3]{MNP13} or \cite[Section 3]{Scho94}). If $C$ is further of genus 1, one can show that $h^1(C) \simeq h^1(E)$, where $E$ is the Jacobian of $C$, by using the following equivalence of categories (see the proof of \cite[Theorem 2.7.2(b)]{MNP13})
\begin{equation*}
    M''_\Qbb \xrightarrow{\simeq} \{\text{category of Jacobian of curves}\} \otimes \Qbb,
\end{equation*}
where $M''_\Qbb$ is the full subcategory of $\CHM_\Qbb(k)$ of motives isomorphic to $h^1(Y)$ for some smooth projective curve $Y$. Let us give explicitly the isomorphism.

\begin{prop}\label{4.17}
    Let $C$ be a smooth projective curve of genus 1 over a number field $k$ and $E$ be its Jacobian, then $h(C) \simeq h(E)$ and $h^1(C) \simeq h^1(E)$.
\end{prop}

\begin{proof}
    Let $\bar k$ be the algebraic closure of $k$ and let us fix a point $x_0 \in C(\bar k)$. We consider the morphism $\phi : C_{\bar k} \to E_{\bar k},$ which maps $x \in C(\bar k)$ to the divisor $N (x) - \sum_{\sigma} \left(\sigma(x_0)\right)$, where $\sigma$ runs through all the embeddings $k(x_0) \hookrightarrow \bar k$ and $N$ is the number of these embeddings. This map is well-defined as $N(x) - \sum_\sigma (\sigma(x_0))$ is a divisor of degree 0.  We have 
$\phi_* \in  \Hom_{\CHM_\Qbb(\bar k)}(h(C_{\bar k}), h(E_{\bar k}))$ and 
$\phi^* \in \Hom_{\CHM_\Qbb(\bar k)}(h(E_{\bar k}), h(C_{\bar k})).$ 
By \cite[Section  2.3]{MNP13}, we have $\phi_* \circ \phi^* = \deg(\phi)[\Delta_{E_{\bar k}}] = N^2 [\Delta_{E_{\bar k}}].$
Conversely, we have 
\begin{equation*}
    \phi^* \circ \phi_* \stackrel{\mathrm{def}}{=} \mathrm{pr}_{13*} ((\Gamma_\phi \times C_{\bar k}) \cdot (C_{\bar k} \times {^t}\Gamma_\phi)).
\end{equation*}
As sets, we observe that
\begin{equation*}
\begin{aligned}
    (\Gamma_\phi \times C_{\bar k}) \cap (C_{\bar k} \times {^t}\Gamma_\phi) &= \{(x, \phi(x), y) | x, y\in C(\bar k)\} \cap \{(z, \phi(t), t)| z, t\in C(\bar k)\}\\
    &= \{(x, \phi(x), y)|x, y\in C(\bar k), \phi(x) = \phi(y)\}\\
    &= \{(x, \phi(x), y)|x, y\in C(\bar k), N(x) - N(y) = 0 \text{ in } E(\bar k)\}\\
    &= \{(x, \phi(x), x+p)|x \in C(\bar k), p \in E_{\bar k}[N]\},
\end{aligned}
\end{equation*}
where $E_{\bar k}[N]$ is the set of $N$-torsion points of $E(\bar k)$ and ``+'' is the canonical action of $E_{\bar k}$ on $C_{\bar k}$.  So 
\begin{equation*}
   \phi^*\circ \phi_* = \sum_{p \in E_{\bar k}[N]} [\Gamma_{\varphi_p}] =  N^2 [\Delta_{C_{\bar k}}],
\end{equation*}
where $\varphi_p : C_{\bar k} \to C_{\bar k}, \ x \mapsto x + p$, and the last equality follows from the fact that $\Gamma_{\varphi_p}$ is rationally equivalent to $\Delta_{C_{\bar k}}$ for $p \in E_{\bar k}[N]$.   
We thus obtain that $\phi_* : h(C_{\bar k}) \to h(E_{\bar k})$ is an isomorphism in the category $\CHM_\Qbb(\bar k)$. For $\alpha \in \Gal(\bar k/k)$ and $x \in C(\bar k)$,
\begin{equation*}
    \begin{aligned}
        (\alpha\circ \phi) (x) =  \alpha(N(x)) - \sum_{\sigma} (\alpha\circ\sigma(x_0)) = N(\alpha(x)) - \sum_\sigma (\sigma(x_0)) = (\phi \circ \alpha)(x).
    \end{aligned}
\end{equation*}
This implies that $\Gamma_\phi$ and $^t\Gamma_\phi$ are $\Gal(\bar k/k)$-invariant. Hence by Galois descent (see e.g. \cite[Theorem 1.3(6)]{DS91}) 
\begin{equation*}
    \CH^1(C_{\bar k} \times_{\bar k} J_{\bar k})^{\Gal(\bar k/k)} \simeq \CH^1(C \times_k E), \ \CH^1(E_{\bar k} \times_{\bar k} C_{\bar k})^{\Gal(\bar k/k)} \simeq \CH^1(E \times_k C),
\end{equation*}
$\phi_*$ defines an isomorphism from $h(C)$ to $h(E)$ in the category $\CHM_k$ with inverse $\phi^*$.

Denote by $A$ the positive zero-cycle of degree $N$ corresponding to $x_0$. We set $p_0(C) := \frac{1}{N}[A\times C]$, $p_2 (C) := \frac{1}{N}[C \times A]$, and $p_1(C) := \Delta_C - p_0(C) - p_2(C)$. 
Let $\Ocal$ be the trivial element in $E(k)$, we set $p_0(E) := \Ocal \times E$, $p_2 (E) = E \times \Ocal$, and $p_1(E) := \Delta_E - p_0(E) - p_2(E)$. By explicit computations, we have
\begin{equation*}
    \phi^*\circ p_0(E) \circ \phi_* = N^2 p_0(C) \quad \text{and} \quad \phi^* \circ p_2(E) \circ \phi_* = N^2 p_2(C).
\end{equation*}
Now we show that $p_1(E)\circ \phi_* \circ p_1(C)$ and $p_1(C) \circ \phi^* \circ p_1(E)$ define isomorphisms from $h^1(C)$ to $h^1(E)$ and inverse, respectively. We have
\begin{align*}
    \phi^* \circ p_1(E) \circ \phi_*
    = \phi^* \circ (\phi_* - p_0(E)\circ \phi_* - p_2(E) \circ \phi_*)
    &= N^2 [\Delta_C]  - \phi^*\circ p_0(E) \circ \phi_* - \phi^* \circ p_2(E) \circ \phi_*\\
    &= N^2[\Delta_C] - N^2 p_0(C) - N^2 p_2(C)\\
    &= N^2 p_1(C).
\end{align*}  
We thus have $ p_1(C) \circ \phi^* \circ p_1(E) \circ p_1(E)\circ \phi_* \circ p_1(C) = p_1(C) \circ \phi^* \circ p_1(E) \circ \phi_* \circ p_1(C) = N^2 p_1(C).$ Similarly, we have $ p_1(E)\circ \phi_* \circ p_1(C) \circ p_1(C) \circ \phi^* \circ p_1(E) = N^2 p_1(E).$
\end{proof}

\subsection{Motivic cohomology}\label{motiviccohomology} Let $k$ be an arbitrary field of characteristic 0. Let us recall briefly the definition and some facts of motivic cohomology. For more details, we refer to \cite[Chapter 5, Section 2]{VSF00}. Voevodsky constructed a triangulated category, called \textit{geometrical motives}, denoted by $\mathrm{DM}_\mathrm{gm}(k)$ and a covariant functor $$M_{\mathrm{gm}} : \mathrm{Sm}(k) \to \mathrm{DM}_\mathrm{gm}(k)$$
(see \cite[page 192]{VSF00}). The motivic cohomology of $X$ with coefficients in $\Qbb$ is defined by 
\begin{equation}
    H^n_{\Mcal}(X, \Qbb(j)) := \Hom_{\mathrm{Dm}_{\mathrm{gm}}(k)} (M_{\mathrm{gm}}(X), \Qbb(j)[n]), \quad  \text{ for } n, j \in \Zbb,
\end{equation}
where $\Qbb(1) \in \mathrm{DM}_\mathrm{gm}(k)$ is the \textit{Tate motive} (see also \cite[page 192]{VSF00}) and $\Qbb(j) = \Qbb(1)^{\otimes j}$. It is known that the motivic cohomology $H^n_\Mcal(X, \Qbb(j))$ is isomorphic to the $j$-eigenspace of Quillen's $K$-group $K_{2j-n}(X)_\Qbb$ with respect to Adams operation (see \cite[Chapter II.4]{Wei13} for the definition), namely, 
\begin{equation}
    H^n_\Mcal(X, \Qbb(j)) \simeq K^{(j)}_{2j-n}(X)
\end{equation}
(see \cite[page 197]{VSF00}). By functorial property of $M_\mathrm{gm}$, for any morphism $f: X \to Y$, we have a $\Qbb$-linear map $f^* : H^n_\Mcal(Y, \Qbb(j)) \to H^n_\Mcal(X, \Qbb(j))$, called pull back of $f$. Moreover, for proper maps $f : X \to Y $ of pure codimension $c = \dim Y - \dim X$, we have a $\Qbb$-linear map, called push-forward of $f$ (see \cite[Theorem 1.3]{DS91})
$$f_* : H^n_\Mcal(X, \Qbb(j)) \to H^{n+2c}_\Mcal(Y, \Qbb(j + c)).$$

Let $X$ and $X'$ be smooth quasi-projective varieties over $k$ and $\pi : X' \to X$ be a finite Galois covering with group $G$. We have Galois descent for motivic cohomology, i.e., 
\begin{equation}\label{Galoisdescent}
    \pi^* : H^n_\Mcal(X, \Qbb(j)) \xrightarrow{\simeq
    } H^n_\Mcal(X', \Qbb(j))^G
\end{equation}
is an isomorphism (see \cite[Theorem 1.3]{DS91}).

Let $X \in \mathrm{Sm}(k)$ and $\iota : Z \hookrightarrow X$ be a closed immersion of smooth varieties of codimension $c$ with open complement $\jmath : X - Z \hookrightarrow X$. We have a \textit{localization sequence} for motivic cohomology (see \cite[page 196]{VSF00} or \cite[Theorem 1.3]{DS91})
\begin{equation}\label{loclization1}
    \cdots \to H^{i-2c}_\Mcal(Z, \Qbb(j-c)) \xrightarrow{\iota_*} H^i_\Mcal(X,  \Qbb(j)) \xrightarrow{\jmath^*} H^i(X - Z, \Qbb(j))
     \to H^{i+1-2c}_\Mcal(Z, \Qbb(j-c)) \to \cdots.
\end{equation}

Let $C$ be a smooth curve over a number field $k$. Denote by $F = k(C)$ the function field of $C$. Then $$H^n_\Mcal(F, \Qbb(j)) = \varinjlim_{U \subset C \text{ open}} H^n_\Mcal(U, \Qbb(j)).$$ We have the following version of localization sequence
\begin{equation}\label{motiviclocalization}
    0 \to H^2_\Mcal(C, \Qbb(3)) \to H^2_\Mcal(F, \Qbb(3)) \xrightarrow{\Res^\Mcal} \bigoplus_{x\in C^1} H^1_\Mcal(k(x), \Qbb(2)),
\end{equation}
where $C^1$ is the set of closed points of $C$. This follows from the localization sequence of Quillen's $K$-groups (see \cite[V.6.12]{Wei13}). The left exactness follows from Borel's theorem (see e.g. \cite[IV.1.18]{Wei13}) which states that $K_4$ of a number field is torsion, hence $H^0_\Mcal(K, \Qbb(2)) \simeq K_4^{(2)}(K) = 0$.

\subsection{The Beilinson regulator map}\label{Beireg}
Let $X$ be a smooth variety over $\Rbb$ or $\Cbb$. The Beilinson regulator map, as defined in \cite{Nek13}, is a $\Qbb$-linear map
\begin{equation}
    \reg_X : H^n_\Mcal (X, \Qbb(j)) \to H^n_\Dcal(X, \Rbb(j)).
\end{equation}
For $n = j = 1$, we have $H^1_\Mcal(X, \Qbb(1)) = \Ocal(X)^\times_\Qbb$ and the regulator map sends an invertible function $f$ to the class of $\log |f|$ (cf. \cite[Appendix A.3]{BZ20}). As the regulator map is compatible with taking cup products, we observe that the regulator map sends the Milnor symbol $\{f_1, \dots, f_n\} \in H^n_\Mcal(X, \Qbb(n))$ to the class of $\log |f_1| \cup \dots \cup \log |f_n|$ in $H^n_\Dcal(X, \Rbb(n))$. When $X$ is defined over a number field $k$, we write $X_\Rbb := X \times_\Qbb \Rbb$, and the Beilinson regulator map  is defined as the composition 
\begin{equation}\label{135}
     H^n_\Mcal(X, \Qbb(j)) \xrightarrow{\text{base change}} H^n_\Mcal(X_\Rbb, \Qbb(j)) \xrightarrow{\reg_{X_\Rbb}} H^n_\Dcal (X_\Rbb, \Rbb(j)).
\end{equation}

Now let $C$ be a smooth curve over a number field $k$. Let $F = k(C)$ be the function field of $C$.  We define Deligne cohomology of $F$ by the direct limit 
\begin{equation}
    H^n_\Dcal(F, \Rbb(j)) := \varinjlim_{U \subset C\text{ open}} H^n_\Dcal (U_\Rbb, \Rbb(j)). 
\end{equation}
And the regulator map for the function field is defined by $\reg_F := \varinjlim_{U\subset C\text{ open}} \reg_U$
\begin{equation}\label{regF}
    \reg_F : H^n_\Mcal(F, \Qbb(j)) \to H^n_\Dcal(F, \Qbb(j)).
\end{equation}

Now let us recall the \textit{regulator integral} for smooth projective curve (see \cite[Section 3]{dJ00} for more details). Let $C$ be a smooth projective curve over a number field $k$. Denote by $C(\Cbb)$ the set of complex points of $C \times_\Qbb \Cbb$. If $\omega$ is a holomorphic 1-form on $C(\Cbb)$ such that $F_\dR(\omega) = \omega$, where $F_\dR$ is defined in Section \ref{DeligneCoho}, then we have a map
\begin{equation}
    H^1(C(\Cbb), \Rbb(2))^+ \to \Rbb(1), \qquad \eta \mapsto \int_{C(\Cbb)} \eta \wedge \bar \omega
\end{equation}
(see \cite[Remark 3.1]{dJ00}). This integral depends on the choice of the orientation of $C(\Cbb)$. Recall that there is a canonical isomorphism $H^2_\Dcal(C_\Rbb, \Rbb(3)) \simeq H^1(C(\Cbb), \Rbb(2))^+$ (see Remark \ref{1.5}). We thus write the Beilinson regulator  map on $C$ as the composition 
\begin{equation}\label{regC}
   \reg_C : H^2_{\Mcal}(C, \Qbb(3)) \xrightarrow{} H^2_\Dcal (C_\Rbb, \Rbb(3)) \simeq H^1(C(\Cbb), \Rbb(2))^+.
\end{equation}
The composition map 
\begin{equation}\label{regint}
    H^2_{\Mcal}(C, \Qbb(3)) \xrightarrow{\reg_C} H^1(C(\Cbb), \Rbb(2))^+ \xrightarrow{\eta \mapsto \int_{C(\Cbb)} \eta \wedge \bar \omega} \Rbb(1)
\end{equation}
is called the \textit{regulator integral}. Similarly, we have the regulator integral for the function field
\begin{equation}\label{regintF}
     H^2_{\Mcal}(F, \Qbb(3)) \xrightarrow{\reg_F} H^1(F, \Rbb(2))^+ \xrightarrow{\eta \mapsto \int_{C(\Cbb)} \eta  \wedge \bar \omega} \Rbb(1),
\end{equation}
where $H^1(F, \Rbb(2))^+ := \varinjlim_{U\subset C \text{ open}} H^1(U(\Cbb), \Rbb(2))^+$.

\subsection{Beilinson's conjecture for genus 1 curves}\label{Beilinson} In this section, we recall Beilinson's conjecture for smooth projective curves of genus 1 (see \cite[Section\ 6]{Nek13}, \cite[Section\ 4]{dJ96} for more details). 
Let us recall the definition of $L$-function attached to the pure motive $h^i(X)$, for $X$ is a smooth projective variety over $\Qbb$ of dimension $n$.

\begin{definition}[{\cite[Section 1.4]{Nek13}}]
Let $p$ be a prime number. For $0\le i \le 2n$, we set
\begin{equation*}
    L_p(h^i(X), s) = \det (1 - \mathrm{Frob}_p p^{-s} | h^i_\ell(X)^{I_p})^{-1},
\end{equation*}
where $\ell \neq p$ is a prime number,  $\mathrm{Frob}_p\in \Gal(\bar \Qbb/\Qbb)$ is a Frobenius element at $p$, acting on the étale realization 
\begin{equation*}
   h^i_\ell(X) :=  H^i_{\et}(X_{\bar \Qbb}, \Qbb_\ell),
\end{equation*}
and $I_p$ is the inertia group at $p$.
\end{definition}

\begin{rmk}
If $X$ has good reduction at $p$, then $L_p(h^i(X), s)$ does not depend on the choice of $\ell$ (\cite[Section 1.4]{Nek13}). And it is conjectured by Serre that if $X$ has bad reduction at $p$, then $L_p(h^i(X),s)$ is independent of the choice of $\ell$ and has integer coefficients (cf. \cite[Conjecture 5.45]{Kah20}). This conjecture holds if $i \in \{0, 1, 2n-1, 2n\}$ (cf. 
    \cite[Theorem 5.46]{Kah20}). In particular, it holds when $X$ is a curve.
\end{rmk}

\begin{definition}[\textbf{$L$-function}](\cite[Section 1.5]{Nek13}) The $L$-function associated to $h^i(X)$ is defined by
\begin{equation*}
    L(h^i(X), s) = \prod_{p \ \mathrm{prime}} L_p(h^i(X), s).
\end{equation*}
\end{definition}

Let $C/\Qbb$ be a smooth projective curve of genus 1 and $E$ be its Jacobian. By Proposition \ref{4.17}, we have $L(h^1(C), s) = L(h^1(E), s)$. Hence $L(h^1(C), s) = L(E, s)$ the Hasse-Weil $L$-function. We then have the following version of Beilinson's conjecture for smooth projective curve of genus 1.
\begin{conj}[{\cite[Section 6]{Nek13}, \cite[Section 4]{dJ96}}]\label{BeiConj}
Let $C$ be a smooth projective curve over $\Qbb$ of genus 1 and $E$ be its Jacobian. For any nontrivial element $\alpha \in H^2_\Mcal (C, \Qbb (3))$, we have
\begin{equation*}
    \dfrac{1}{(2 \pi i)^2} \left<\gamma_C^+, \reg_C(\alpha)\right>_{C(\Cbb)} = a \cdot L'(E, -1), \qquad (a \in \Qbb^\times),
\end{equation*}
where $\gamma_C^+$ is a generator of $H_1(C(\Cbb), \Qbb)^+$, $\reg_C$ is the Beilinson regulator map \eqref{regC}, and $\left< \cdot, \cdot \right>$ is then the pairing in de Rham cohomology.
\end{conj}

\section{Goncharov's polylogarithmic complexes}\label{Goncharov}
In 1990s, Goncharov introduced polylogarithmic complexes and regulator maps from the cohomology of these complexes to Deligne cohomology. They have connections with motivic cohomology and the Beilinson regulator map (cf. \cite{Gon95, Gon96, Gon98}). In this section, we recall briefly these constructions of Goncharov, which will be used in Section \ref{motivic} to construct elements in motivic cohomology.

\subsection{Goncharov's complexes}\label{Goncomplex}

For any field $F$ of characteristic 0 and an integer  $n \ge 1$, Goncharov defined $\Bscr_n(F)$ to be the quotient of the $\Qbb$-vector space $\Qbb[\Pbb_F^1]$ by a certain (inductively defined) subspace $\Rscr_n(F)$. For $x\in F \cup \{\infty\}$, we denote by $\{x\}_n$ the class of $\{x\} \in \Qbb[\Pbb^1_F]$ in $\Bscr_n(F)$. 
Goncharov then constructed a \textit{weight $n$ polylogarithmic complex}, in degree 1 to $n$ (see \cite[Section 1]{Gon95})
\begin{equation*}
\Gamma_F(n) : \Bscr_n(F) \to \Bscr_{n-1}(F) \otimes F^\times_\Qbb \to \Bscr_{n-2}(F) \otimes \bigwedge^2 F^\times_\Qbb \to \cdots \to \Bscr_2(F) \otimes \bigwedge^{n-2} F^\times_\Qbb \to \bigwedge^n F^\times_\Qbb,
\end{equation*}
where the maps are given by 
\begin{equation*}
    \{x\}_{n-p} \otimes y_1 \wedge \dots \wedge y_{p} \mapsto \{x\}_{n-p-1} \otimes x \wedge y_1 \wedge \dots \wedge y_{p} \qquad \text{if } 0 \le p < n-2,
\end{equation*}
and 
\begin{equation*}
    \{x\}_2 \otimes y_1 \wedge \dots \wedge y_{n-2} \mapsto (1-x)\wedge x \wedge y_1 \wedge \dots \wedge y_{n-2}.
\end{equation*}
It is conjectured that the cohomology of this complex computes the motivic cohomology. 
\begin{conj}[{\cite[Conjecture A, p. 222]{Gon95}}]
$H^p(\Gamma_F(n)) \simeq H^p_\Mcal(F, \Qbb(n))$ for $p, n \ge 1$.
\end{conj}
Besides, Goncharov also defined $\Qbb$-vector space $B_n(F)$ ($1\le n \le 3$) to be the quotient of $\Qbb[\Pbb^1_F]$ by a certain subspace $R_n(F)$, which are generated by explicit relations as follows. 
\begin{equation}\label{130}
    \begin{aligned}
     R_1(F) &:= \left<\{x\} + \{y\} - \{xy\}, x, y \in F^\times; \{0\}; \{\infty\}\right>,\\
     R_2(F) &:= \left<\{x\} + \{y\} + \{1-xy\} + \left\{\frac{1-x}{1-xy}\right\} + \left\{\frac{1-y}{1-xy}\right\}, x, y \in F^\times\setminus \{1\}; \{0\}; \{\infty\}\right>,
\end{aligned}
\end{equation}
and $R_3(F)$ is generated by explicit relations corresponding to the functional equations for the trilogarithm (see \cite[p. 214]{Gon95}) that we do not mention here. We still denote by $\{x\}_k$ the class of $\{x\}\in \Qbb[\Pbb^1_F]$ in $B_k(F)$.  As Goncharov's constructions, $B_1(F) = \Bscr_1(F) = F^\times_\Qbb$ (see \cite[Section 1.8-1.9]{Gon95}). And it is proved by De Jeu that $B_2(F) = \Bscr_2(F)$ (see \cite[Remark 5.3]{dJ00}). Goncharov showed that $R_3(F) \subset \Rscr_3(F)$ and conjectured that they are equal (see \cite[p. 225]{Gon95}).

Goncharov then constructed the polylogarithmic complexes $\Gamma(F, n)$ for $n=2, 3$ with the same shape as $\Gamma_F(n)$ but $\Bscr_n(F)$ are replaced by $B_n(F)$ (notice that these complexes are denoted by $B_F(n)$ in \cite[Section 1.8]{Gon95}). In this article, we only use the vector spaces with explicit relations $B_n(F)$ and the corresponding polylogarithmic complexes $\Gamma(F, n)$ for $n = 2, 3$. For $n =2$, it is given as follows, in degree 1 and 2:
\begin{equation*}
    \Gamma(F, 2) :\qquad  B_2(F) \xrightarrow{\alpha_2(1)} \bigwedge^2 F^\times_\Qbb,\quad \{x\}_2 \mapsto (1-x) \wedge x.
\end{equation*}
We have $H^2 (\Gamma(F, 2)) \simeq H^2_\Mcal(F, \Qbb(2))$ by Matsumoto's theorem. And $H^1(\Gamma(F, 2)) \simeq H^1_\Mcal(F, \Qbb(2))$ by Suslin's work \cite{Sus91}. The group $H^1(\Gamma(F, 2))$ is also called \textit{Bloch group}, denoted by $\Bcal(F)$ (see \cite[Section\ 2]{Zag91}).

For $n=3$, we have the following polylogarithmic complex in degree 1 to 3:
\begin{equation}\label{Gamma(F,3)}
\begin{tikzcd}[cramped,row sep=small]
    \Gamma(F, 3):& B_3(F) \arrow[r, "\alpha_3(1)"] & B_2(F)\otimes F^\times_\mathbb{Q} \arrow[r, "\alpha_3(2)"] &\bigwedge^3 F^\times_\mathbb{Q},\\
    &\{x\}_3 \arrow[r, mapsto, shorten=2mm]& \{x\}_2 \otimes x &\\
    & & \{x\}_2 \otimes y \arrow[r, mapsto, shorten=1mm] & (1-x)\wedge x \wedge y.
\end{tikzcd}
\end{equation}
We have $H^3 (\Gamma(F, 3)) \simeq H^3_\Mcal(F, \Qbb(3)).$ In degree 2, Goncharov constructed a map $K_4(F)_\Qbb \to H^2(\Gamma(F, 3))$ and  conjectured that this induces an isomorphism $H^2_\Mcal(F, \Qbb(3)) \simeq K_4^{(3)}(F) \to H^2(\Gamma(F, 3))$. De Jeu constructed a map in other direction $H^2(\Gamma(F, 3)) \to H^2_\Mcal(F, \Qbb(3))$. We discuss about the later map in Section \ref{Relating to Gon}.

\subsection{The residue homomorphism of complexes} \label{Gonres}
Let $K$ be a field with a discrete valuation $v$. Denote by $\Ocal_K$, $k_v$, $\pi_v$ the ring of integers, the residue field, and a uniformizer, respectively. Goncharov defined a residue homomorphism on his polylogarithmic complexes (see \cite[Section 1.14]{Gon95})
\begin{equation}
    \partial_v : \Gamma(K, 3) \to \Gamma(k_v, 2)[-1].
\end{equation}
More precisely, it is given by 
\begin{equation}\label{ee}
    \xymatrix{B_3(K) \ar[r]^{\alpha_3(1)\quad} & B_2(K) \otimes K_\Qbb^\times \ar[r]^{\quad\alpha_3(2)} \ar[d]^{\partial_v^{(2)}} & \bigwedge^3 K_\Qbb^\times \ar[d]^{\partial_v^{(3)}}\\
    & B_2(k_v) \ar[r]_{\alpha_2(1)} &  \bigwedge^2 (k_v^\times)_{\Qbb},}
\end{equation}
where the vertical maps are defined as follows. For $f \in K^\times$, we denote by $f_v$ the image of $f \pi_v^{-\mathrm{ord}_v(f)}$ under the canonical map $\Ocal_K^\times \to k_v^\times$. We have
\begin{align}\label{Gonresmap}
    &\partial_v^{(2)} : \{f\}_2 \otimes g \mapsto \mathrm{ord}_v(g) \{f(v)\}_2, \text{ with the convention }\{0\}_2 = \{1\}_2 = \{\infty\}_2 = 0 \text{ in } B_2(k_v), \\
    &\partial_v^{(3)} : f \wedge g \wedge h \mapsto \mathrm{ord}_v(f) g_v \wedge h_v - \mathrm{ord}_v(g) f_v \wedge h_v + \mathrm{ord}_v(h) f_v \wedge g_v.
\end{align}

Now let $C$ be a smooth connected curve over a number field $k$ and let $F$ be its function field. Denote by $C^1$ the set of closed points of $C$, and $k(x)$ the residue field of $x \in C^1$. We have the following morphism of complexes
\begin{equation}\label{eee}
\partial := \oplus_{x\in C^1} \partial_x : \Gamma(F, 3) \to \bigoplus_{x\in C^1} \Gamma(k(x), 2)[-1].
\end{equation}
Goncharov defined the polylogarithmic complex $\Gamma(C, 3)$ as the mapping cone of \eqref{eee}. We then have the following exact sequence
\begin{equation}\label{cc}
    0 \to H^2(\Gamma(C, 3)) \to H^2(\Gamma (F, 3)) \xrightarrow{\partial} \bigoplus_{x \in C^1}H^{1}(\Gamma(k(x), 2)),
\end{equation}
where the left exactness is due to the fact that the cohomology group of degree 0 of polylogarithmic complexes vanishes. This should be isomorphic to the localization sequence of motivic cohomology \eqref{motiviclocalization}. We will construct a morphism from complexes \eqref{cc} to \eqref{motiviclocalization} in \eqref{dJdiagram}, using work of De Jeu.

\subsection{Goncharov's regulator maps}\label{Gonregulator} In this section, we recall Goncharov's regulator maps. Let $X$ be a regular variety over a number field. Denote by $F$ the function field of $X$. For a nonempty open subscheme $U \subset X$, $U(\Cbb)$ denotes the set of complex points of $U \times_\Qbb \Cbb$. Let $\Omega^j (\eta_X) := \varinjlim_{U \subset X \text{ open}} \Omega^j(U)$ and $\Omega ^j (U)$ is the space of real smooth $j$-forms on $U(\Cbb)$. Goncharov gave explicitly a homomorphism of complexes (see \cite[Theorem 2.2]{Gon98}):
\begin{equation}
\xymatrix{B_3(F) \ar[rr]^{\alpha_3(1)\quad \quad \ } \ar[d]^{r_3(1)} && B_{2} (F) \otimes F^\times_\Qbb \ar[rr]^{\quad \quad \quad  \alpha_3(2)} \ar[d]^{r_3(2)} && \bigwedge^3 F^\times_\Qbb \ar[d]^{r_3(3)}\\
\Omega^0(\eta_X) \ar[rr]^{d} && \Omega^1(\eta_X) \ar[rr]^{d} && \Omega^{2}(\eta_X).}
\end{equation}
For $f\in F^\times\setminus \{1\}$, $g, h \in F^\times$, the vertical maps in degrees 2 and 3 are given respectively by
$$r_3(2) : \{f\}_2 \otimes g \mapsto \rho(f, g), \qquad r_3(3) : f \wedge g \wedge h \mapsto -\eta(f, g, h),$$ where 
\begin{equation}\label{kk}
\begin{aligned}
    \eta(f, g, h) &:= \log|f| \left(\dfrac{1}{3} d\log|g|\wedge d\log|h| - d\arg(g) \wedge d\arg(h)\right)\\
    &\quad+ \log|g| \left(\dfrac{1}{3} d\log|h|\wedge d\log|f| - d\arg(h) \wedge d\arg(f)\right)\\
    &\quad+ \log|h| \left(\dfrac{1}{3} d\log|f|\wedge d\log|g| - d\arg(f) \wedge d\arg(g)\right),
    \end{aligned}
\end{equation} 
and
\begin{equation}\label{bb}
     \rho(f, g) := - D(f)\  d \arg g + \dfrac{1}{3} \log |g| \theta(1-f, f),
\end{equation}
where
\begin{equation}
    \theta(h, f) = \log|h| d \log |f| - \log |f| d \log |h|,
\end{equation}
and $D : \Pbb^1(\Cbb) \to \Rbb$ is the Bloch-Wigner dilogarithm function \eqref{dilog}.  In particular, we have $$d \rho(f, g) = -\eta(1-f, f, g) = \eta(f, 1-f, g) \text{ for } f \in F^\times\setminus \{1\}, g\in F^\times.$$ 

Let $C$ be a smooth connected curve over a number field and let $F$ be its function field. Goncharov showed that the map $r_3(2)$ gives rise to a regulator map  on the cohomology of his complexes of the function field
\begin{equation}\label{r32F}
    r_3(2)_F : H^2(\Gamma(F, 3)) \to H^2_\Dcal(F, \Rbb(3)) \simeq H^1(F, \Rbb(2))^+
\end{equation}
(see \cite[Section\ 2.7]{Gon98}). Moreover, the map $r_3(2)_F$ is compatible with taking residues \eqref{eee} (see \cite[Theorem 2.6]{Gon98}), hence it extends to a homomorphism
\begin{equation}\label{Gonmap}
    r_3(2)_C : H^2(\Gamma(C, 3)) \to H^2_\Dcal (C\otimes_\Qbb \Rbb, \Rbb(3)) \simeq H^1(C(\Cbb), \Rbb(2))^+.
\end{equation}

Now let us compute the residues of the differential form representing $r_{3}(2)_F(\alpha)$ for $\alpha \in H^2(\Gamma(F, 3))$. First let us recall the definition of the residues of a differential form.
\begin{definition}\cite[Definition 7.3]{BZ20}\label{resmap}
    Let $C$ be a smooth connected curve over a number field and let $Z$ be a subset of closed points of $C$. We set $Y := C \setminus Z$. Let $\eta \in H^1(Y(\Cbb), \Rbb)$. The residue of $\eta$ at $p \in C(\Cbb)$ is
\begin{equation}\label{residue}
\mathrm{Res}_p (\eta) = \int_{\gamma_p} \eta,
\end{equation}
where $\gamma_p$ is the boundary of any small disc containing $p$ and avoiding $Z(\Cbb) \setminus \{p\}$. 
\end{definition}

\begin{lm}\label{4.3}
   Let $\alpha = \sum_j c_j \{f_j\}_2 \otimes g_j \in H^2(\Gamma(F, 3))$ with $c_j \in \Qbb$, and $f_j  \in F^\times \setminus \{1\}, g_j \in F^\times$. Denote by $Z$ the closed subset of $C$ consisting of zeros and poles of $f_j, 1-f_j, g_j$ for all $j$. Set $Y = C \setminus Z$. Then $r_3(2)_F(\alpha)$ is represented by the differential form $ \rho := \sum_j c_j \rho(f_j, g_j) \in H^1(Y(\Cbb), \Rbb(2))^+$ and its residue at $p \in C(\Cbb)$ is given by
    \begin{equation*}
        \Res_p(\rho) = -2 \pi \sum_j c_j v_p(g_j) D(f_j(p)),
    \end{equation*}
    where $D$ is the Bloch-Wigner dilogarithm function \eqref{dilog}.
\end{lm}

\begin{proof}
The first statement follows directly from the definition of the map $r_{3}(2)_F$ \eqref{r32F}. Now we compute the residues. Let $f, g \in \Cbb(C)^\times$ such that  all the zeros and poles of $f, 1-f, g$ are contained in $Z$. Let $p \in C(\Cbb)$ and  $\gamma_p$ be a sufficiently small loop  around $p$ and does not surround any point of $Z \setminus \{p\}$.  Using the local coordinate $z = r e^{is}$, for $r >0$ small and $s \in [0, 2 \pi]$, we have
$f(z) = (r e^{is})^{v_p(f)} F(r e^{i s}) \text{ and } g(z) = (r e^{i s})^{v_p(g)} G(r e^{is}),$
where $F$ and $G$ are holomorphic functions such that $F(0), G(0) \neq 0$. Then
\begin{equation}\label{14}
    \begin{aligned}
    \int_{\gamma_p} D(f) d \arg (g) &= \int_0^{2 \pi} D(f(r e^{is})) \ d \arg \left((r e^{is})^{v_p(g)} G(r e^{is})\right)\\
    &= \int_0^{2\pi} D(f(r e^{is})) v_p(g) ds + \int_0^{2\pi} D(f(r e^{is})) \ d \arg G(r e^{is}).
    \end{aligned}
\end{equation}
As
\begin{equation*}
\begin{aligned}
    d \arg G(z) &= \dfrac{1}{2 i} \left(\dfrac{d G}{G} - \dfrac{d \overline{G}}{\overline{G}}\right)
    &= \dfrac{1}{2 i} \left(\dfrac{1}{G} \dfrac{\partial G}{\partial z} d z - \dfrac{1}{\overline{G}} \dfrac{\partial \overline{G}}{\partial \bar z} d \bar z \right),
\end{aligned}
\end{equation*}
we have
\begin{equation*}
d \arg G(r e^{is}) = \dfrac{1}{2 i} \left(\dfrac{G_z}{G} r i e^{is} ds - \dfrac{\overline{G}_{\bar z}}{\overline{G}} r (-i) e^{-is} ds\right) = O(r) ds,
\end{equation*}
where $G_z$ is the derivative of $G$ with respect to $z$. Then by taking $r\to 0$ in  (\ref{14}), the limit of $\int_{\gamma_p} D(f) d \arg(g)$ as $\gamma_p$ shrinks to $p$ is
\begin{equation}\label{15}
    \int_0^{2\pi} D(f(p)) v_p(g) ds = 2\pi v_p(g) D(f(p)).
\end{equation}
Moreover, we have
\begin{equation*}
     \log|f| = \log|F(r e^{is})| + v_p(f) \log r, 
\end{equation*}
and 
\begin{equation*}
    d \log|f| = d \log |F| = \dfrac{1}{2}\left(\dfrac{dF}{F}+\dfrac{d\overline{F}}{\overline{F}}\right)  = O(r) ds.
\end{equation*}
Therefore, $\theta(1-f, f) = O(r \log r) \ ds \ \to 0 \text{ as } r \to 0.$
Hence $$\Res_p(\rho) =  -2\pi \sum_j c_j v_p(g_j) D(f_j(p)), \text{ for } p \in C(\Cbb).$$
\end{proof}

\section{De Jeu's polylogarithmic complexes}\label{DeJeu}

In this section, we recall De Jeu's polylogarithmic complexes and his maps from the cohomology of these complexes to the motivic cohomology. In particular, they give rise to maps from the cohomology of Goncharov's polylogarithmic complexes to the motivic cohomology. We then compare the images of Goncharov's regulator and Beilinson's regulator composed with these maps. These results are used in the construction of the motivic cohomology classes in Section  \ref{motivic}.  
In this article, we consider only the cases of the polylogarithmic complexes of weight 2 and weight 3. The references for this section are \cite{dJ95, dJ96, dJ00}.

\subsection{De Jeu's polylogarithmic complexes}\label{DeJeucomplex} Let $F$ be a field of characteristic 0. De Jeu  defined $\widetilde{M}_{(j)}(F)$ to  be a certain $\Qbb$-vector space generated by symbols $[f]_j$ with $f \in F^\times \setminus \{1\}$ and constructed the following complex in degree 1 to 2
\begin{equation}
    \widetilde \Mcal_{(2)}^\bullet (F):\qquad 
    \widetilde M_{(2)} (F) \to \bigwedge^2F^\times_\Qbb, \quad [f]_2 \mapsto (1-f)\wedge f,
\end{equation}
and this complex in degree 1 to 3
\begin{equation}
\begin{tikzcd}[cramped,row sep=small]
    \widetilde{\Mcal}_{(3)}^\bullet(F):& \widetilde{M}_{(3)}(F) \arrow[r] & \widetilde{M}_{(2)}(F) \otimes F_{\Qbb}^\times \arrow[r] &\bigwedge^3 F^\times_\mathbb{Q},\\
    &{[f]}_3 \arrow[r, mapsto, shorten=2mm] & {[f]}_2 \otimes f &\\
    & & {[f]}_2 \otimes g \arrow[r, mapsto, shorten=1mm] & (1-f)\wedge f \wedge g
\end{tikzcd}
\end{equation}
(see \cite[Corollary 3.22, Example 3.24]{dJ95} or \cite[Section 2]{dJ00}). We have $H^n(\widetilde\Mcal_{(n)}^\bullet (F)) \simeq H^n_\Mcal(F, \Qbb(n))$ for $n \in \{2,3\}$. Let $k$ be a number field. By Suslin's work, we have the following isomorphism (up to a universal choice of sign) (see \cite[Theorem 5.3]{dJ95} or \cite[Theorem 2.3]{dJ00})
\begin{equation}\label{phin1}
    \varphi_{(2)}^1 : H^1(\widetilde \Mcal_{(2)}^\bullet(k)) \xrightarrow{\simeq} H^1_\Mcal(k, \Qbb(2)).
\end{equation} 
Let $\sigma : k \hookrightarrow \Cbb$ be any embedding of $k$ into $\Cbb$. We consider the following composition map \begin{equation}\label{bjg}
    H^1(\widetilde{\Mcal}_{(2)}^\bullet(k)) \xrightarrow{\varphi_{(2)}^1} H^1_\Mcal(k, \Qbb(2)) \xrightarrow{\reg_{k}} \Rbb(1),
\end{equation}
where $\reg_k$ is the composition  $H^1_\Mcal(k, \Qbb(2)) \xrightarrow{\sigma^*} H^1_\Mcal(\Cbb, \Qbb(2)) \xrightarrow{\reg_\Cbb} H^1_\Dcal(\Cbb, \Rbb(2)) \simeq \Rbb(1)$; the last isomorphism here is the canonical isomorphism mentioned in Remark \ref{1.5}. It is shown that the map 
\eqref{bjg} is given by $[z]_2$ to $\pm i D(\sigma(z))$,  where $D$ is the Bloch-Wigner dilogarithm (see \cite[Proposition 4.1]{dJ95}). So we can fix the sign of $\varphi_{(2)}^1$ such that it is induced by  $[z]_2 \mapsto i D(\sigma(z))$.

Moreover, De Jeu (\cite[p. 529]{dJ96}) constructed a map (up to a universal choice of sign)
\begin{equation}\label{phi32}
\varphi_{(3)}^2 : H^2(\widetilde{\Mcal}^\bullet_{(3)}(F)) \to H^2_\Mcal(F, \Qbb(3)).
\end{equation}
We discuss more about this map when $F$ is the function field of a curve in the following section.

\subsection{De Jeu's maps}\label{dJmapsection} Let $C$ be a smooth geometrically connected curve over a number field $k$. Denote by $F$ the function field of $C$ and $k(x)$ the residue field of a closed point $x \in C^1$. De Jeu \cite[Proposition 5.1] {dJ96} also defined the residue map 
\begin{equation}\label{DJresiduemap}
    \delta : \widetilde \Mcal_{(3)}^\bullet (F) \to \bigoplus_{x \in C^1} \widetilde \Mcal_{(2)}^\bullet (k(x))[-1]
\end{equation}
similarly as Goncharov's residue map \eqref{eee}. The complex $\widetilde \Mcal_{(3)}^\bullet(C)$ is also defined to be the mapping cone of \eqref{DJresiduemap}. As the maps $\varphi_{(3)}^2$ \eqref{phi32} and $\varphi_{(2)}^1$ \eqref{phin1} are defined universally up to sign, we have the following (possibly non-commutative) diagram
    \begin{equation*}
          \begin{split}  \xymatrix{
H^2(\widetilde \Mcal^\bullet_{(3)}(F))  \ar[d]_{\pm 2\delta} \ar[r]^{\varphi_{(3)}^2} & H^2_\Mcal(F, \Qbb(3)) \ar[d]^{\Res^\Mcal\hspace{0.9cm}} 
\\
\bigoplus_{x\in C^1} H^1 (\widetilde \Mcal^\bullet_{(2)}(k(x))) \ar[r]^{\simeq}_{\varphi_{(2)}^1}
 & \bigoplus_{x \in C^1} H^1_\Mcal(k(x), \Qbb(2)). }
 \end{split} 
    \end{equation*}
Recall that we have the following cup product of $K$-groups
\begin{equation*}
   \cup: K_3^{(2)}(k) \otimes K_1^{(1)}(F) \to K_4^{(3)}(F).
\end{equation*}
Since $K^{(j)}_{2j-i}(X) \simeq H^i_\Mcal(X, \Qbb(j))$ and $K_1^{(1)}(F) \simeq F_\Qbb^\times$, we have 
\begin{equation*}
     H^1_\Mcal(k, \Qbb(2)) \cup F_\Qbb^\times \subset H^2_\Mcal(F, \Qbb(3)).
\end{equation*}
Denote by $H =  H^1_\Mcal(k, \Qbb(2)) \cup F_\Qbb^\times$. De Jeu showed that $(\Res^\Mcal \circ  \varphi_{(3)}^2) \pm   2(\varphi_{(2)}^1 \circ  \delta)$ has image in  $\Res^\Mcal|_H (H)$ (see \cite[Theorem 5.2]{dJ96}). With the fixed choice of sign of $\varphi_{(2)}^1$ in Section \ref{DeJeucomplex}, we can choose the sign of $\varphi_{(3)}^2$ such that $(\Res^\Mcal \circ  \varphi_{(3)}^2) -   2(\varphi_{(2)}^1 \circ  \delta)$ has image in  $\Res^\Mcal|_H (H)$ (see \cite[diagram (15)]{dJ96}).
    \begin{equation*}
          \begin{split}  \xymatrix{
H^2(\widetilde \Mcal^\bullet_{(3)}(F))  \ar[d]_{2\delta} \ar[r]^{\varphi_{(3)}^2} & H^2_\Mcal(F, \Qbb(3)) \ar[d]^{\Res^\Mcal\hspace{0.9cm}} \   \supset \  H 
\\
\bigoplus_{x\in C^1} H^1 (\widetilde \Mcal^\bullet_{(2)}(k(x))) \ar[r]^{\simeq}_{\varphi_{(2)}^1}
 & \bigoplus_{x \in C^1} H^1_\Mcal(k(x), \Qbb(2)). }
 \end{split} 
    \end{equation*}
Let $\xi \in H^2(\widetilde \Mcal_{(3)}^\bullet(F))$. As $\Res^\Mcal|_H$ is injective (see e.g., \cite[Remark 4.4]{dJ00}), there exists a unique $h\in H$ such that
$$\Res^\Mcal|_H (h) = ((\Res^\Mcal \circ \varphi_{(3)}^2)  -  2(\varphi_{(2)}^1 \circ  \delta))(\xi).$$
Then we define a map
\begin{equation}\label{modification}
    \varphi_F : H^2(\widetilde\Mcal_{(3)}^\bullet(F))  \to H^2_\Mcal(F, \Qbb(3)),
\end{equation}
by setting $ \varphi_F (\xi) := \varphi_{(3)}^2(\xi) - h.$ It is a $\Qbb$-linear map making the following diagram commute
\begin{equation}\label{DeJeudiagram}
  \begin{split}  \xymatrix{
H^2(\widetilde \Mcal^\bullet_{(3)}(F)) \ar[d]_{2\delta} \ar[r]^{\varphi_F} & H^2_\Mcal(F, \Qbb(3)) \ar[d]^{\Res^\Mcal\hspace{0.9cm}} \\
\bigoplus_{x\in C^1} H^1 (\widetilde \Mcal^\bullet_{(2)}(k(x)) \ar[r]^{\simeq}_{ \varphi_{(2)}^1}
 & \bigoplus_{x \in C^1} H^1_\Mcal(k(x), \Qbb(2)).}
 \end{split}
\end{equation}  
This modification was mentioned briefly by De Jeu in \cite[Remark 5.3]{dJ96}. From diagram \eqref{DeJeudiagram}, $\varphi_F$ induces a map 
\begin{equation*}
    \varphi_C: H^2 (\widetilde \Mcal_{(3)}^\bullet(C)) \to H^2_\Mcal(C, \Qbb(3))
\end{equation*}
such that the following diagram commutes
\begin{equation}
\xymatrix{
    0 \ar[r] & H^2(\widetilde \Mcal_{(3)}^\bullet(C)) \ar[r] \ar[d]_{\varphi_C} & H^2(\widetilde \Mcal_{(3)}^\bullet(F)) \ar[r]^{ 2\delta\hspace{0.6cm} } \ar[d]_{\varphi_F} & \bigoplus_{x\in C^1} H^1(\widetilde \Mcal^\bullet_{(2)}(k(x))  \ar[d]^{\simeq} \\ 
     0 \ar[r] & H^2_\Mcal(C, \Qbb(3)) \ar[r] & H^2_\Mcal(F, \Qbb(3)) \ar[r]^{\Res^\Mcal\hspace{0.5cm}} & \bigoplus_{x\in C^1} H^1_\Mcal(k(x), \Qbb(2)),}
\end{equation}
where the lower horizontal sequence is the localization sequence of motivic cohomology mentioned in \eqref{motiviclocalization}.

\subsection{Relation to Goncharov's complexes}\label{Relating to Gon}
Let $F$ be an arbitrary field of characteristic 0.  De Jeu showed that there is a map $B_2(F) \to \widetilde{M}_{(2)}(F)$ given by $ \ \{x\}_2 \mapsto [x]_2$ (see \cite[Lemma 5.2]{dJ00}). This map 
fits into the following commutative diagram
\begin{equation}\label{psi12}
    \xymatrix{
    \Gamma(F, 2) : & B_2(F) \ar[r]\ar[d] & \bigwedge^2 F_\Qbb^\times \ar@{=}[d]\\
    \widetilde{\Mcal}_{(2)}^\bullet (F) : &\widetilde{M}_{(2)}(F) \ar[r]& \bigwedge^2 F_\Qbb^\times.
    }
\end{equation}
This diagram gives rise to a map $\psi_{(2)}^1: H^1(\Gamma(F, 2)) \to H^1(\widetilde{\Mcal}^\bullet_{(2)}(F))$. In particular, if $k$ is a number field, we have $ \psi_{(2)}^1: H^1(\Gamma(k, 2)) \xrightarrow{\simeq} H^1(\widetilde{\Mcal}^\bullet_{(2)}(k))$
is an isomorphism (see \cite[Section 5]{dJ95}). We then define $\beta_{(2)}^1 :  H^1(\Gamma(k, 2)) \to H^1_\Mcal (k, \Qbb(2))$ to be the composition of $\varphi_{(2)}^1$ and $\psi_{(2)}^1$
\begin{equation}
    \xymatrix{
    H^1(\Gamma(k, 2)) \ar[r]^{\beta_{(2)}^1}_{\simeq} \ar[rd]_{\psi_{(2)}^1}^{\simeq}& H^1_\Mcal (k, \Qbb(2))\\
    &H^1(\widetilde{\Mcal}_{(2)}^\bullet(k)) \ar[u]_{\varphi_{(2)}^1}^\simeq.
    }
\end{equation}

The map $B_2(F) \to \widetilde M_{(2)}(F), \{x\}_2 \mapsto [x]_2$ also fits into the following commutative diagram
\begin{equation}\label{psi23}
    \xymatrix{
    \Gamma(F, 3) :& B_3(F) \ar[r] & B_2(F) \otimes F^\times_\Qbb \ar[r] \ar[d] & \bigwedge^3 F^\times_\Qbb \ar@{=}[d]\\
    \widetilde{\Mcal}_{(3)}^\bullet(F) : &\widetilde{M}_{(3)}(F)     \ar[r] & \widetilde{M}_{(2)}(F) \otimes F^\times_\Qbb \ar[r]& \bigwedge^3 F^\times_\Qbb.
    }
\end{equation}
The middle vertical arrow in the diagram \eqref{psi23} sends objects of the form $\{x\}_x \otimes x$ to $[x]_2 \otimes x$, so that it maps the image of $B_3(F)$ to the image of $\widetilde{M}_{(3)}(F)$. It then induces a map
\begin{equation}\label{psiF}
    \psi_F : H^2(\Gamma(F, 3)) \to H^2 (\widetilde{\Mcal}^\bullet_{(3)}(F)).
\end{equation}

Now let $C$ be a smooth geometrically connected curve over a number field $k$. Let $F$ be its function field and for any $x \in C^1$, we denote by $k(x)$ the residue of $C$ at $x$. Since the residue maps of Goncharov and De Jeu are defined similarly, we have the following commutative diagram
\begin{equation*}
    \xymatrix{
    H^2(\Gamma(F, 3))   \ar[d]_{2\partial} \ar[r]^{\psi_F} & H^2 (\widetilde{\Mcal}^\bullet_{(3)}(F)) \ar[d]^{2\delta} \\
    \bigoplus_{x\in C^1} H^1(\Gamma(k(x), 2)) \ar[r]_{\psi_{(2)}^1}^\simeq
&\bigoplus_{x\in C^1} H^1 (\widetilde \Mcal^\bullet_{(2)}(k(x)).
    }
\end{equation*}
Then $\psi_F$ induces a map $\psi_C : H^2(\Gamma(C, 3)) \to H^2(\widetilde \Mcal_{(3)}^\bullet(C))$ that makes the following diagram commute
\begin{equation*}
\xymatrix{
 0 \ar[r] & H^2(\Gamma(C, 3)) \ar[r] \ar[d]^{\psi_C} & H^2(\Gamma(F, 3)) \ar[r]^{2\partial\quad}\ar[d]^{\psi_F} & \bigoplus_{x\in C^1} H^1 (\Gamma(k(x), 2)) \ar[d]\\
    0 \ar[r] & H^2(\widetilde \Mcal_{(3)}^\bullet(C)) \ar[r]  & H^2(\widetilde \Mcal_{(3)}^\bullet(F)) \ar[r]^{2\delta\hspace{0.6cm}}  & \bigoplus_{x\in C^1} H^1(\widetilde \Mcal^\bullet_{(2)}(k(x)).}
\end{equation*}
Putting $\beta_F := \varphi_F \circ \psi_F$, we have the following commutative diagram
\begin{equation}\label{dig1}
\xymatrix{
H^2(\Gamma(F, 3))  \ar@/^2pc/[rr]^{\beta_F}  \ar[d]_{ 2\partial} \ar[r]^{\psi_F} & H^2 (\widetilde{\Mcal}^\bullet_{(3)}(F)) \ar[d]^{2\delta} \ar[r]^{\varphi_F} & H^2_\Mcal(F, \Qbb(3)) \ar[d]^{\Res^\Mcal\hspace{0.9cm}} \\
\bigoplus_{x\in C^1} H^1(\Gamma(k(x), 2)) \ar@/_2pc/[rr]_{\beta_{(2)}^1}^{\simeq} \ar[r]^{\simeq}_{\psi_{(2)}^1}
&\bigoplus_{x\in C^1} H^1 (\widetilde \Mcal^\bullet_{(2)}(k(x)) \ar[r]^{\simeq}_{\varphi_{(2)}^1} & \bigoplus_{x \in C^1} H^1_\Mcal(k(x), \Qbb(2)).}
\end{equation}
Again, the map
\begin{equation}\label{beta}
    \beta_F : H^2(\Gamma(F, 3)) \to H^2_\Mcal(F, \Qbb(3))
\end{equation}
induces a map $\beta_C:  H^2(\Gamma(C, 3)) \to H^2_\Mcal(C, \Qbb(3))$ such that the following diagram commutes
\begin{equation}\label{dJdiagram}
\xymatrix{
    0 \ar[r] & H^2(\Gamma(C, 3)) \ar[r] \ar[d]_{\beta_C} & H^2(\Gamma(F, 3)) \ar[r]^{2\partial\hspace{0.6cm} } \ar[d]^{\beta_F} & \bigoplus_{x \in C^1} H^1(\Gamma(k(x), 2))  \ar[d]_{\simeq}^{\beta_{(2)}^1} \\ 
     0 \ar[r] & H^2_\Mcal(C, \Qbb(3)) \ar[r] & H^2_\Mcal(F, \Qbb(3)) \ar[r]^{\Res^\Mcal\hspace{0.5cm}} & \bigoplus_{x\in C^1} H^1_\Mcal(k(x), \Qbb(2)).}
\end{equation}
In particular, we have $\beta_C = \varphi_C \circ \psi_C$.

\subsection{Regulator maps}\label{regmaps} Let $C$ be a smooth proper geometrically connected curve over a number field $k$ and let $F$ be its function field. We have the following lemma, which is a consequence of De Jeu's theorem \cite[Theorem 3.5]{dJ00} and Goncharov's theorem \cite[Theorem 3.3]{Gon96}.
\begin{lm}[\text{De Jeu}]\label{3.2} 
   Let $\omega$ be a holomorphic 1-form on $C(\Cbb)$ such that $F_\dR(\omega) = \omega$, where $F_\dR$ is the action defined in Section \ref{DeligneCoho}. Let $\alpha = \sum_j c_j \{f_j\}_2 \otimes g_j \in H^2(\Gamma(F, 3))$ with $c_j \in \Qbb$ and $f_j \in F^\times \setminus \{1\},  g_j \in F^\times $. With the fixed sign of $\varphi_{(3)}^2$ as in Section \ref{dJmapsection}, we have 
   \begin{equation}
       \int_{C(\Cbb)} \reg_F (\beta_F(\alpha)) \wedge \bar \omega = 2\int_{C(\Cbb)} r_3(2)_F (\alpha) \wedge\bar \omega,
   \end{equation}
where $\beta_F$ is the map defined in \eqref{beta}, $\reg_F$ is Beilinson's regulator map \eqref{regF}, and $r_3(2)_F$ is Goncharov's regulator map \eqref{r32F}.
\end{lm}
\begin{proof}
We consider the regulator integral \eqref{regintF}
\begin{equation*}
    H^2_\Mcal(F, \Qbb(3)) \xrightarrow{\reg_F} H^1 (F, \Rbb(2))^+ \xrightarrow{\eta \mapsto \int_{C(\Cbb)} \eta \wedge\bar \omega} \Rbb(1).
\end{equation*}
The image of $H^1_\Mcal(k, \Qbb(2)) \cup F_\Qbb^\times$ under the regulator integral is trivial (see \cite[Theorem 4.2]{dJ96}). This can be seen by noting that 
\begin{equation*}
    \int_{C(\Cbb)} d \arg g \wedge \omega =  2\pi \int_{g^{-1} (\Rbb_{>0})}  \omega = 2\pi \int_{0}^\infty g_*\omega =0,
\end{equation*}
where $g_*$ is the pushforward by the correspondence from $C$ to $\Pbb^1$, and the fact that $g_*\omega = 0$ since $\Pbb^1$ has no holomorphic forms. Hence, for $\xi \in H^2(\widetilde\Mcal_{(3)}^\bullet(F))$, we have
\begin{equation}\label{factorquotient}
    \begin{aligned}
      \int_{C(\Cbb)} \reg_F(\varphi_F (
     \xi
     )) \wedge \bar \omega &= \int_{C(\Cbb)} \reg_F(\varphi_{(3)}^2(\xi) - h) \wedge \bar \omega 
\quad \text{for some } h \in H^1_\Mcal(k, \Qbb(2)) \cup F_\Qbb^\times \ (\text{see }\eqref{modification}) \\
&= \int_{C(\Cbb)} \reg_F (\varphi_{(3)}^2(\xi)) \wedge \bar \omega - \int_{C(\Cbb)} \reg_F (h) \wedge \bar \omega \\
&= \int_{C(\Cbb)} \reg_F(\varphi_{(3)}^2(\xi)) \wedge\bar \omega.
\end{aligned}
\end{equation}
Now let $\alpha = \sum_j c_j \{f_j\}_2 \otimes g_j \in H^2 (\Gamma(F, 3))$.  Then $\psi_F(\alpha) = \sum_j c_j [f_j]_2 \otimes g_j \in H^2(\widetilde{\Mcal}_{(3)}^\bullet(F))$. Using \eqref{factorquotient} with $\xi = \psi_F(\alpha)$,  we have  
\begin{equation*}
\begin{aligned}
      \int_{C(\Cbb)} \reg_F(\beta_F(\alpha)) \wedge \bar \omega =  \int_{C(\Cbb)} \reg_F(\varphi_F (\xi)) \wedge \bar \omega   &= \int_{C(\Cbb)} \reg_F(\varphi_{(3)}^2(\xi)) \wedge \bar \omega.  \\
\end{aligned}
\end{equation*}
With the fixed sign of $\varphi_{(3)}^2$ in Section \ref{dJmapsection}, one can show that
\begin{equation*}
    \int_{C(\Cbb)} \reg_F(\varphi_{(3)}^2(\xi)) \wedge \bar \omega = \dfrac{8}{3} \sum_j c_j \int_{C(\Cbb)} \log|g_j| \theta(1-f_j, f_j) \wedge \bar \omega
\end{equation*}
(see \cite[Theorem 3.5]{dJ00}). On the other hand, by some computations, one obtains the following formula
\begin{equation*}
    \int_{C(\Cbb)} r_3(2)_F (\alpha) \wedge\bar \omega = \dfrac{4}{3} \sum_j c_j \int_{C(\Cbb)} \log|g_j| \theta(1-f_j, f_j) \wedge \bar \omega
\end{equation*}
(see \cite[Theorem 3.3]{Gon96}). Therefore, we have 
\begin{equation*}
    \int_{C(\Cbb)} \reg_F (\beta_F(\alpha)) \wedge \bar \omega = 2 \int_{C(\Cbb)}  r_3(2)_F(\alpha) \wedge \bar \omega.
\end{equation*}
\end{proof}

As $C$ is proper, the map
\begin{equation}\label{injmap}
    H^1(C(\Cbb), \Rbb(2))^+ \to \Hom(H^0(C(\Cbb), \Omega^1)^+, \Rbb(1)), \qquad \eta \mapsto (\omega \mapsto \int_{C(\Cbb)} \eta \wedge \bar \omega)
\end{equation}
is injective (see \cite[Remark 3.1]{dJ00}). By Lemma \ref{3.2}, for $\alpha \in H^2(\Gamma(C, 3))$, we have
 \begin{equation*}
       \int_{C(\Cbb)} \reg_C (\beta_C(\alpha)) \wedge \bar \omega = 2\int_{C(\Cbb)} r_3(2)_C (\alpha) \wedge\bar \omega,
   \end{equation*}
where $r_3(2)_C$ is Goncharov's regulator map (\ref{Gonmap}). Therefore, we have
the following commutative diagram
\begin{equation}\label{124}
    \xymatrix{
       H^2(\Gamma(C, 3))  \ar[r]^{\psi_C}  \ar@/^2pc/[rr]^{\beta_C} \ar[rd]_{r_3(2)_C} &H^2(\widetilde \Mcal_{(3)}^\bullet(C))\ar[r]^{\varphi_C} \ar[d] & H^2_\Mcal(C, \Qbb(3))  \ar[ld]^{\frac{1}{2}\reg_C}\\
     &H^1(C(\Cbb), \Rbb(2))^+,&
    }
\end{equation}
where  the middle vertical map is just the composition $\frac{1}{2} \reg_C \circ \varphi_C$. In \cite[Corollary 5.5]{dJ00}, De Jeu showed that the images of the $r_3(2)_C$ and $\reg_C$, as vector spaces, are the same.

\begin{lm}\label{integral}  Let $\alpha = \sum_j c_j \{f_j\}_2 \otimes g_j \in H^2(\Gamma(F, 3))$  with $c_j \in \Qbb$ and $f_j\in F^\times \setminus \{1\}$, $ g_j \in F^\times$. Denote by $Y = C\setminus Z$ where $Z$ is the closed subscheme of $C$ consisting of the zeros and poles of $f_j, 1-f_j, g_j$ for all $j$. With the fixed choice of signs of $\varphi_{(2)}^1$ and $\varphi_{(3)}^2$, we have
    \begin{equation*}
        \int_\gamma  \reg_F(\beta_F(\alpha)) = 2 \int_\gamma  r_3(2)_F(\alpha) \qquad \text{ for any loop } \gamma \in H_1(Y(\Cbb), \Zbb).
    \end{equation*}
\end{lm} 
\begin{proof}
First notice that $\beta_F(\alpha) \in H^2_\Mcal(F, \Qbb(3))$ actually belongs to the subgroup   $H^2_\Mcal(Y, \Qbb(3))$ (see \cite[Theorem 5.4]{Bru22}). Then $\reg_F(\beta_F(\alpha))$ belongs to $H^1(Y(\Cbb), \Rbb(2))^+$. In particular, the integral $\int_\gamma \reg_F(\beta_F(\alpha))$ is well-defined for any loop $\gamma \in H_1(Y(\Cbb), \Zbb)$. Since $r_3(2)_F(\alpha)$ is represented by the form $\sum_j c_j \rho(f_j, g_j)$ which defines an element in $H^1(Y(\Cbb), \Rbb)^+$, the integral $\int_\gamma r_3(2)_F (\alpha)$ is also well-defined. 
   
In particular, we have $\reg_F(\beta_F(\alpha)) - 2 r_3(2)_F(\alpha) \in H^1(Y(\Cbb), \Rbb(2))^+$. We consider the Mayer-Vietoris sequence
\begin{equation}\label{seq}
\xymatrix{0 \ar[r] & H^1(C(\Cbb), \Rbb(2))^+ \ar[r] & H^1(Y(\Cbb), \Rbb(2))^+ \ar[rr]^{\quad \oplus (2 \pi i )^{-1} \Res_p} && \oplus_{p \in Z(\Cbb)} \ \Rbb(1),}
\end{equation}
where $\Res_p$ is the residue map defined in Definition \ref{resmap}. 
We are going to show that $\reg_F(\beta_F(\alpha)) - 2 r_3(2)_F(\alpha)$ extends to $H^1(C(\Cbb), \Rbb(2))^+$. Let $p \in Z(\Cbb)$ supported on a closed point $x \in Z^1$ with the embedding $\sigma : k(x) \hookrightarrow \Cbb$, i.e.,
\begin{equation*}
   \xymatrix{ \Spec \Cbb \ar@/^1pc/[rr]^{\id} \ar[r] \ar[rd]^p \ar[d]_\sigma & Z \times_\Qbb \Cbb \ar[r] \ar[d] & \Spec \Cbb \ar[d]\\
    \Spec k(x) \ar[r]_x & Z \ar[r] & \Spec \Qbb.}   
\end{equation*}
With the fixed signs of $\varphi_{(2)}^1$ and $\varphi_{(3)}^2$ as before, we have the following commutative diagram 
\begin{equation}\label{dig2}
\xymatrix{
H^2(\Gamma(F,3)) \ar[r]^{\beta_F} \ar[d]_{ 2\partial_x} & H^2_\Mcal(F, \Qbb(3))\ar[d]^{\Res^\Mcal_x} \ar[r]^{\reg_F} & H^1(F, \Rbb(2))^+ \ar[d]^{(2\pi i)^{-1} \Res_{p}} & \\
H^1(\Gamma(k(x), 2)) \ar[r]_{\beta_{(2)}^1} & H^1_\Mcal(k(x), \Qbb(2)) \ar[r]_{\qquad \reg_{k(x)}\ } & \Rbb(1),
}
\end{equation}
here $\reg_{k(x)}$ is the composition $H^1_\Mcal(k(x), \Qbb(2)) \xrightarrow{\sigma^*}  H^1_\Mcal(\Cbb, \Qbb(2)) \xrightarrow{\reg_\Cbb} H^1_\Dcal(\Cbb, \Rbb(2)) \simeq \Rbb(1) $ as mentioned in \eqref{bjg}. The commutativity of the right square follows from the compatibility of the Beilinson regulators and residues maps. We then have 
\begin{align*}
\Res_p(\reg_F(\beta_F(\alpha))) 
&= (4 \pi i ) \ \reg_{k(x)} (\beta_{(2)}^1(\partial_x (\alpha))) \\
&= (4\pi i)   (\reg_{k(x)} \circ \beta_{(2)}^1) (\sum_{j} c_j v_x(g_j)\{f_j(x)\}_2), 
\end{align*} 
where the last equality follows from the definition of Goncharov's residues map \eqref{Gonresmap}. As mentioned in \eqref{bjg}, the sign of $\varphi_{(2)}^1$ is chosen such that the lower composition map in the diagram \eqref{dig2} is induced by the map $\{z\}_2 \mapsto i D(\sigma(z))$. Therefore, we have
\begin{align*}
\Res_p(\reg_F(\beta_F(\alpha))) &=  (4 \pi i)i  \ \sum_j c_j v_x(g_j) D(\sigma(f_j(x)))\\
&= - 4 \pi   \sum_j c_j v_{\sigma(x)}(g_j) D(f_j(\sigma(x) )) \quad (\text{as } f_j, g_j \in k(C))\\
&= 2 \  \Res_p (r_3(2)_F(\alpha))\quad  (\text{by Lemma } \ref{4.3}).
\end{align*}
So $\Res_p(\reg_F(\beta_F(\alpha)) - 2 r_3(2)_F(\alpha)) = 0 \text{ for all } p \in Z(\Cbb),$
hence $\reg_F(\beta_F(\alpha)) - 2 r_3(2)_F(\alpha)$ extends to $H^1(C(\Cbb), \Rbb(2))^+$. Therefore,  the class $\reg_F(\beta_F(\alpha)) - 2 r_3(2)(\alpha)$ is represented by $\eta + dt$, where $\eta$ is a $F_\dR$-invariant closed differential 1-form on $C(\Cbb)$  and  $t$ is a logarithmic growth function on $Y(\Cbb)$.  Now let $\omega$ be a holomorphic 1-form on $C(\Cbb)$ such that $F_\dR(\omega) = \omega$. Since $t$ is a logarithmic growth function on $Y(\Cbb)$, we have 
    $$\int_{C(\Cbb)} dt \wedge \bar \omega = \int_{C(\Cbb)} d(t \bar \omega) = 0$$
by using Stokes' theorem (see the proof of \cite[Theorem 4.6]{dJ96}). We then have 
    \begin{equation*}
        \int_{C(\Cbb)} \eta \wedge \bar \omega  = \int_{C(\Cbb)} (\eta + dt) \wedge \bar \omega = \int_{C(\Cbb)}  (\reg_F(\beta_F(\alpha)) - 2 r_3(2)_F(\alpha)) \wedge \bar \omega =  0,
    \end{equation*}
    where the last equality is by Lemma \ref{3.2}. Since $\omega$ is an arbitrary $F_\dR$-invariant holomorphic 1-form and such forms span a real vector space dual to $H^1(C(\Cbb), \Rbb(2))^+$ (\cite[Remark 3.1]{dJ00}), we obtain that $\eta = ds$ for some function $s$ on $C(\Cbb)$. So 
    $\reg_F(\beta_F(\alpha)) - 2 r_3(2)_F(\alpha))$ is represented by $d(s+t)$ for some logarithmic growth function  $s+t$ on  $Y(\Cbb)$. Hence
    \begin{equation*}
        \int_{\gamma} \reg_F(\beta_F(\alpha)) = 2 \int_\gamma r_3(2)_F(\alpha) \qquad \text{for any loop }\gamma \in H_1(Y(\Cbb), \Zbb).
    \end{equation*}
\end{proof}

\section{Main result}\label{main}
In Section \ref{Deligne}, we construct an element in Deligne cohomology and  in Section \ref{relate1}, we  connect it to the Mahler measure. In Section  \ref{motivic}, we construct an element in $K_4^{(3)}$ of a curve such that its regulator   is related to the Deligne cohomology class constructed in Section \ref{Deligne}. Finally, we prove Theorem \ref{0} in Section \ref{main}.

\subsection{Constructing an element in Deligne cohomology}\label{Deligne}
Let $P(x ,y, z) \in \Qbb[x, y, z]$ be an irreducible polynomial. We denote by $V_P$ the zero locus of $P$ in $(\Cbb^\times)^3$ and $V_P^\reg$ the smooth part of $V_P$. For $f, g, h \in \Cbb(V_P^\reg)^\times$, we recall the differential form mentioned in (\ref{kk})
\begin{equation}\label{84}
\begin{aligned}
    \eta(f, g, h) &= \log|f| \left(\dfrac{1}{3} d \log|g|\wedge d \log|h| - d \arg(g) \wedge d \arg(h)\right)\\
    &\quad+ \log|g| \left(\dfrac{1}{3} d \log|h|\wedge d \log|f| - d \arg(h) \wedge d \arg(f)\right)\\
    &\quad+ \log|h| \left(\dfrac{1}{3} d \log|f|\wedge d \log|g| - d \arg(f) \wedge d \arg(g)\right).
    \end{aligned}
\end{equation}
The form is bilinear and antisymmetric in $f, g, h$. It defines on $V_P^\reg \setminus S_{f, g, h}$ where $S_{f, g, h}$ is the set of zeros and poles of $f, g$ and $h$. Moreover, $\eta(f, g, h)$ is a closed form on $V_P^\reg \setminus S_{f, g, h}$ since
$$d \eta(f, g, h) = \mathrm{Re}\left(\dfrac{d f}{f} \wedge \dfrac{d h}{h} \wedge\dfrac{d g}{g}\right),$$
which is zero in $V_P^\reg\setminus S_{f, g, h}$. 
\begin{lm}
The differential form $\eta(x,y,z)$ defines an element in the Deligne cohomology $H^{3}_{\Dcal}(\Gbb^{3}_m,\Rbb(3)).$ Moreover, it represents the class $\reg_{\Gbb^3_m} (\{x, y, z\})$, where $\reg_{\Gbb^3_m} : H^3_\Mcal (\Gbb_m^3, \Qbb(3)) \to H^3_\Dcal(\Gbb^3_m, \Rbb(3))$ is the Beilinson regulator map and $\{x, y, z\} \in H^3_\Mcal (\Gbb_m^3, \Qbb(3))$ is the Milnor symbol.
\end{lm}

\begin{proof}
By definition, $\eta(x, y, z) \in E^2_{\log, \Rbb}(\Gbb_m^3)$, and defines an element in $H^3_\Dcal(\Gbb^3_m, \Rbb(3))$. By an observation in Section  \ref{Beireg},  $\reg_{\Gbb^3_m} (\{x, y, z\}) $ is represented by the cup product $\log |x| \cup \log |y| \cup \log |z|$ in Deligne cohomology. By cup product's formula (\ref{65}), we have
\begin{equation*}\label{66}
    \begin{aligned}
     \log |x| \cup \log |y| \cup \log |z|
     &= (\log |x| \cup \log |y|) \cup \log |z|\\
     &= (-1)^2 r_2 (\log|x| \cup \log |y|) \log |z| + (\log|x| \cup \log |y|) \ r_1(\log|z|)\\
     &= \left(\partial \left(\dfrac{1}{2}\log |x| \dfrac{dy}{y} - \dfrac{1}{2} \log |y| \dfrac{d x}{x}\right)- \bar \partial \left(\dfrac{1}{2}\log|y| \dfrac{d \bar x}{\bar x} - \dfrac{1}{2}\log|x|\dfrac{d \bar y}{\bar y}\right)\right)\log |z|\\
     &\quad + i \cdot (\log|x| d \arg y - \log |y| d \arg (x))\wedge (\partial \log |z| - \bar \partial \log |z|)\\
     &=\left(\dfrac{1}{2} \dfrac{dx}{x} \wedge \dfrac{dy}{y} + \dfrac{1}{2} \dfrac{d\bar x}{\bar x} \wedge \dfrac{d\bar y}{\bar y}\right) \log |z|  - (\log|x| d\arg y - \log |y| d\arg x) \wedge d\arg z\\
     &= \log |z| \left( d\log|x| \wedge d\log |y| - d\arg (x) \wedge d\arg y \right) \\
     &\quad - \log |y| \  d\arg (z) \wedge d\arg x - \log |x|\  d\arg (y) \wedge d\arg z.
    \end{aligned}
\end{equation*}
Therefore, 
\begin{equation*}
    \begin{aligned}
     &\eta(x, y, z) - \log |x| \cup \log |y| \cup \log |z| \\
     &= \dfrac{1}{3} \log |x| d \log |y| \wedge d \log |z| + \dfrac{1}{3} \log |y| d \log |z| \wedge d \log |x| -\dfrac{2}{3} \log |z| d \log |x| \wedge d \log |y|\\
     &= -\dfrac{1}{3} d (\log |x| \log |z| \ d \log |y| ) + \dfrac{1}{3} d (\log |y| \log |z| \ d \log |x|),
    \end{aligned}
\end{equation*}
which is an exact form, hence $\reg_{\Gbb^3_m} (\{x, y, z\})$ is represented by $\eta(x, y, z)$.
\end{proof}

Consequently, pulling back $\eta(x, y, z)$ by the embedding  $V_P^\reg \xhookrightarrow{i} \Gbb^3_m$, We have $\eta(x, y, z)_{|_{V_P^\reg}}$ is a representative of $\reg_{V_P^\reg} (\{x, y, z\})$ in $H^3_\Dcal(V_P^\reg, \Rbb(3)).$ We come to the definition of \textit{exact polynomials}.

\begin{definition}[\textbf{Exact polynomial}]\label{a}
A polynomial $P(x, y, z)$ is called \textit{exact} if $\reg_{V_P^\reg}(\{x, y, z\})$ is trivial, i.e., $\eta(x, y, z)$ is an exact form on $V_P^\reg$.
\end{definition}

\begin{rmk} \label{6.3}
If $P$ satisfies Lal\'{i}n's condition (see Theorem \ref{lalinthm}(iii)):
\begin{equation}\label{lll}
    x \wedge y \wedge z = \sum_j f_j \wedge (1-f_j)\wedge g_j \quad \text{in } \bigwedge^3 \Qbb(V_P)^\times_\Qbb
\end{equation}
for some functions $f_j \in \Qbb(V_P)^\times \setminus \{1\}$ and $g_j \in \Qbb(V_P)^\times$, then $P$ is exact because $ \eta(x, y, z)|_{V_P^\reg} = \sum_j \eta(f_j, 1-f_j, g_j) = \sum_j d\rho(f_j, g_j) = d(\sum_j \rho(f_j, g_j))$, where $\rho(f, g)$ is the differential form defined in \eqref{bb}. In particular, the polynomials of the form $A(x) + B(x) y + C(x) z$, where $A(x), B(x), C(x)$ are products of cyclotomic polynomials, are exact. Indeed, we have
\begin{equation}\label{104}
    \begin{aligned}
        x \wedge y \wedge z &= x \wedge y \wedge \dfrac{A(x) + B(x)y}{C(x)}\\
        &= x \wedge y \wedge \left(\dfrac{A(x)}{C(x)} \cdot  \dfrac{A(x) + B(x) y}{A(x)}\right)\\
        &= x \wedge y \wedge \dfrac{A(x)}{C(x)} + x\wedge y \wedge \left(1 + \dfrac{B(x) y}{A(x)}\right)\\
        &= x \wedge y \wedge \dfrac{A(x)}{C(x)} + x \wedge \dfrac{B(x) y}{A(x)} \wedge \left(1 + \dfrac{B(x) y}{A(x)}\right) - x \wedge \dfrac{B(x)}{A(x)} \wedge \left(1 + \dfrac{B(x) y}{A(x)}\right).
    \end{aligned}
\end{equation}
For cyclotomic polynomials $\Phi(x)$, we have
\begin{equation*}
    \begin{aligned}
        x \wedge y \wedge \Phi_n(x) &= x \wedge y \wedge \dfrac{x^n - 1}{\prod_{d|n, d\neq n}\Phi_d(x)}\\
        &= x \wedge y \wedge (x^n - 1) - 
        \sum_{d|n, d \neq n} x \wedge y \wedge \Phi_d(x)\\
        &= -\dfrac{1}{n}x^n \wedge (1-x^n) \wedge y - \sum_{d|n, d\neq n} x \wedge y \wedge \Phi_d.
    \end{aligned}
\end{equation*}
For $n=1$, $x \wedge y \wedge (x+1) = -x \wedge (1+x) \wedge y$. So we get (\ref{lll}) by induction on $n$. 
\end{rmk}

From now on, let us assume our polynomial $P$ satisfies the condition 
\eqref{lll}. We consider the involution
\begin{equation}\label{tau1}
    \tau : \Gbb^3_m \to \Gbb^3_m, (x, y, z) \mapsto (1/x, 1/y, 1/z),
\end{equation}
which maps $V_P$ to $V_{P^*}$, where $P^* (x, y, z) := \bar P(1/x, 1/y, 1/z) =  P(1/x, 1/y, 1/z)$. Let $W_P$ be the curve defined by
\begin{equation}
    \begin{cases} P(x, y, z) = 0,\\
    P(1/x, 1/y, 1/z) = 0.
    \end{cases}
\end{equation}
We call $W_P$ the \textit{Maillot variety}. The restriction $\tau_{|_{W_P}} : W_P \to W_P$ is an automorphism. We view $W_P$ as a curve over $\Qbb$. Then let $C$ be the normalization of $W_P$. The condition \eqref{lll} implies that 
\begin{equation}\label{llll}
    x \wedge y \wedge z = \sum_j f_j \wedge (1-f_j) \wedge g_j \quad \text{ in } \bigwedge^3 \Qbb(C)^\times_\Qbb
\end{equation}
for some functions $f_j \in \Qbb(C)^\times \setminus \{1\}$ and $g_j \in \Qbb(C)^\times$.

\begin{definition}\label{xiandxi*} Let $F = \Qbb(C)$ be the function field of $C$. We denote by 
\begin{equation}\label{76}
    \xi := \sum_j \{f_j\}_2 \otimes g_j, \qquad \xi^* := \sum_j \{f_j \circ \tau\}_2 \otimes (g_j \circ \tau), \qquad \lambda := \xi + \xi^*,
\end{equation}
which are elements in $B_2(F)  \otimes F^\times_\Qbb$. Let us consider the following closed subschemes of $V_P$ and $V_{P^*}$, respectively 
\begin{equation}
    Z_1 = \{\text{zeros and poles of  } f_j, 1-f_j, g_j \text{ on } V_P \text{ for all } j\},
\end{equation}
\begin{equation}
    Z_2 = \{\text{zeros and poles of  } f_j \circ \tau, 1-f_j\circ \tau, g_j \circ \tau \text{ on } V_{P^*} \text{ for all } j\}.
\end{equation}
We define the following differential 1-forms on $V_P^\reg \setminus Z_1$ and $V_{P^*}^\reg\setminus Z_2$, respectively
\begin{equation}
    \rho(\xi) := \sum_j \rho(f_j, g_j), \qquad \rho(\xi^*) := \sum_j \rho(f_j \circ \tau, g_j \circ \tau),
\end{equation}
where $\rho(f, g)$ is mentioned in $\eqref{bb}$. Denote by $Z$ the following closed subscheme of $C$ 
\begin{equation}
    Z = \{\text{zeros and poles of } f_j, 1-f_j, g_j, f_j \circ \tau, 1-f_j \circ \tau, g_j \circ \tau \text{ on } C \text{ for all } j\}, 
\end{equation}
and set  $Y = \iota(W_P^\reg) \setminus Z$, where $\iota : W_P^\reg \hookrightarrow C$ is the canonical embedding. Using the canonical embeddings of $Y(\Cbb)$ into $V_P$ and $V_{P^*}$, we define the following differential 1-form on $Y(\Cbb)$
\begin{equation}
   \rho(\lambda) =  \rho(\xi)|_{Y(\Cbb)} + \rho(\xi^*)|_{Y(\Cbb)}.
\end{equation}
\end{definition}

\begin{lm}\label{o}  The element $\lambda$  defines a class in  $H^2 (\Gamma(F, 3))$. Moreover, $\lambda$ defines an element in $H^2(\Gamma(Y, 3))$.
\end{lm}
\begin{proof}
We recall the polylogarithmic complex of Goncharov
\begin{center}
\begin{tikzcd}[cramped,row sep=small]
\Gamma(F, 3): \quad  B_3(F) \arrow[r] & B_2(F) \otimes F^\times_\Qbb  \arrow[r, "\alpha_3(2)"] & \bigwedge^3 F^\times_\Qbb \\ 
& \{f\}_2 \otimes g \arrow[r, mapsto] & (1-f) \wedge f \wedge g.
\end{tikzcd}
\end{center}
We have
\begin{equation*}
    \alpha_3(2)(\xi) = \sum_j \alpha_3(2) (\{f_j\}_2 \otimes g_j) = \sum_j (1-f_j) \wedge f_j \wedge g_j = -x\wedge y \wedge z,
    \end{equation*}
and 
\begin{equation*}
    \begin{aligned}
    \alpha_3(2)(\xi^*) = \sum_j \alpha_3(2) \left(\{f_j\circ \tau\}_2 \otimes (g_j \circ \tau)\right)
    &= \sum_j (1-f\circ \tau) \wedge (f_j \circ \tau) \wedge (g_j \circ \tau)\\
    &= \tau^* \left(\sum_j  (1-f_j) \wedge f_j \wedge g_j\right)\\
    &= \tau^* (- x \wedge y \wedge z)\\
    &= -\dfrac{1}{x}\wedge \dfrac{1}{y} \wedge \dfrac{1}{z}\\
    &=  x \wedge y \wedge z,
    \end{aligned}
\end{equation*}
so $\alpha_3(2) (\lambda) = \alpha_3(2)(\xi) + \alpha_3(2)(\xi^*) = 0$. Then $\lambda$ defines a class in $H^2(\Gamma(F, 3))$.  Now we consider the following exact sequence (see Section \ref{Gonres}):
\begin{equation}
    0 \to H^2(\Gamma(Y, 3))\to H^2(\Gamma(F, 3)) \xrightarrow{\oplus \partial_p} \bigoplus_{p \in Y^1} H^1(\Gamma(\Qbb(p), 2)),
\end{equation}
where $Y^1$ is the set of closed points of $Y$. The residue of $\lambda$ at $p \in Y^1$ is given by
\begin{equation}\label{resbloch}
   \partial_p (\lambda)  = \sum_j v_p(g_j) \{f_j(p)\}_2 +  v_p(g_j \circ \tau) \{f_j \circ \tau(p)\}_2 \in H^1(\Gamma(\Qbb(p), 2),
\end{equation}
which is trivial for every point $p \notin S$, where $S$ is the following closed subscheme of $C$
\begin{equation}\label{S}
    S  = \{\text{zeros and poles of } g_j, g_j \circ \tau \text{ on } C \text{ for all } j\}.
\end{equation}
We  have $\partial_p(\lambda) = 0$ for all $p\in Y^1$, hence  $\lambda$ defines an element in $H^2(\Gamma(Y, 3))$. 
\end{proof}

\begin{lm}\label{11}
The differential 1-form  $\rho(\lambda)$ defines an element in $H^2_\Dcal(Y_\Rbb, \Rbb(3)) \simeq H^1 (Y(\Cbb), \Rbb(2))^+$. For any point $p \in C(\Cbb)$, the residue of $\rho(\lambda)$ at $p$ is given by
\begin{equation}\label{16} 
\mathrm{Res}_p (\rho(\lambda)) = -2\pi \left(\sum_j v_p(g_j) D(f_j(p)) + v_p(g_j \circ \tau) D(f_j \circ \tau)(p)\right),
\end{equation}
where $D$ is the Bloch-Wigner dilogarithm function \eqref{dilog}.
\end{lm}

\begin{proof}
The first statement follows from the fact that $\rho(\lambda)$ represents $r_3(2)_Y (\lambda)$, where
\[r_3(2)_Y : H^2
(\Gamma(Y, 3)) \to H^2_\Dcal(Y, R(3)) \simeq H^1(Y (\Cbb), R(2))^+\]
is Goncharov’s regulator map \eqref{Gonmap}, or by checking directly that $\rho(\lambda)$ is closed and fixed under the action
of the involution $F_\dR$. The formula \eqref{16} follows directly from Lemma \ref{4.3}.
\end{proof}

\begin{rmk}\label{4.7} 
If all the residues $u_p:= \partial_p (\lambda)$ are trivial for all $p \in S$ (see \eqref{S}), then $\lambda$ defines a unique class $\lambda_C$ in $H^2(\Gamma(C, 3))$ and $\rho(\lambda)$ represents the class $r_3(2)_C(\lambda_C) \in H^2_\Dcal(C, \Qbb(3)) \simeq H^1(C(\Cbb), \Rbb(2))^+$.
\end{rmk}

\subsection{Relating the Mahler measure to the Deligne cohomology}\label{relate1}  We retain the notation from the previous section. In this section, we connect $\rho(\lambda)$ to the Mahler measure of $P$. Recall that the Deninger chain associated to $P$ is defined by
\begin{equation}\label{24}
    \Gamma = \{(x, y, z) \in (\Cbb^\times)^3 : P(x, y, z)=0, |x|=|y|=1, |z|\ge 1\}.
\end{equation}
Its orientation is induced from $\Tbb^2$: for each $(x_0, y_0)\in \Tbb^2$, there are finitely many $z \in \Cbb$ such that $|z| \ge 1$ and $P(x_0, y_0, z) = 0$, then by letting $(x_0, y_0)$ runs on the torus $\Tbb^2_{(x, y)}$ along the usual orientation, we get the orientation of $\Gamma$.  Its boundary is given by 
$$\partial\Gamma = \{(x, y, z) \in (\Cbb^\times)^3 : P(x, y, z) = 0, |x| = |y| = |z| = 1\}.$$ 
Deninger \cite[Proposition 3.3]{Den97} showed that if  $\Gamma$ is contained in $V_P^\mathrm{reg}$, then we get the following formula 
\begin{equation}\label{4.2.2}
    \m (P) = \m (\tilde P) - \dfrac{1}{4\pi^2}\int_{\Gamma} \eta(x, y, z),
\end{equation}
where $\tilde P(x, y)$ is the leading coefficient of $P(x, y , z)$ considered as a polynomial in $z$. If furthermore, $\partial \Gamma = \emptyset$, then $[\Gamma] \in H_2(V_P^\reg, 
\Zbb)$ and the Mahler measure is written as a pairing in Deligne cohomology
\begin{equation*}
    \m(P) =\m(\tilde{P}) - \dfrac{1}{4\pi^2} \langle [\Gamma], \reg_{ V_P^\reg} (\{x, y, z)\}\rangle_{V_P^\reg}.
\end{equation*}

Since $P(x, y, z)$ has rational coefficients, we can write 
$$\partial \Gamma = \{P(x, y, z) = P(1/x, 1/y, 1/z) =0\} \cap \{|x|=|y|= |z|= 1\},$$
which is contained in $W_P$, and may contain some singularities of $W_P$. 
We have the following lemma.

\begin{lm}\label{90} We assume that
\begin{equation}\label{lalin'}
    x \wedge y \wedge z = \sum_j f_j \wedge (1-f_j)\wedge g_j \quad \text{in } \bigwedge^3 \Qbb(V_P)^\times_\Qbb.
\end{equation}
Suppose that $\Gamma$ is contained in $V_P^\reg$ and that $\partial \Gamma$ is contained in $Y(\Cbb)$ (see Definition \ref{xiandxi*}). Then $\partial \Gamma$ defines an element in the singular homology group $H_1(Y(\Cbb), \Zbb)^+$, where $``+"$ denotes the invariant part by the complex conjugation, and we can write the Mahler measure as a pairing in Deligne cohomology of $Y_\Rbb$
\begin{equation}\label{4.9}
    \m (P) = \m (\tilde P) - \dfrac{1}{8\pi^2} \left<[\partial \Gamma], [\rho(\lambda)]\right>_{Y},
\end{equation}
where the pairing is given by
\begin{equation}
    \left<\cdot, \cdot\right>_Y : H_1(Y(\Cbb), \Zbb)^+ \times H^{1}(Y(\Cbb), \Rbb(2))^+ \to \Rbb(2).
\end{equation}
\end{lm}

\begin{proof}
Since $\Gamma \subset V_P^\reg$ and $\partial \Gamma \subset Y(\Cbb)$,  $\partial \Gamma$ defines an element in $H_1(Y(\Cbb), \Zbb)$ by considering the following sequence
\begin{center}
\begin{tikzcd}[cramped,row sep=small]
 H_2(V_P^\reg, \partial \Gamma, \Zbb) \arrow[r] & H_1(\partial \Gamma, \Zbb) \arrow[r] &  H_1(Y(\Cbb), \Zbb) \\
 {[\Gamma]} \arrow[r, mapsto, shorten=5mm] & {[\partial \Gamma]} \arrow[r, mapsto, shorten=5mm] & {[\partial \Gamma]}.
\end{tikzcd}
\end{center}
Now we show that $\partial \Gamma$ is invariant under the complex conjugation. Notice that the action of the complex conjugation on $\partial \Gamma$ is the same as the action of the involution $\tau$ \eqref{tau1} on $\partial \Gamma$ because $\bar x = x^{-1}$ for $x \in \Tbb^1$. So it suffices to show that $\tau$ fixes $\partial \Gamma$. Clearly, $\tau(\partial \Gamma) = \partial \Gamma$ as sets. We claim that $\tau$ preserves the orientation of $\partial \Gamma$. Notice that the orientation of $\partial \Gamma$ is induced from the orientation of $\Gamma$, and the orientation of $\Gamma$ comes from the orientation of $\Tbb^2$. We have
\begin{equation}\label{T2}
    \tau|_{\Tbb^2} : \Tbb^2 \to \Tbb^2, (x, y) \mapsto (1/x, 1/y)
\end{equation}
preserves the orientation of $\Tbb^2$, hence $\tau$ preserves the orientation of $\partial \Gamma$. 
Note that the condition \eqref{lalin'} implies that $\eta(x, y, z)|_{V_P^\reg} = d \rho(\xi)$. Then by applying Stokes' theorem to \eqref{4.2.2}, we get 
\begin{equation}\label{good}
    \m (P) = \m (\tilde P) - \dfrac{1}{4\pi^2} \int_{\partial \Gamma} \rho (\xi)|_{Y(\Cbb)}.
\end{equation}
We have
$$\int_{\partial \Gamma} \rho(\xi)|_{Y(\Cbb)} = \int_{\tau(\partial \Gamma)} \tau^* (\rho(\xi)|_{Y(\Cbb)}) = \int_{\partial \Gamma} \tau^*(\rho(\xi)|_{Y(\Cbb)}) = \int_{\partial \Gamma} \rho(\xi^*)|_{Y(\Cbb)},$$
where the second equality is because $\tau(\partial \Gamma) = \partial \Gamma$ as sets and $\tau$ preserves the orientation of $\partial \Gamma$. Then by equation \eqref{good}, we get $$\m(P) - \m (\tilde P) = - \dfrac{1}{4 \pi^2}\int_{\partial \Gamma} \rho(\xi)|_{Y(\Cbb)} =  - \dfrac{1}{8 \pi^2} \int_{\partial \Gamma} \rho(\xi)|_{Y(\Cbb)} +\rho(\xi^*)|_{Y(\Cbb)} = -   \dfrac{1}{8 \pi^2} \int_{\partial \Gamma} \rho(\lambda),$$
which is exactly the formula \eqref{4.9}.
\end{proof}

\subsection{Constructing an element in the motivic cohomology}\label{motivic}
In Section \ref{Deligne}, we constructed an element $\lambda$ that defines a class in $H^2(\Gamma(Y, 3))$ and its regulator is represented by the differential 1-form $\rho(\lambda)$. In this section, we construct an element in $H^2(\Gamma(C_K, 3))$, where $C_K = C \times_\Qbb K$ for a certain number field $K$. It gives rise to an element in motivic cohomology $H^2_\Mcal(C_K, \Qbb(3))$ via De Jeu's map $\beta_{C_K}$. Finally, we show that this motivic cohomology class descends to an element in $H^2_\Mcal(C, \Qbb(3))$.

Recall that the residues $u_p$ are trivial for all $p \notin S$, where $S$ is the closed subset of $C$ defined in \eqref{S}. As discussed in Remark \ref{4.7}, if $u_p$ vanish for all $p \in S$, $\lambda$ defines an element in $H^2(\Gamma(C, 3))$. When the residues are nontrivial, we modify $\lambda$ by its residues. This method is inspired by Bloch's trick (see e.g., \cite{Blo10}, \cite{Nek13}). Let $S'$ be the closed subset of  $S$ consisting of the points  $p$ such that $u_p \neq 0$. Let $K$ be the splitting field of $S'$ in $\Cbb$; this is the smallest Galois extension $K/\Qbb$ that contains all the residue fields $\Qbb(p)$ for $p \in S'$. For a geometric point $q: \Qbb(p) \hookrightarrow K$ over a point $p$ of $S'$, we define $u_q$ as the image of $u_p$ under the embedding $\Bcal(\Qbb(p))\xhookrightarrow{q} \Bcal(K)$. Then for $q \in S'(K)$, $u_q$ defines an element in the Bloch group $\Bcal(K)$. It is compatible with the Galois action, i.e., $\sigma(u_q) = u_{\sigma(q)}$ for $q\in S'(K)$ and $\sigma \in \Gal(K/\Qbb)$,
\begin{equation}\label{sigma(up)}
    \begin{tikzcd}
        \Bcal(\Qbb(p)) \arrow[r,hook, "q"] \arrow[rd, hook] & \Bcal(K) \arrow[d,"\sigma"]\\
      &\Bcal(K).
    \end{tikzcd}
\end{equation}
Denote by $K(C)$ the function field of $C\times_\Qbb K$. The inclusion $\Qbb(p) \xhookrightarrow{j} K(C)$ induces a map $B_i(\Qbb(C)) \xrightarrow{j} B_i(K(C))$, which is not an inclusion generally. We have the following commutative diagram 
\begin{equation*}
    \xymatrix{ B_3(\Qbb(C))  \ar[d] \ar[r] & B_2 (\Qbb(C))\otimes \Qbb(C)^\times_\Qbb \ar[r]^{\hspace{1cm}\alpha_3(2)} \ar[d]_{j }& \bigwedge^3 \Qbb(C)^\times_\Qbb \ar@{^{(}->}[d]
    \\B_3(K(C))  \ar[r] & B_2 (K(C))\otimes K(C)^\times_\Qbb \ar[r]^{\hspace{1cm}\alpha_3(2)} & \bigwedge^3 K(C)^\times_\Qbb.}
\end{equation*}
This implies a map $j : H^2(\Gamma(\Qbb(C), 3)) \to H^2(\Gamma(K(C), 3))$. By Lemma \ref{o}, $\lambda$ defines a class in $H^2(\Gamma(\Qbb(C), 3))$, hence $j(\lambda)$ defines a class in $H^2(\Gamma(K(C), 3))$. We have the following exact sequence 
\begin{center}\begin{tikzcd}[cramped,column sep=small]
 0 \to H^2(\Gamma(C_K, 3)) \arrow[r] & H^2(\Gamma(K(C), 3)) \arrow[rr,"\oplus \partial_q"] && \bigoplus_{q \in (C_K)^1} H^1(\Gamma(K, 2)),
\end{tikzcd}
\end{center}
where $(C_K)^1$ is the set of closed points of $C_K$. We have $\partial_q(j(\lambda)) = u_q$ for  $q \in S'_K = S'(K)$ and trivial otherwise. 

We assume that the difference of any two geometric points $q_1, q_2 \in S'(K)$ in the Jacobian of $C$ is torsion of order dividing a fixed integer $N$. Fix $\Ocal \in  S'(K)$. Then for any point $q \in S'(K) - \{\Ocal\}$, there is a rational function $f_q \in K(C)^\times$ such that 
\begin{equation}\label{f_p}
    \div (f_q) = N(\Ocal) - N(q)
\end{equation}
in $C_{K}$. And we set $f_\Ocal = 1$.

\begin{definition}\label{lambda'}
We set
\begin{equation}
    \lambda' := j(\lambda) + \sum_{q \in S'(K) - \{\Ocal\}} \dfrac{1}{N} (u_q\otimes f_q),
\end{equation}
which defines an element in $B_2(K(C)) \otimes K(C)_\Qbb^\times$.
\end{definition}

\begin{lm}\label{r}
The element $\lambda'$ defines a class in $H^2(\Gamma(K(C), 3))$. \end{lm}

\begin{proof} 
For $q\in S'(K)$, we recall  the following Goncharov's complex \eqref{ee}:
\begin{equation}
    \xymatrix{
    B_3(K(C))  \ar[r] & B_2 (K(C))\otimes K(C)^\times_\Qbb \ar[r]^{\hspace{1cm}\alpha_3(2)} \ar[d]_{\partial_q}& \bigwedge^3 K(C)^\times_\Qbb \ar[d]\\
    & B_2(K) \ar[r]^{\alpha_2(1)} & \bigwedge^2 K^\times_\Qbb.}
\end{equation}
As discussed above, $j(\lambda)$ defines a class in $H^2(\Gamma(K(C), 3))$, hence $\alpha_3(2)(j(\lambda)) = 0$. For $q \in S'(K)$, we have $\alpha_2(1)(u_q) = 0$ because $u_q \in \Bcal(K)$. We thus have $\alpha_3(2)  \left(u_q \otimes f_q\right)  = (\alpha_2(1) (u_q)) \wedge f_q  = 0$. This implies that 
$$\alpha_3(2) (\lambda') = \alpha_3(2)(j(\lambda)) + \dfrac{1}{N} \sum_{q \in S'(K)- \{\Ocal\}} \alpha_3(2)(u_q \otimes f_q) = 0,$$ hence $\lambda'$ defines an element in $H^2 (\Gamma(K(C), 3))$.  
\end{proof}

Notice that $\lambda'$ depends on the choice of rational function $f_q \in K(C)^\times$. However, the following lemma is sufficient for us.

\begin{lm}\label{h}
The image of $\lambda'$ under De Jeu's map  (\ref{beta})
\begin{equation}
    \beta_{K(C)} : H^2 (\Gamma(K(C), 3)) \to H^2_\Mcal (K(C), \Qbb(3)),
\end{equation}
does not depend on the choice of $f_q \in K(C)^\times$.
\end{lm}
\begin{proof}
Let $q \in S'(K)$. Let $f_q' \in K(C)^\times$ be another rational function such that $\div (f_q') = N(\Ocal)- N(q)$. Then $\div(f_q/f'_q) = 0$, hence $f_q/f'_q$ defines an element in a finite field extension of $K$, denoted by $L$. Then $u_q \otimes (f_q/f_q')$ defines an element in $B_2(L)\otimes L^\times$. In the proof of Lemma \ref{r}, we showed that $\alpha_3(2)(u_q \otimes f_q) = 0$, this implies that 
$$\alpha_3(2) (u_q \otimes (f_q/f_q')) = \alpha_3(2) (u_q \otimes f_q) - \alpha_3(2)(u_q \otimes f_q') = 0,$$
hence $u_q \otimes (f_q/f_q')$ defines a class in $H^2(\Gamma(L, 3))$. We consider De Jeu's map
\begin{equation*}
   \beta_L : H^2 (\Gamma(L, 3))  \to  K_4(L)_\Qbb.
\end{equation*}
By Borel's theorem, $K_4$ group of a number field is torsion, so $K_4(L)_\Qbb = 0$. This implies that the images of $u_q \otimes (f_q/f'_q)$ under the map $\beta_L$ in $K_4(L)_\Qbb$ all vanish. Hence the image of $\lambda'$ under De Jeu's map does not depend on the choice of $f_q$.
\end{proof}

\begin{lm}\label{l}
All the residues of $\lambda'$ 
in the following localization sequence vanish,
\begin{equation*}
     0 \to H^2(\Gamma(C_{K}, 3)) \to  H^2(\Gamma(K(C), 3)) \xrightarrow{ 2\partial} \bigoplus_{q \in (C_K)^1} H^1(\Gamma(K, 2)),
\end{equation*}
and thus $\lambda'$ defines a unique element $\lambda'_{C_K} \in H^2(\Gamma(C_K, 3))$.
\end{lm}

\begin{proof}
We have  $\partial_q(\lambda') = 0$ for all $q \notin S'(K)$.  For $q \in S'(K)$, we have 
\begin{equation*}
\begin{aligned}
      \partial_q (\lambda') &= u_q + \sum_{q' \in S'(K)-\{\Ocal\}} \dfrac{1}{N}\  \partial_{q} (u_{q'}\otimes f_{q'})\\
      &= u_q + \sum_{q' \in S'(K) - \{\Ocal\}} \dfrac{1}{N} \cdot v_q (f_{q'}) \cdot u_{q'}\\
      &= \begin{cases}
      u_q + \dfrac{1}{N}\cdot v_q(f_q) \cdot u_q = u_q-u_q = 0 &\text{if } q \neq \Ocal,\\
      u_\Ocal + \sum_{q' \in S'(K)-\{\Ocal\}} \dfrac{1}{N}\cdot v_\Ocal (f_{q'}) \cdot u_{q'} = \sum_{q' \in S'(K)} u_{q'} &\text{if } q = \Ocal.  
      \end{cases}\\
\end{aligned}
\end{equation*}
Now let $\pi_K : C_{K} \to \Spec K$ be the structure morphism and $i_K : (C_K)^1 \hookrightarrow C_{K}$ be the canonical embedding. We have the following commutative diagram (see diagram (\ref{dJdiagram})), where the two horizontal sequences are exact
\begin{equation*}
\xymatrix{
0\ar[r]& H^2(\Gamma(C_{K}, 3)) \ar[r] \ar[d]_{\beta_{C_K}} &  H^2(\Gamma(K(C), 3)) \ar[r]^{ 2 \partial \quad} \ar[d]_{\beta_{K(C)}} & \bigoplus_{q \in (C_K)^1} H^1(\Gamma(K, 2)) \ar[d]_{\simeq}^{\beta_{(2)}^1} &  \\ 
0 \ar[r] & H^2_\Mcal(C_{K}, \Qbb(3))\ar[r] & H^2_\Mcal(K(C), \Qbb(3)) \ar[r]^{\Res^\Mcal\quad}&
\bigoplus_{q \in (C_K)^1}  H^1_\Mcal(K, \Qbb(2)) \ar[r]^{\qquad (i_K)_*\ } \ar[rd]_{\sum }& H^3_\Mcal (C_{K}, \Qbb(3)) \ar[d]^{(\pi_K)_*} \\ & & & &H^1_\Mcal (K, \Qbb(2)),}
\end{equation*}
and $\Sigma$ is the trace map, which sends $(u_q)_{q\in S'(K)}$ to 
$\sum_{q \in S'(K)} u_q$. Then
we have  $\sum_{q\in S'(K)} u_q = 0$ by the commutativity of the bottom triangle. This shows that $\partial_q (\lambda') = 0$ for all $q \in S'(K)$, then $\lambda'$ defines a unique element in $H^2(\Gamma(C_{K}, 3))$. 
\end{proof}

By the previous lemma, we constructed an element $\lambda'_{C_K}\in H^2(\Gamma(C_K, 3))$. Via the map $$\beta_{K_C}: H^2(\Gamma(C_K, 3)) \to H^2_\Mcal(C_K, \Qbb(3)),$$
we obtain a class $\beta_{C_K}(\lambda'_{C_K}) \in H^2_\Mcal (C_{K}, \Qbb(3))$.

\begin{lm}\label{k}
The class $\beta_{C_K} (\lambda'_{C_K}) \in  H^2_\Mcal (C_{K}, \Qbb(3))$ is $\Gal(K/ \Qbb)$-invariant.
\end{lm}
\begin{proof}
The Galois action of $\Gal(K/\Qbb)$  on $H^2_\Mcal (C_{K}, \Qbb(3))$ is induced from the action on the function field, hence it is sufficient to check that $\beta_{K(C)}(\lambda') \in H^2(\Gamma(K(C), 3))$ is $\Gal(K/\Qbb)$-invariant. Let $\sigma \in \mathrm{Gal}(K/\Qbb)$, we have
\begin{equation*}
\begin{aligned}
    \sigma (\lambda') &= \sigma(j(\lambda)) + \sum_{q \in S'(K)-\{\Ocal\}} \dfrac{1}{N}\ \sigma(u_q \otimes f_q)
    = j(\lambda) + \sum_{q \in S'(K)-\{\Ocal\}} \dfrac{1}{N}\ u_{\sigma(q)} \otimes \sigma(f_q),
\end{aligned}
\end{equation*}
because $\lambda \in B_2(\Qbb(C)) \otimes \Qbb(C)^\times$ and $\sigma(u_q) = u_{\sigma(q)}$ (see diagram \eqref{sigma(up)}). Since $\div(f_q) = N (\Ocal) - N (q)$ for $q \in S'(K) - \{\Ocal\}$, we have $\div(\sigma(f_q)) = N (\sigma(\Ocal)) - N(\sigma(q)).$
And by definition of  $f_{\sigma(q)}$, we have $\div(f_{\sigma(q)}) = N(\Ocal) - N (\sigma(q))$. Hence 
\begin{equation}\label{125}
    \mathrm{div}(\sigma(f_q)) = \div (f_{\sigma(q)}) -N(\Ocal) + N (\sigma(\Ocal)) = \mathrm{div}(f_{\sigma(q)}) - \div(f_{\sigma(\Ocal)}) = \div(f_{\sigma(q)}/ f_{\sigma(\Ocal)}).
\end{equation} 
Write $\Ocal' = \sigma^{-1}(\Ocal) \in S'(K)$, we have 
\begin{equation*}
    \begin{aligned}
&\hspace{0.5cm}\lambda + \sum_{q \in S'(K) - \{\Ocal\}} \dfrac{1}{N} u_{\sigma(q)} \otimes \dfrac{f_{\sigma(q)}}{f_{\sigma(\Ocal)}}\\
&=\lambda + \sum_{q \in S'(K)-\{\Ocal, \Ocal'\}} \dfrac{1}{N} \ u_{\sigma(q)} \otimes \dfrac{f_{\sigma(q)}}{f_{\sigma(\Ocal)}} + \dfrac{1}{N}\ u_\Ocal \otimes \dfrac{f_{\Ocal}}{f_{\sigma(\Ocal)}}\\
&= \lambda + \sum_{q \in S'(K)-\{\Ocal, \Ocal'\}} \dfrac{1}{N} \ u_{\sigma(q)} \otimes \dfrac{f_{\sigma(q)}}{f_{\sigma(\Ocal)}} - \dfrac{1}{N}\ u_\Ocal \otimes f_{\sigma(\Ocal)}  \quad (\text{as } f_\Ocal = 1)\\
&= \lambda + \sum_{q \in S'(K)-\{\Ocal, \Ocal'\}} \dfrac{1}{N} \ u_{\sigma(q)} \otimes f_{\sigma(q)}  - \sum_{q \in S'(K) - \{\Ocal, \Ocal'\}} \dfrac{1}{N} \  u_{\sigma(q)} \otimes f_{\sigma(\Ocal)} - \dfrac{1}{N} \ u_\Ocal \otimes f_{\sigma(\Ocal)} \\
        &=  \lambda + \sum_{q \in S'(K)-\{\Ocal, \Ocal'\}} \dfrac{1}{N} \ u_{\sigma(q)} \otimes f_{\sigma(q)} - \sum_{q \in S'(K) - \{\Ocal\}} \dfrac{1}{N} \  u_{\sigma(q)} \otimes f_{\sigma(\Ocal)}\\
        &= \lambda + \sum_{q \in S'(K)-\{\Ocal, \Ocal'\}} \dfrac{1}{N} \ u_{\sigma(q)} \otimes f_{\sigma(q)} +  \dfrac{1}{N} \    u_{\sigma(\Ocal)} \otimes f_{\sigma(\Ocal)} \quad (\text{as } \sum_{q \in S'(K)} u_q = 0)\\
        &= \lambda + \sum_{q \in S'(K) - \{\Ocal'\}}\dfrac{1}{N} u_{\sigma(q)}\otimes f_{\sigma(q)}\\
        &= \lambda'.
    \end{aligned}
\end{equation*}
Hence by Lemma \ref{h}, $\beta_{K(C)}(\sigma(\lambda')) = \beta_{K(C)} (\lambda')$ for all $\sigma \in \Gal(K/\Qbb)$. Since De Jeu's map $\beta$ is functorial, it is compatible with the Galois action, so $\sigma(\beta_{K(C)}(\lambda')) =  \beta_{K(C)}(\lambda')$ for all $\sigma \in \Gal(K/
\Qbb)$.
\end{proof}

So that $\beta_{C_K}(\lambda'_{C_K})$ defines a class in  $H^2_\Mcal(C_{K},\Qbb(3))^{\Gal(K/\Qbb)}$. Denote by $\pi : C_K \to C$, we have the Galois descent of motivic cohomology as mentioned in \eqref{Galoisdescent}
\begin{equation}\label{descent}
    \pi^* : H^2_\Mcal (C, \Qbb(3)) \xrightarrow{\simeq} H^2_\Mcal (C_{K}, \Qbb(3))^{\mathrm{Gal}(K/\Qbb)}.
\end{equation}
Hence we have $(\pi^*)^{-1}(\beta_{C_K}(\lambda'_{K_C}))$ is an element in $H^2_\Mcal(C, \Qbb(3))$.

\subsection{Proof of Theorem \ref{0}}\label{relate}
In this section, we keep the notations as in Section \ref{motivic}. To prove Theorem \ref{0},  the main idea is that we relate the regulator of the motivic cohomolgy class constructed in Section \ref{motivic} to the Deligne cohomology class constructed in Section  \ref{Deligne}, hence to the Mahler measure of the polynomial $P$ by Section \ref{relate1}.

First, as mentioned in \eqref{resbloch}, $u_p$ defines an element in $\Bcal(\Qbb(p))$ for $p \in S$. By Remark \ref{4.7}, if $u_p = 0$ for all $p \in S$, then $\beta_C(\lambda_C) \in H^2_\Mcal(C, \Qbb(3))$ and $\rho(\lambda)$ represents the class $r_3(2)_C(\lambda_C) \in H^1(C(\Cbb), \Rbb(2))^+$. Denote by $ i: Y(\Cbb) \hookrightarrow C(\Cbb)$ the canonical embedding. Then by Lemma \ref{90}, we have
\begin{equation*}
\begin{aligned}
      \m (P) - \m(\tilde P) = - \dfrac{1}{8\pi^2} \left<[\partial \Gamma], [\rho(\lambda)]\right>_{Y(\Cbb)} &= -\dfrac{1}{8\pi^2} \left<[\partial\Gamma], i^*r_3(2)_C(\lambda_C) \right>_{Y(\Cbb)}\\
      &= -\dfrac{1}{8\pi^2} \left<i_*[\partial\Gamma], r_3(2)_C(\lambda_C) \right>_{C(\Cbb)}\\
      &= -\dfrac{1}{16\pi^2} \left<i_*[\partial\Gamma], \reg_C(\beta_C(\lambda_C))\right>_{C(\Cbb),}
\end{aligned}
\end{equation*}
where the last equality follows from the diagram \eqref{124}. Apply Beilinson's conjecture \eqref{BeiConj} to $\beta_C(\lambda_C) \in H^2_\Mcal(C, \Qbb(3))$ and $i_*[\partial \Gamma] \in H_1(C(\Cbb), \Qbb)^+$ (see Lemma \ref{90}), we thus have 
\begin{equation*}
    \m(P) - \m(\tilde P) = a \cdot L'(E, -1) \qquad (a \in \Qbb).
\end{equation*}

When the residues $u_p \in \Bcal(\Qbb(p))$ are nontrivial for some $p \in S' \subset S$, we define $u_q \in \Bcal(K)$ for $q\in S'(K)$, where $K$ is the splitting field of $S'$ in $\Cbb$ (see the beginning of Section \ref{motivic}). Let $K(C)$ denote the function field of $C_K$. Fix a point $\Ocal \in S'(K)$. Recall that we define the following element in $B_2(K(C)) \otimes K(C)^\times$ (see Definition \ref{lambda'})
\begin{equation*}
     \lambda' = j(\lambda) + \sum_{q \in S'(K) - \{\Ocal\}} \dfrac{1}{N} (u_q\otimes f_q),
\end{equation*}
where $f_q \in K(C)^\times$ is defined right before Definition \ref{lambda'}. We prove that $\lambda'$ defines a class in $H^2(\Gamma(K(C), 3))$ (see Lemma \ref{r}). The differential form $\rho(\lambda)$ represents the class $r_3(2)_{\Qbb(C)}(\lambda)$, we then have 
\begin{equation*}
    \m(P) - \m(\tilde P) = -\dfrac{1}{8\pi^2} \int_{\partial \Gamma} \rho(\lambda) =  -\dfrac{1}{8\pi^2} \int_{\partial \Gamma} r_3(2)_{\Qbb(C)} (\lambda) = -\dfrac{1}{8\pi^2} \int_{\partial \Gamma} r_3(2)_{K(C)} (j(\lambda)),
\end{equation*}
where  $j : \Qbb(C) \hookrightarrow K(C)$. Hence
\begin{equation}\label{ct}
\begin{aligned}
    \m(P) - \m(\tilde P) &=  -\dfrac{1}{8\pi^2} \int_{\partial \Gamma} r_3(2)_{K(C)}\left(\lambda' - \sum_{q\in S'(K)-\{\Ocal\}} \dfrac{1}{N} u_q \otimes f_q\right)\\
    &= -\dfrac{1}{8\pi^2} \int_{\partial \Gamma} r_3(2)_{K(C)}(\lambda') + \dfrac{1}{8N\pi^2} \sum_{q\in S'(K)- \{\Ocal\}} \int_{\partial \Gamma} r_3(2)_{K(C)} (u_q \otimes f_q)\\
    &= - \dfrac{1}{16\pi^2} \int_{\partial \Gamma} \reg_{K(C)} (\beta_{K(C)}(\lambda')) + \dfrac{1}{8N\pi^2} \sum_{q\in S'(K)- \{\Ocal\}} D(u_q) \int_{\partial \Gamma} d\arg(f_q),
\end{aligned}
\end{equation}
where the last equality follows from Lemma \ref{integral} and the fact that $\partial \Gamma$  is assumed to be contained in $Y(\Cbb)$ (see Lemma \ref{90}).
We have the following commutative diagram
\begin{equation*}
    \xymatrix{
    H^2(\Gamma(C_{K}, 3)) \ar@{^{(}->}[d] \ar[r]^{\beta_{C_K}} & H^2_\Mcal(C_{K}, \Qbb(3)) \ar[r]^{\reg_{C_K}\quad } \ar@{^{(}->}[d] & H^1(C_K(\Cbb), (\Rbb(2))^+ \ar@{^{(}->}[d] \\
    H^2(\Gamma(K(C), 3)) \ar[r]_{\beta_{K(C)}} & H^2_\Mcal(K(C), \Qbb(3)) \ar[r]_{\reg_{K(C)}} & H^1(K(C), \Rbb(2))^+.
    }
\end{equation*}
As $\lambda' \in H^2(\Gamma(K(C), 3))$ defines a unique a class $\lambda'_{C_K} \in H^2(\Gamma(C_K, 3))$ (see Lemma \ref{l}), we have  
\begin{equation*}
    \reg_{K(C)} (\beta_{K(C)}(\lambda')) = \reg_{C_K} (\beta_{C_K}(\lambda'_{C_K})).
\end{equation*}
Hence the first integral in \eqref{ct} can be written as the following pairing in de Rham cohomology 
\begin{equation}\label{pair}
    \left<[\partial \Gamma], (i_K)^*\reg_{C_K}(\beta_{C_K}(\lambda')) \right>_{Y_K(\Cbb)} = \left<(i_K)_*[\partial \Gamma], \reg_{C_K}(\beta_{C_K}(\lambda'_{C_K})) \right>_{C_K(\Cbb)},
\end{equation}
where $i_K : Y_K(\Cbb) \hookrightarrow C_K(\Cbb)$ is the canonical embedding. Moreover, by the functorial property of the Beilinson regulator map, we have the following commutative diagram
\begin{equation*}
  \xymatrix{
    H^2(\Gamma(C_{K}, 3)) \ar[r]^{\beta_{C_K}} &  H^2_\Mcal(C_{K}, \Qbb(3)) \ar[d]^{
    \reg_{C_{K}}} &  H^2_\Mcal (C, \Qbb(3)) \ar@{_{(}->}[l]_{\pi^*} \ar[d]^{\reg_C}\\
    &H^1(C_{K}(\Cbb), \Rbb(2))^+ \ar@{=}[r] &H^1(C(\Cbb), \Rbb(2))^+,
    }
\end{equation*}
where $\pi^*$ is induced from the isomorphism $\pi^* : H^2_\Mcal(C, \Qbb(3)) \xrightarrow{\simeq} H^2_\Mcal(C_K, \Qbb(3))^{\Gal(K/\Qbb)}$ mentioned in \eqref{descent}.  Hence by identifying de Rham cohomology as well as singular homology of $C_K(\Cbb)$ and $C(\Cbb)$, we have
\begin{equation*}
\begin{aligned}
    \left<(i_K)_*[\partial \Gamma], \reg_{C_K}(\beta_{C_K}(\lambda'_{C_K})) \right>_{C_K(\Cbb)} 
    &=\left<{i}_*[\partial \Gamma], \reg_C ((\pi^*)^{-1}(\beta_{C_K}(\lambda'_{C_K}))) \right>_{C(\Cbb)}.
    \end{aligned}
\end{equation*}
Applying Beilinson's conjecture to $(\pi^*)^{-1}(\beta_{C_K}(\lambda'_{C_K})) \in H^2_\Mcal(C, \Qbb(3))$, we obtain that
\begin{equation*}
    \dfrac{1}{(2\pi i)^2} \left<i_*[\partial \Gamma], \reg_C ((\pi^*)^{-1}(\beta_{C_K}(\lambda'_{C_K})))\right>_{C(\Cbb)} = a \cdot L'(E, -1), \quad a \in  \Qbb,
\end{equation*}
where $E$ is the Jacobian of $C$. From \eqref{ct}, by setting $b_q = \frac{1}{2\pi} \int_{\partial \Gamma } d \arg f_q$, we have 
\begin{equation}\label{mainid}
\begin{aligned}
    \m(P) - \m (\tilde P)
    &= a \cdot L'(E, -1) + \dfrac{1}{4N \pi} \sum_{q\in S'(K)\setminus \{\Ocal\}} b_q \cdot D(u_q), \quad a \in \Qbb.
\end{aligned}
\end{equation}
We will show that for $f \in \bar \Qbb(C)^\times$ and $\gamma : [0, 1] \to C(\Cbb)$ is a loop, $\int_{\gamma} d \arg f$ is a integral multiple of $2\pi$. In fact, we can always find a partition 
$$0=a_0 < a_1 < \cdots < a_{n-1} < a_n = 1,$$
such that $\gamma$ is the union of $\gamma_j : [a_j, a_{j+1}] \to C(\Cbb)$ for $j = 0, \dots, n-1$ and $\gamma_j ([a_j, a_{j+1}])$ is contained in a local coordinate chart of $C(\Cbb)$. Then 
\begin{equation*}
\begin{aligned}
     \int_\gamma d \arg f &= \sum_{j=0}^{n-1} \int_{\gamma_j} d \arg f \\
     &= \sum_{j=0}^{n-1} \arg f(\gamma_j(a_{j+1})) - \arg  f(\gamma_j(a_j))\\
     &= - \arg f(\gamma_0(0)) + \arg f(\gamma_{n-1}(1)) + \sum_{j=0}^{n-2} \arg f(\gamma_j(a_{j+1})) - \arg f(\gamma_{j+1}(a_{j+1}))\\
     &= 2\pi k,
\end{aligned}
\end{equation*}
for some integer $k$ since $\gamma_0(0) = \gamma(0) = \gamma(1) = \gamma_{n-1}(1)$ and $\gamma_j (a_{j+1}) = \gamma_{j+1}(a_{j+1})$ for $j=0, \dots, n-2$. In particular, we get $\int_{\partial \Gamma} d \arg f = 2 \pi k$, for some $k \in \Zbb$. This implies that $b_q \in \Zbb$ for all $q \in S'(K)$. Although the coefficients $b_q$ depend on the choice of $\Ocal$, the $D$-values in identity \eqref{mainid} do not. Indeed, if we remove $\Ocal$ from $S'(K)$, its complex conjugation $c (\Ocal) \in S'(K)$ maintains the $D$-values in identity \eqref{mainid} because $D(u_{c(\Ocal)}) = D(c(u_\Ocal)) = -D(u_\Ocal)$, where $c$ is the complex conjugation. \hfill  \qedsymbol

\begin{rmk}\label{6.22} 
\begin{enumerate}
\item [(a)] By Lemma \ref{11}, we have
\begin{equation*}
    \m(P) = \m(\tilde P) + a \cdot L'(E, -1) - \dfrac{1}{8N\pi^2} \sum_{q \in S'(K)- \{\Ocal\}} b_q \cdot \Res_q(\rho(\lambda)).
\end{equation*}

    \item [(b)] In some cases, $D$-values on Bloch group elements can relate to Dirichlet $L$-values. Let $\chi$ be a primitive Dirichlet character of conductor $f$, we have 
\begin{equation*}
    L(\chi, 2) = \dfrac{1}{G(\bar \chi)} \sum_{k=1}^f \bar\chi(k) \ Li_2(e^{2\pi i k/f}),
\end{equation*}
where $G(\bar \chi) = \sum_{k=1}^f \bar \chi(k) e^{2\pi i k/f}$ is the Gauss sum of $\chi$. In particular, when $\chi$ is odd quadratic then
\begin{equation*}
    L(\chi, 2) = \dfrac{1}{\sqrt{f}} \sum_{k=1}^f \chi(k) \ D(e^{2\pi i k/f}).
\end{equation*}
Then
\begin{equation*}
    L'(\chi,-1) = \dfrac{f^{3/2}}{4 \pi } L(\chi, 2) = \dfrac{f}{4\pi}  \sum_{k=1}^f \chi(k) \ D(e^{2\pi i k/f}).
\end{equation*}
In particular, let $\chi_{-3}$ (respectively $\chi_{-4}$) be the nontrivial character of modulo 3 (respectively modulo 4), we have
\begin{equation*}
    L'(\chi_{-3}, -1) = \dfrac{3}{2\pi} D(e^{2\pi i/3})  = \dfrac{1}{\pi} D(e^{i\pi/3}), \quad L'(\chi_{-4}, -1) = \dfrac{2}{\pi} D(e^{i \pi /2}).
\end{equation*}

\end{enumerate}
\end{rmk}

\section{Examples}\label{Examples}
In this section, we apply Theorem \ref{0} to several Mahler measure's identities. We also describe some polynomials to which our main theorem fails to apply. They are numerically conjectured by Boyd and Brunault. 

\subsection{Pure identities}\label{www} In this Section, we apply Theorem \ref{0} to study pure identities of Mahler's measure
\begin{equation*}
    \m (P) \sim_{\Qbb^\times} L'(E, -1),
\end{equation*}
where the notation $a \sim_{\Qbb^\times} b$ means $a/b \in \Qbb^\times$. Notice that most of polynomials in this section are of the form considered in Remark \ref{6.3} 
\begin{equation}\label{exact}
    P(x,y,z)=A(x) + B(x)y +C(x) z,
\end{equation} where $A, B, C$ are products of cyclotomic polynomials. In this case, $\m(\tilde P) = 0$ and $\m(P) \neq 0$. A typical example of pure identity is the Mahler measure of $z + (x+1)(y+1)$, which is conjectured by D. Boyd 
\begin{equation*}
    \m(z+(x+1)(y+1)) = -2 L'(E_{15},-1).
\end{equation*}
It was proved under Beilinson's conjecture and up to a rational factor by Lalín \cite[Section 4.1]{Lal15}. It is then completely proven by Brunault \cite{Bru23}. This polynomial also satisfies our main theorem, we do not discuss it here but focus on other examples. All figures in this section are generated by Maple software. 
\bigskip

a) We prove under Beilinson's conjecture the following pure identity
\begin{equation}\label{pure1}
    \m((1 + x)(1 + y)(x + y) + z) \stackrel{?}{\sim}_{\Qbb^\times} L'(E_{14},-1),
\end{equation}
which is the first identity mentioned in Table \ref{table:1}. In this case, $P$ is not of the form \eqref{exact}, but we still have $\m(\tilde P) = 0$ and the following decomposition
\begin{equation*}
    x \wedge y \wedge z = - x \wedge (1+x) \wedge y \ + \ y \wedge (1+y) \wedge x \ +\  \dfrac{y}{x} \wedge \left(1+\dfrac{y}{x}\right) \wedge x.
\end{equation*}
Hence 
\begin{equation*}
    f_1 = -g_2 = -g_3   = -x, \quad  f_2 = -g_1 = - y, \quad f_3 = -y/x.
\end{equation*}
We have that $W_P$ is given by
\begin{equation*}
    (xy+x+y)(1+x+y)((x + 1)y^2 + (x^2 + x + 1)y + x^2 + x) = 0,
\end{equation*}
which is the union of lines $L_1 : xy + x + y = 0$, $L_2 : 1 + x + y = 0$ and the curve $$C : (x + 1)y^2 + (x^2 + x + 1)y + x^2 + x = 0,$$ which is a nonsingular curve of genus 1.  By the following change of variables
\begin{equation*}
    x = -\dfrac{Y+X^2+1}{X(X - 1)}, \quad y = -\dfrac{Y}{X(X + 1)} - \dfrac{1}{X},
\end{equation*}
we obtain that the Jacobian of $C$ is given by
\begin{equation*}
   E/\Qbb :  Y^2 + XY + Y = X^3 - X,
\end{equation*}
which is an elliptic curve of type 
$14a4$. Its torsion subgroup is $\Zbb/6\Zbb = \left<A\right>$ with $A = (1,-2)$. With the help of Magma \cite{BCP}, we have 
\begin{align*}
    \div(x) = -(5A) +(A) -(4A) +(2A), \quad
    \div(y) = (\Ocal) +(A) - (4A) - (3A).
\end{align*}
Denote by $S$ the closed subscheme of $E$ consisting of all points in supports of above divisors. The values of $f_j$ and $f_j \circ \tau$ at $p \in S$ are either $0, 1$ or $\infty$ for all $j$, then the elements $v_p(g_j) \{f_j(p)\}_2$ and $v_p(g_j \circ \tau) \{f_j \circ \tau(p)\}_2$ are all trivial in $B_2(\Qbb)$ for all $j$ and $p\in S$.  Figure \ref{fig:1} describes the Deninger chain
$$\Gamma :  \{|x| = |y| = |(1+x)(1+y)(x+y)| \ge 1\},$$ and its boundary in polar coordinates $x = e^{it}$ and $y = e^{is}$ for $t, s\in [-\pi, \pi]$.
\begin{figure}[h!]
\centering\includegraphics[width=6cm]{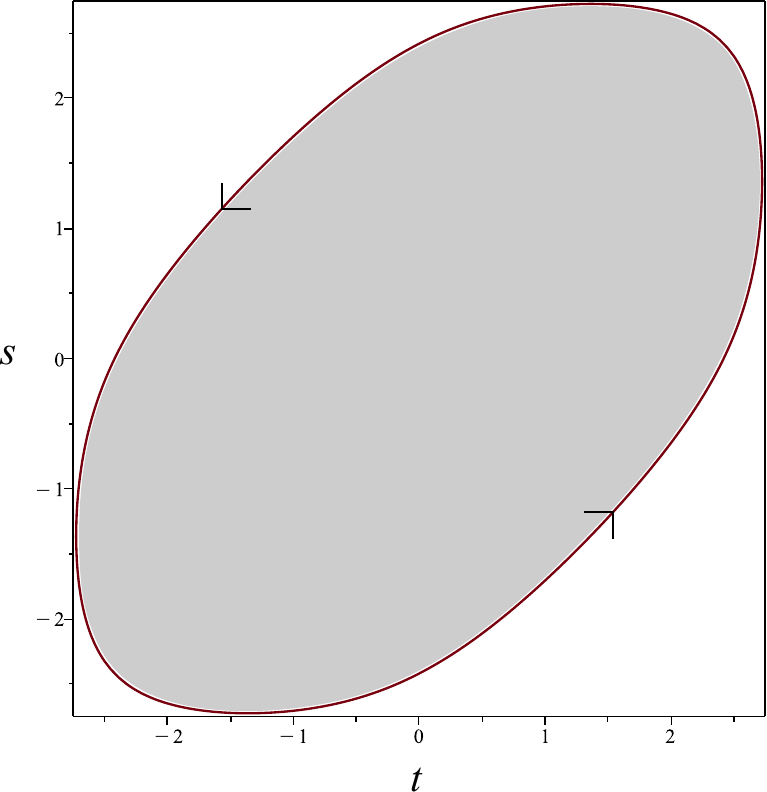}
    \caption{The Deninger chain $\Gamma$.}
    \label{fig:1}
\end{figure}
We obtain that $\partial \Gamma$ is contained completely in $C$ and $\partial \Gamma$ does not contain any zeros and poles of $f_j, 1- f_j, g_j$ for all $j$. Then by Theorem \ref{0}, we have identity \eqref{pure1} under Beilinson's conjecture.
\bigskip 

b) We study the pure identity (2) in Table \ref{table:1} 
\begin{equation}
    \m(1 + x + y + z + xy + xz + yz) \stackrel{?}{\sim}_{\Qbb^\times} L'(E_{14},-1).
\end{equation}
First we notice that
\begin{equation*}
    \m(1+ x + y + xy + z(1+ x + y )) = \m(1+x + y + z(1+x+y+xy)),
\end{equation*}
so it suffices to prove the following identity
\begin{equation}\label{pure8}
    \m(1+x + y + z(1+x+y+xy)) \stackrel{?}{\sim}_{\Qbb^\times} L'(E_{14},-1).
\end{equation}
We have $\m(\tilde P) = \m(1+x+y+xy)  = \m(x+1) \m(y+1) = 0$. We have the following decomposition 
\begin{align*}
    x \wedge y \wedge z &= x \wedge (1+x) \wedge y  \ - \  y \wedge (1+y) \wedge x \ + \  (x+y) \wedge (1+x+y) \wedge x \\
    &\quad  - \  (x+y) \wedge (1+x+y) \wedge y \ - \ \dfrac{x}{y} \wedge (1+ \dfrac{x}{y}) \wedge (1+x+y),
\end{align*}
so 
\begin{equation*}
    f_1 = -x, \quad f_2 = -y, \quad f_3 = f_4 = -(x+y), \quad f_5 = -x/y, \quad g_1 = g_4 = y, \quad g_2 = g_3 = x, \quad g_5 = 1+x+y.
\end{equation*}
We have that $W_P$ is given by $ x(x+1) y^2 + (x^2+x+1)y + x+1 = 0$, which is an irreducible nonsingular curve of genus 1. Using the following change of variables,
\begin{equation*}
    x = -\dfrac{Y+X^2 +1}{X(X - 1)}, \quad y = \dfrac{Y}{X(X+1)},
\end{equation*}
we obtain that the Jacobian of $W_P$ is given by  $E/\Qbb :  Y^2 + XY + Y = X^3 - X,$ which is the same elliptic curve in Example (a). We have
\begin{align*}
    &\div(x) = -(5A) + (A) - (4A) + (2A),  & &\div (1+x+y) = 2(\Ocal) - (5A) + 2(A) - (4A)  - (2A)  - (3A),\\
    &\div (y) = (\Ocal) + (5A) - (2A) - (3A),  & &\div(1+1/x+1/y) = -(\Ocal) + 2(4A) - (2A) + 2(3A) - (5A)  - (A).
\end{align*}
With the same reason as Section (a), we get identity \eqref{pure8} conditionally on Beilinson's conjecture. Moreover, as mentioned in the introduction, we have 
\begin{equation*}
    \m((1 + x)(1 + y)(x + y) + z) \stackrel{?}{\sim}_{\Qbb^\times} \m(1 + x + y + z + xy + xz + yz),
\end{equation*}
because they are rational multiples of the same elliptic curve $L$-value $L'(E_{14}, -1)$.
\bigskip 

c) Similarly, one gets identity (11) in Table \ref{table:1}
\begin{equation*}
     \m(1 + x + y + z + xy + xz + yz - xyz) \stackrel{?}{\sim}_{\Qbb^\times}L'(E_{36},-1).
\end{equation*}
This identity is interesting because this is the only example that has been found with CM elliptic curves.
\bigskip

d) The same arguments apply to all pure identities in Table \ref{table:1}, except for identities (5), (6), (7), (8), and (12).  In this section, we study the first four identities. It suffices to consider identity (5) since Lalín and Nair showed that the polynomials (5), (6), (7), and (8) share the same Mahler measure (see  \cite{LN23}). Let us recall identity (5) in Table \ref{table:1}
\begin{equation}\label{pure2}
    \m (1 + (x+1)y +(x-1)z) \stackrel{?}{\sim}_{\Qbb^\times} L'(E_{21}, -1).
\end{equation}
The polynomial $P = 1 + (x+1)y +(x-1)z$ is of the form \eqref{exact}. We have the following decomposition \begin{equation*}
    x \wedge y \wedge z = x \wedge (1-x) \wedge y + (x+1)y \wedge (1+(x+1)y) \wedge x - x \wedge (1+x) \wedge (1+(x+1)y),
\end{equation*}
so
\begin{equation*}
    f_1 = -f_3 = g_2 = x, \quad  f_2 = -(x+1)y, \quad g_1 = y, \quad g_3 = 1+(x+1)y,  \quad f_2 \circ \tau = -\dfrac{x+1}{xy}, \quad g_3 \circ \tau = \dfrac{xy + x + 1}{xy}.
\end{equation*}
We have that $W_P$ is given by $  x(x + 1)y^2 + (2x^2 + x + 2)y + 1 + x=0,$
which is a non-singular curve of genus 1. Using the following change of variables
\begin{equation*}
    x = -\dfrac{X^2 - 6X + 3Y}{X(X - 6)}, \quad y  = \dfrac{Y - 3X - 3}{X(X + 1)},
\end{equation*} we get an equation for the Jacobian of $W_P$
\begin{equation*}
   E/\Qbb : Y^2 - 3XY - 3Y = X^3 - 5X^2 - 6X,
\end{equation*}
which is an elliptic curve of type $21a1$. Its torsion subgroup is $\Zbb/2\Zbb \times \Zbb/4\Zbb = \left<A\right> \times \left<B\right>$ with 
$A = (-1, 0)$, and $B = (0, 0)$. With the help of Magma \cite{BCP}, we have
\begin{align*}
    &\div(x) = -(A+B) + (A + 3B) - (3B) + (B), & &\div(1+(x+1)y) = 2(2B) - (3B) - (B),\\
    &\div(y) =  (\Ocal) + (A+B) - (B) - (A), &
    &\div(xy+x+1) = (\Ocal) + 2(A+2B) - (A+B) - (3B) - (A).
\end{align*}
Let $S$ be the closed subscheme of $E$ consisting all the points in the supports of the above divisors. We have
\begin{align*}
    \sum_j v_B(g_j)\{f_j(B)\}_2 + v_B(g_j \circ \tau) \{f_j \circ \tau(B)\}_2
    &= v_B(g_2) \{f_2(B)\}_2 + v_B(g_2 \circ \tau) \{f_2 \circ \tau (B)\}_2 \\
    &= \{\infty\}_2 - \{1/2\}_2 = \{2\}_2,
\end{align*}
which is nontrivial in $B_2(\Qbb)$. Therefore, the theorem of Lalín mentioned in the introduction does not apply to this example. As $S$ consists of points in $E_\mathrm{tors}$, we can choose $N$ in Theorem \ref{0} equal  to $\# E_\mathrm{tors} = 8$. Since the points of $S$ have rational coordinates and the $f_i$ have rational coefficients, then $f_i$ take rational values on $S$. Therefore, the Bloch-Wigner dilogarithmic values in identity (\ref{mainid}) all vanish. Figure \ref{fig:3} describes the Deninger chain and its boundary in polar coordinates $x=  e^{it}$, $y = e^{it}$ for $s, t \in [-\pi,\pi]$.
\begin{figure}[h!]
    \centering
\includegraphics[width=6cm]{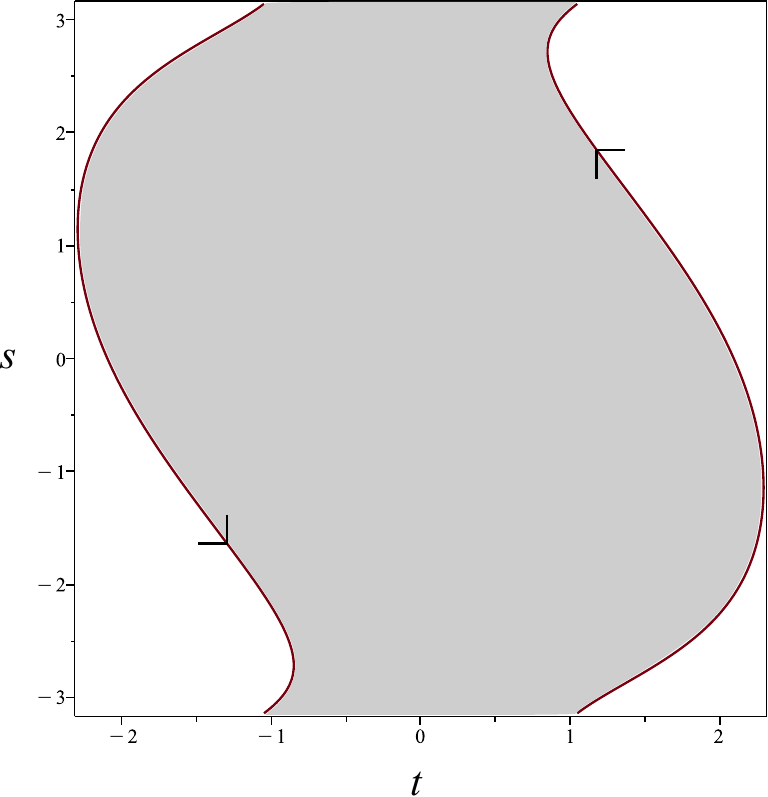}
    \caption{The Deninger chain $\Gamma$.}
    \label{fig:3}
\end{figure}
The boundary $\partial \Gamma$ consists of 2 loops and does not contain any zeros and poles of $f_j, 1-f_j, g_j, f_j \circ \tau, 1-f_j \circ \tau, g_j \circ \tau$ for all $j$. Hence by Theorem \ref{0}, we get identity (\ref{pure2}) conditionally on Beilinson's conjecture. In particular, under Beilinson's conjecture, we have
\begin{equation*}
     \m(1 + (x + 1)y + (x - 1)z) \stackrel{?}{\sim}_{\Qbb^\times} \m((x+1)^2(y+1) + z),
\end{equation*}
as they are rational multiples of $L'(E_{21}, -1)$.
\bigskip

e) There is an interesting remark on identities (4) and (10) of Table \ref{table:1}.  By some trivial change of variables, we obtain 
\begin{equation*}
    \m((x+1)^2+(1-x)(y+z)) = \m((x+1)(y+1) +(x-1)^2 z). 
\end{equation*}
Theorem \ref{0} applies to $P = (x+1)^2+(1-x)(y+z)$ but not to $P = (x+1)(y+1) +(x-1)^2 z$. Indeed, in the latter case, $W_P$ is given by 
\begin{equation*}
    (-x^3 - 2x^2 - x)y^2 + (x^4 - 6x^3 + 2x^2 - 6x + 1)y - x^3 - 2x^2 - x = 0,
\end{equation*}
which is a curve having a singularity at $(1, -1)$, and the boundary $\partial \Gamma$ passes this singular point. Figure \ref{fig:11} describes the Deninger chain $\Gamma$ and its boundary $\partial \Gamma$ in polar coordinates $x=e^{it}, y= e^{is}$ for $t, s \in [-\pi, \pi]$, where the marked points indicate the singular point $(1, -1)$.  Using Magma \cite{BCP}, one can check that $\partial \Gamma$ is no longer a loop  on the normalization of $W_P$.
\begin{figure}[h!]
    \centering
    \includegraphics[width=6cm]{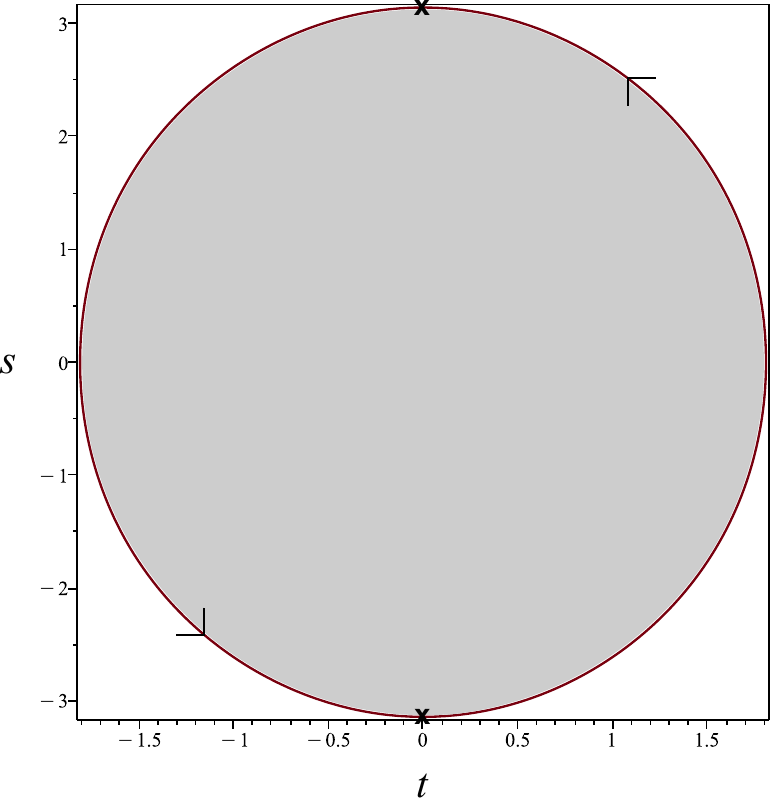}
    \caption{The Deninger chain $\Gamma$.}
    \label{fig:11}
\end{figure}

\noindent
The same situation happens with identity (10) in Table \ref{table:1}. By some trivial changes of variables, we have
\begin{equation*}
    \m((1+x)^2+y+z) = \m(1+y+(1+x)^2z),
\end{equation*}
where Theorem \ref{0} applies to the first polynomial but not the second one.
\bigskip 

f) We study identity (12) in Table \ref{table:1}
\begin{equation}\label{391}
     \m(1 + xy + (1 + x + y)z) \stackrel{?}{\sim}_{\Qbb^\times} L'(E_{90},-1).
\end{equation}
We have the following decomposition 
\begin{equation*}
    \begin{aligned}
        x \wedge y \wedge z =  xy \wedge (1+xy) \wedge x - (x+y) \wedge (1+x+y) \wedge x &+ (x+y) \wedge (1+x+y) \wedge y \\
        &+ \dfrac{-x}{y} \wedge \left(1 + \dfrac{x}{y}\right)\wedge (1 +x+y),
    \end{aligned}
\end{equation*}
so
\begin{equation*}
    f_1 = -xy, \quad f_2 = f_3 = -(x+y), \quad f_4  = -x/y, 
\end{equation*}
\begin{equation*}
    g_1 = g_2 = x, \quad g_3 = y, \quad  g_4 = 1+x+y.
\end{equation*}
The curve $W_P$ is given by 
\begin{equation*}
    (-x^2 + x + 1)y^2 + (x^2 + x + 1)y + x^2 + x - 1 = 0,
\end{equation*}
which is an irreducible curve of genus 1 and does not contain any rational points. Figure \ref{fig9.4} describes the Deninger chain and its boundary in polar coordinates. We find that $\partial \Gamma$ does not contain any singular points of $W_P$. 
\begin{figure}[h!]
    \centering
    \includegraphics[width=6cm]{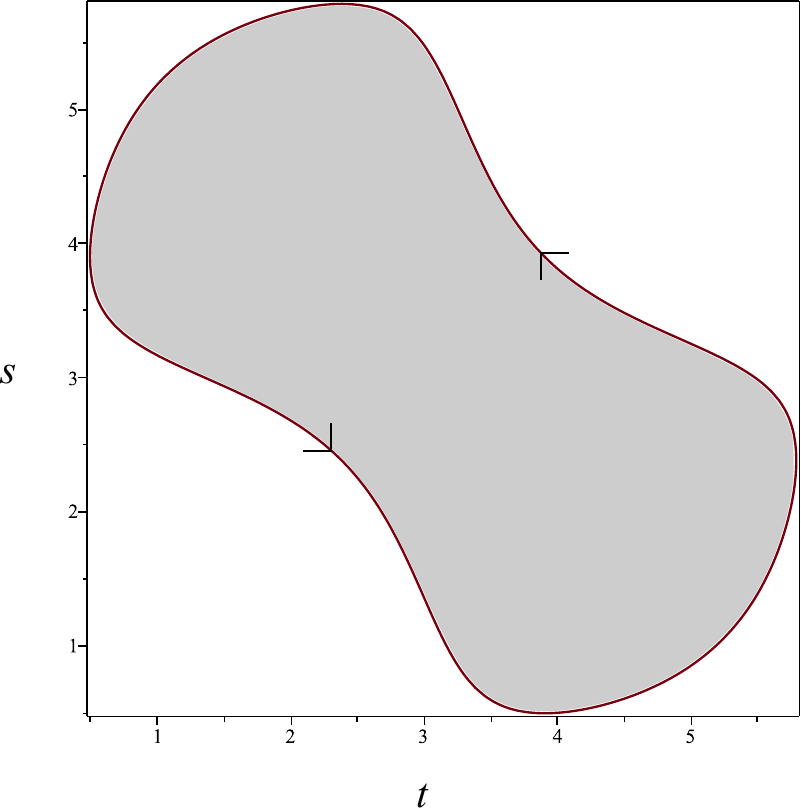}
    \caption{ The Deninger chain $\Gamma$ corresponding to \eqref{391}.}
    \label{fig9.4}
\end{figure}

By Magma \cite{BCP}, we obtain that the Jacobian of $C$ is given by
\begin{equation*}
    E/\Qbb : Y^2 + XY + Y = X^3 - X^2 - 8X + 11,
\end{equation*}
which is an elliptic curve of type $90b1$. Its torsion subgroup is $\Zbb/6\Zbb = \left<A\right>$, with $A = (3,1)$. Denote by $K$ the real quadratic field $\Qbb(\alpha)$ with $\alpha \in \Rbb$ such that $\alpha^2 + \alpha - 1 = 0$. Let $B_1 = (6\alpha+9, -24\alpha - 35)$, $B_2 = (-4\alpha +1, 12\alpha - 3)$, $B_3  = (\frac{9}{5}, \frac{24\alpha - 23}{25})$, $B_4 = (2, -\alpha -2)$, $B_5 = (-6\alpha + 3, -18\alpha+7)$, $B_6  = (4\alpha + 5, 8\alpha + 9)$ and we denote by $(B_i)$ the divisor in $E$ corresponding to the point $B_i$. We have the following divisors in $E/K$
\begin{align*}
    &\div(x) = (4A) + (B_1) - (A) - (B_2), \\
    &\div(y) = (\Ocal) + (B_3) - (B_4) -(3A), \\
    &\div(1+x+y) = 2(2A) +2(B_5) -(A) -(B_4) -(3A) -(B_2), \\
    &\div(1+1/x+1/y) = -(\Ocal) +2(5A) +2(B_6) -(B_3) -(4A) -(B_1).
\end{align*}
Note that the Bloch group of real quadratic fields is trivial after tensoring with $\Qbb$. Therefore, the residues $u_{B_i}$ are trivial because they are elements of $\Bcal(\Qbb(\alpha))$, the Bloch group (tensored with $\Qbb$) of the real quadratic field $\Qbb(\alpha)$. The remaining points in the supports of the above divisors are of the form $mA$ for $m = 1, \dots, 6$. Hence we can choose $N$ in Theorem \ref{0} as the order of $A$ which equals 6. As the points $mA$ have rational coordinates and the functions $f_i$ have rational coefficients, the Bloch-Wigner dilogarithmic values at $u_{mA}$ in \eqref{mainid} all vanish.   We then get identity \eqref{391} under Beilinson's conjecture for genus 1 curves \ref{BeiConj}. 
\bigskip

(g) Theorem \ref{0} does not imply identity \eqref{counterex}
\begin{equation*}
    \m((1 + x)(1 + y) + (1 - x- y)z) \stackrel{?}{\sim}_{\Qbb^\times} L'(E_{450},-1).
\end{equation*}
We have the following decomposition
\begin{align*}
    x \wedge y \wedge z &= - x \wedge (1+x) \wedge y + y \wedge (1+y) \wedge x + (x+y) \wedge (1-x-y) \wedge x  \\
    &\quad - (x+y) \wedge (1-x-y) \wedge y  - \dfrac{x}{y} \wedge \left(1+\dfrac{x}{y}\right) \wedge (1 -x-y).
\end{align*}
Therefore, we have 
\begin{equation*}
    f_1 = -x, \quad f_2 = -y, \quad f_3 = f_4 = x+y, \quad f_5 = -x/y
\end{equation*}
\begin{equation*}
    g_1 = g_4 =  y, \quad g_2 = g_3 = x, \quad g_5 = 1-x-y.
\end{equation*}
The Maillot variety is given by  
\begin{equation*}
    (x^2 + 3x)y^2 + (3x^2 + x + 3)y + 3x + 1 = 0.
\end{equation*}
By the following change of variables 
\begin{equation*}
    x = -\dfrac{3}{(X^2 - 9X)Y} + \dfrac{(-X^2 + 3X - 27)}{X^2 - 9X}, \quad
 y = \dfrac{3}{(X^2 - 27X + 162)Y} + \dfrac{(2X^2 - 21X - 27)}{X^2 - 27X + 162},
\end{equation*}
we get the following elliptic curve
\begin{equation*}
     E/\Qbb : \quad  Y^2 + XY = X^3 - X^2 - 27X + 81,
\end{equation*}
which is of conductor 450. Its torsion group is isomorphic to $\Zbb/2$ and generated by $A = (-6 , 3)$. We have the following divisors in $E$
\begin{equation*}
\begin{aligned}
    &\div (x) = -(P_1) + (P_2) -(P_3) + (P_4),\\
    &\div (y) = (P_3) + (P_5) - (P_6) - (P_2),\\
    &\div (1-x-y) =  2 (\Ocal)  + 2 (P_7) - (P_6) - (P_1) - (P_2) - (P_3),\\
    &\div(1-1/x-1/y) = 2 (P_8) + 2(A) - (P_2) - (P_3) - (P_4) - (P_5),
\end{aligned}
\end{equation*}
where  
\begin{equation*}
    \begin{aligned}
        &P_1 : = (9 , 18), & &P_2 := (9 , -27),& & P_3 := (0 , 9), &&
        P_4 := (0 , -9),\\
        &P_5 := \left(-\dfrac{9}{4} , -\dfrac{81}{8}\right), & & P_6:= (18, 63) &&
        P_7 := (4, 3), & & P_8 := (3, -6). &&
    \end{aligned}
\end{equation*}
We have that the residue 
\begin{equation*}   
    u_{P_1} := \sum_{j=1}^5 v_{P_1} (g_j) \{f_j(P_1)\}_2  + v_{P_1}(g_j \circ \tau) \{f_j \circ \tau(P_1)\}_2 = -3 \{3\}_2
\end{equation*}
is nontrivial in the Bloch group $\Bcal(\Qbb)$. Since $P_1$ is torsion-free, it violates the finite order condition of Theorem \ref{0}.

\subsection{Identities with Dirichlet characters}\label{ly} In this section, we study Mahler measure identities of the form 
\begin{equation*}
    \m(P) \stackrel{?}{=}  a \cdot L'(E, -1) + \sum_\chi b_\chi \cdot L'(\chi, -1),
\end{equation*}
where $a \in \Qbb, b_\chi \in \Qbb^\times $, $E$ is an elliptic curve and $\chi$ are odd quadratic Dirichlet characters.
\bigskip

a) We prove the first identity of Table \eqref{mix1} conditionally on Beilinson's conjecture \ref{BeiConj}. The polynomial $P = 1 + (x^2 - x + 1)y + (x^2 + x + 1)z$ is of the form \eqref{exact}. We have the following decomposition on $V_P$
\begin{equation*}
\begin{aligned}
     x \wedge y \wedge z  &= -\frac{1}{3}x^3 \wedge (1-x^3) \wedge y \  +\  x \wedge (1-x) \wedge y \ +\  (x^2 - x +1)y \wedge (1+ (x^2-x+1)y) \wedge x \\
     &\quad - \frac{1}{3}x^3 \wedge (1+ x^3) \wedge (1+(x^2-x+1)y) \ + \ x \wedge (1+x) \wedge (1+(x^2-x+1)y).
\end{aligned}
\end{equation*}
We have 
\begin{equation*}
    f_1 = x^3, \ f_2 = x, \ f_3 = -(x^2 - x + 1)y, \ f_4 = -x^3, f_5 = -x,
\end{equation*}
\begin{equation*}
    g_1 = g_2 = y, \ g_3 = x, \ g_4 = g_5 = 1 + (x^2 - x +1)y.
\end{equation*}
The curve $W_P$ is given by $ x^2(x^2 - x + 1) y^2 - x(4x^2 -x + 4) y + x^2 - x + 1 = 0,$
which is a non-singular curve of genus 1 and does not contain any rational point. By the change of variables $ x = X, y = Y/X,$ we get a new equation
\begin{equation*}
     (X^2 - X + 1) Y^2 - (4X^2 - X + 4)Y + X^2 - X + 1 = 0.
\end{equation*}
Using Pari/GP \cite{PARI}, one gets the following Weierstrass form for the Jacobian of $W_P$
\begin{equation*}
    E/\Qbb : v^2 + uv = u^3 - u^2 -45u -104,
\end{equation*}
which is an elliptic of type $45a2$.  We denote by $k = \Qbb(\alpha)$ with $\alpha^2 - \alpha +1 = 0$. A base change of $E$ over $k$ can be given by 
\begin{equation*}
   E_k : V^2 + 3UV + 3V = U^3 - U^2 - 9U,
\end{equation*}
by using the following change of variables
\begin{align*}
    x &= \dfrac{(2-\alpha )V + \alpha U^2  - 3(\alpha -1)U + 3}{U^2 + (4\alpha - 2)U - 3\alpha}, \\
    y &= \dfrac{((1-\alpha)U^2 -(\alpha + 4)U + (\alpha + 4))V + 2\alpha U^3 - (8\alpha - 3)U^2 + (3\alpha - 12)U + 12}{U^4 - (4\alpha + 1)U^3 + (15\alpha - 7)U^2 - (17\alpha - 13)U + 6(\alpha - 1)}.
\end{align*}
The torsion subgroup of $E_k$ is $\Zbb/2\Zbb \times \Zbb /4\Zbb = \left<A\right> \times \left<B\right>$ with  $ A = (-3,3)$ and $B = (0,0)$. Let $K$ be the number field $\Qbb (\alpha, r, s)$ with 
\begin{equation*}
    r^2  - 2(2\alpha - 1)r + 3(\alpha - 1) = 0, \text{ and } s^2 + 2(2\alpha - 1)s - 3\alpha = 0.
\end{equation*}
We set $P_1 = (r, \alpha r - 2 \alpha -2), P_2 = (s, (1-\alpha)s + 2 \alpha - 4)$, which are points in $E(K)$. We denote by $(P_i)$ the divisor corresponding to $P_i$ in $E_k$. Note that the divisors $(P_i)$ have degree 2 on $E_k$. Using Magma, one obtain the following divisors in $E_k$
\begin{align*}
    &\div (g_3) = \div(x) = (P_1)- (P_2),\\
    &\div(g_1) = \div(g_2) = \div(y) = (\Ocal) + (A+3B) -(A+B) + (P_2) - (2B) - (P_1) ,\\
    &\div (g_4)=\div(g_5) = \div(1+(x^2-x+1)y) = 2(3B) +2(A) - (P_1) -(P_2) ,\\
    &\div(g_4 \circ \tau) = \div(g_5 \circ \tau) = \div(1+(1/x^2-1/x+1)(1/y)) = 2(B) + 2(A+2B) - (P_1) - (P_2). 
\end{align*}
The values of $f_j$ and $f_j \circ \tau$ at $P_1, P_2$ and their conjugates are either 0 or $\infty$, so we are only concerned with the other points. We obtain that
\begin{align*}
& u_A & & = v_A(g_4) \{f_4(A)\}_2 + v_A(g_5) \{f_5(A)\}_2 = \{-1\}_2 + \{1/ \alpha\}_2 = -\{\alpha\}_2,\\
\\
&u_B  &  &= v_B(g_4 \circ \tau) \{f_4\circ \tau (B)\}_2 + v_B(g_5\circ \tau) \{f_5\circ \tau (B)\}_2 =   2\{-1\}_2 +2\{\alpha\}_2 = 2\{\alpha\}_2,\\
\\
&u_{2B} & &= v_{2B}(g_1)\{f_1(2B)\}_2 + v_{2B}(g_1\circ \tau) \{f_1 \circ \tau (2B)\}_2 + v_{2B}(g_2)\{f_2(2B)\}_2 + v_{2B}(g_2\circ \tau) \{f_2 \circ \tau (2B)\}_2  \\
   & & &= -\{-1\}_2 + \{-1\}_2 - \{1/\alpha\}_2 + \{\alpha\}_2 = 2 \{\alpha\}_2,\\
   \\
&u_{3B}  & &= v_{3B}(g_4)\{f_4(3B)\}_2 + v_{3B}(g_5) \{f_5(3B)\}_2 = 2\{-1\}_2 + 2 \{\alpha\}_2 = 2 \{\alpha\}_2,\\
\\
&u_{A+B} & &= v_{A+B}(g_1) \{f_1(A+B)\}_2 + v_{A+B} (g_1 \circ \tau) \{f_1 \circ \tau (A+B)\}_2\\    
   & & & \quad + v_{A+B}(g_2) \{f_2(A+B)\}_2 + v_{A+B} (g_2 \circ \tau) \{f_2 \circ \tau (A+B)\}_2 \\
   & & &= -\{-1\}_2 + \{-1\}_2 - \{\alpha\}_2 + \{1/\alpha\}_2 = -2 \{\alpha\}_2,\\
   \\
&u_{A+2B} & &= v_{A+2B}(g_4 \circ \tau) \{f_4 \circ \tau(A+2B)\}_2 = 2 \{-1\}_2 + 2 \{1/\alpha\}_2  = -2\{\alpha\}_2,\\
\\
&u_{A+3B} & &=  v_{A+3B}(g_1)\{f_1(A+3B)\}_2 + v_{A+3B}(g_1\circ \tau) \{f_1 \circ \tau (A+3B)\}_2\\
    & & &\quad +  v_{A+3B}(g_2)\{f_2(A+3B)\}_2 + v_{A+3B}(g_2\circ \tau) \{f_2 \circ \tau (A+3B)\}_2\\
    & & &= \{-1\}_2 - \{-1\}_2 + \{1/\alpha\}_2 - \{\alpha\}_2 = -2 \{\alpha\}_2,
\end{align*}
which are all nontrivial in $B_2(K)$. Notice that $P_1, P_2$ have order 8 in $E(K)$ and all the other points belong to the torsion subgroup of $E_k$, whose cardinality equals 8, hence we choose $N$ in Theorem \ref{0} equal to 8.  Figure \ref{fig:4} indicates the Deninger chain (the shaded region) and its boundary in polar coordinates $x=e^{it}, y= e^{is}$ for $s, t \in [-\pi, \pi]$.
\begin{figure}[h!]
    \centering
    \includegraphics[width=6cm]{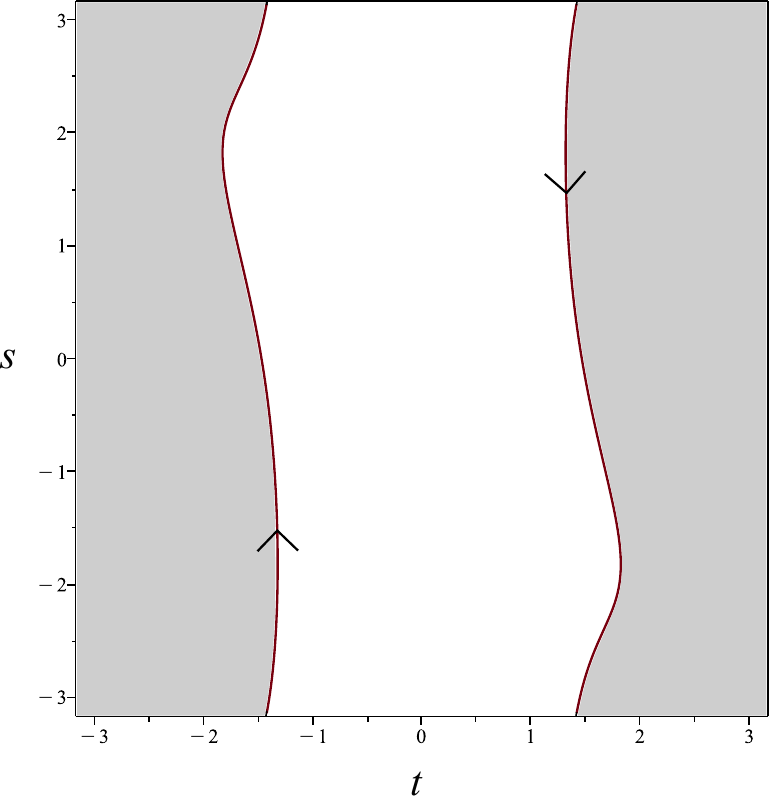}
    \caption{The Deninger chain $\Gamma$.}
    \label{fig:4}
\end{figure}
The boundary $\partial \Gamma$ consists of 2 loops, which do not contain any points in the supports of the above divisors. By equation \eqref{4.9}, we have
$\m(P) = -\frac{1}{8\pi^2} \int_{\partial \Gamma} \rho(\lambda),$ so $\partial \Gamma$ must be nontrivial as otherwise $m(P)$ vanishes. Hence $\partial \Gamma$ defines a generator of $H_1(C(\Cbb), \Zbb)^+$. Then by Theorem \ref{0}, under Belinson's conjecture, we have
\begin{equation*}
    \m(1+ (x^2 - x + 1)y + (x^2 + x + 1)z) \stackrel{?}{=} a \cdot L'(E_{45}, -1) +  \dfrac{b}{32\pi} \cdot D(\alpha), \ a \in 
    \Qbb^\times, b \in \Zbb\setminus \{0\}.
\end{equation*}
We are unable to determine the coefficient $b$ as computing the integrals $\int_{\partial \Gamma} \ d \arg f_p$ for $p \in S$ is difficult.  By Remark \ref{6.22}, we have 
\begin{equation*}
    D(\alpha) = \dfrac{3\sqrt{3}}{4} L(\chi_{-3}, 2) = \pi L'(\chi_{-3}, -1).
\end{equation*}
Finally, we get 
\begin{equation*}
    \m(1+ (x^2 - x + 1)y + (x^2 + x + 1)z) \stackrel{?}{=} a \cdot L'(E_{45}, -1) + \dfrac{b}{32} \cdot L'(\chi_{-3}, -1), \ a \in 
    \Qbb^\times, b \in \Zbb\setminus \{0\}.
\end{equation*}

b) Using a method of Lalín \cite[Section 4.2]{Lal15}, we prove without assuming Beilinson's conjecture identity (6) of Table \ref{table:3}, which involves only the $L$-function of the Dirichlet character $\chi_{-4}$ 
\begin{equation*}
    \m(x^2 + 1 + (x+1)^2y + (x-1)^2z) = 2 L'(\chi_{-4}, -1).
\end{equation*}
We have $\m(\tilde P) = 0$. We have the following decomposition on $V_P$
\begin{align*}
    x \wedge y \wedge z &= -\dfrac{1}{2} x^2 \wedge (1+x^2) \wedge y + 2 x \wedge (1-x) \wedge y + x \wedge \dfrac{(x+1)^2y}{x^2 +1} \wedge \left(1 + \dfrac{(x+1)^2 y}{x^2+1}\right)\\
    &\quad - 2x \wedge (1+x) \wedge \left(1 + \dfrac{(x+1)^2 y}{x^2+1}\right) + \dfrac{1}{2} x^2 \wedge (1+x^2) \wedge \left(1 + \dfrac{(x+1)^2 y}{x^2+1}\right).
\end{align*}
We have 
\begin{align*}
    \rho(\xi)  &= -\dfrac{1}{2}\rho(-x^2, y) + 2\rho(x, y) + \rho\left(\dfrac{-(x+1)^2y}{x^2+1}, x\right) - 2 \rho\left(-x, 1 + \dfrac{(x+1)^2 y}{x^2+1}\right) + \dfrac{1}{2} \rho\left(-x^2, 1 + \dfrac{(x+1)^2 y}{x^2+1}\right),
\end{align*}
where
\begin{equation*}
    \rho(f, g) =  -D(f) d \arg g + \dfrac{1}{3} \log |g| (\log |1-f| d \log|f| - \log |f| d\log |1-f|).
\end{equation*}
We have that $W_P$ is given by 
\begin{equation*}
    (x^2 + 1)((x + 1)^2y^2 + (x^2 + 8x + 1)y + (x+1)^2) = 0,
\end{equation*}
which is the union of $L : x^2 + 1 = 0$ and the curve $C : (x + 1)^2y^2 + (x^2 + 8x + 1)y + (x+1)^2 = 0$. The figure below describes the Deninger chain $\Gamma$ in polar coordinates
\begin{equation*}
    \Gamma : \left|\dfrac{x^2+ 1+(x+1)^2y}{(x-1)^2}\right| \ge 1, x = e^{it}, y= e^{is}, s, t\in [-\pi, \pi].
\end{equation*}

\begin{figure}[h!]
    \centering
    \includegraphics[width=10cm]{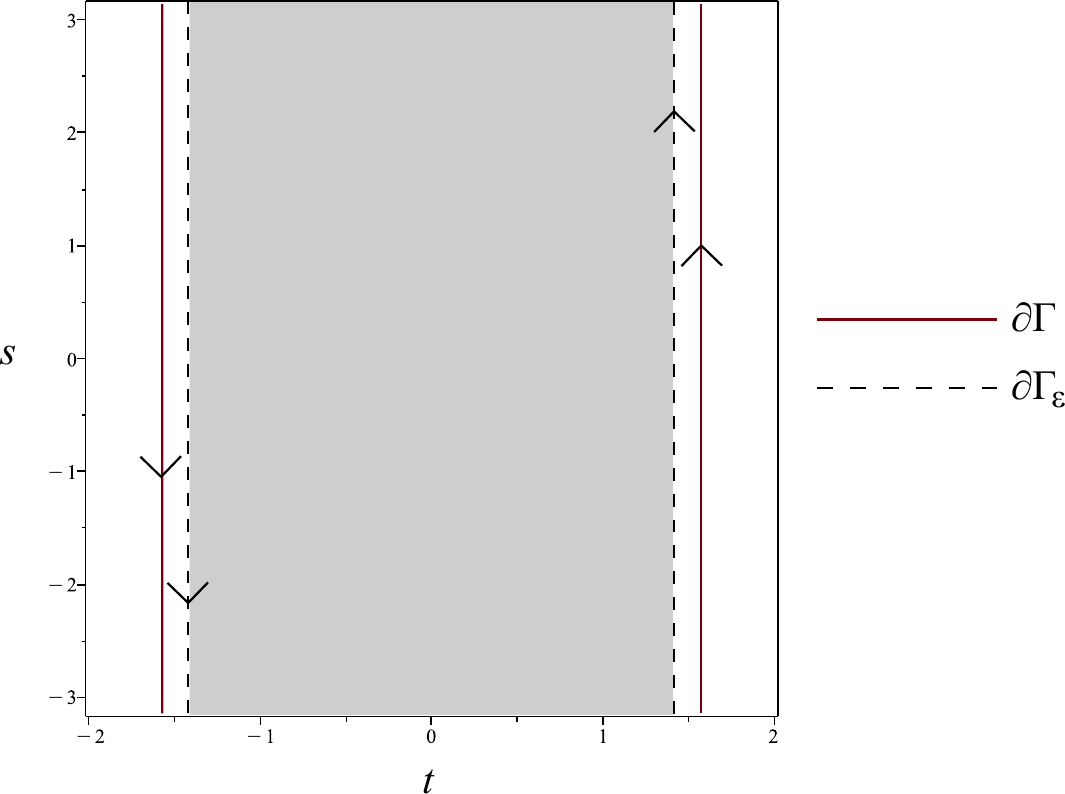}
    \caption{The integration domain.}
    \label{fig:5}
\end{figure}

\noindent Its boundary $\partial \Gamma$ consists of 2 loops 
$\gamma = \{t = \pi/2, -\pi \le s \le \pi\} \text{ and } \delta = \{t = -\pi/2, -\pi \le s \le \pi\}$ (with orientations as shown in the figure), which are contained in $L$. As $\partial \Gamma$ contains poles of $\rho(\xi)$, we do not have (\ref{good}) directly. We adjust the Deninger chain as follows, for $\varepsilon > 0$
\begin{equation*}
    \Gamma_\varepsilon :   \left|\dfrac{x^2+ 1+(x+1)^2y}{(x-1)^2}\right| \ge 1, x =  e^{i(1+\varepsilon)t}, y=e^{is}, \text{ for } s, t \in [-\pi, \pi],
\end{equation*}
which is the shaded region in Figure \ref{fig:5} with the boundary $\partial \Gamma_\varepsilon = \gamma_\varepsilon \cup \delta_\varepsilon$, where
\begin{equation*}
    \gamma_\varepsilon = \{t = \dfrac{\pi}{ 2(1+\varepsilon)}, -\pi \le s \le \pi\}, \quad \delta_\varepsilon = \{t = -\dfrac{\pi}{2(1 + \varepsilon)}, -\pi \le s \le \pi\}.
\end{equation*} 
We consider the differential forms $\eta$ and $\rho(\lambda)$ defined in \eqref{84} and Definition \ref{xiandxi*}, respectively. We have
\begin{equation}\label{150}
\int_{\Gamma_\varepsilon} \eta = \int_{\partial \Gamma_\varepsilon} \rho(\xi) = \dfrac{1}{2} \int_{\partial \Gamma_\varepsilon} \rho(\lambda),
\end{equation}
where the first equality is obtained by using Stokes's theorem and the second equality can be proved similarly as the proof of Lemma \ref{90}. As $\rho(\lambda)$ is a closed differential form, we can take the limit of equation (\ref{150}) as $\varepsilon \to 0$ without changing the value of the integration, and so that
\begin{equation*}
    \m(P) =   -\dfrac{1}{4\pi^2} \lim_{\varepsilon \to 0} \int_{\partial \Gamma_\varepsilon} 
     \rho(\xi).
\end{equation*}
We have
\begin{equation*}
\begin{aligned}
 \int_{\partial \Gamma_\varepsilon} \rho(\xi) 
    &= 
 \int_{\partial \Gamma_\varepsilon} -\dfrac{1}{2}\rho(-x^2, y) + 2\rho(x, y) + \rho\left(\dfrac{-(x+1)^2y}{x^2+1}, x\right) \\
    &\hspace{2cm} - 2 \rho\left(-x, 1 + \dfrac{(x+1)^2 y}{x^2+1}\right) + \dfrac{1}{2} \rho\left(-x^2, 1 + \dfrac{(x+1)^2 y}{x^2+1}\right)\\
    &= \int_{\gamma_\varepsilon \cup \delta_\varepsilon} 2 \rho(x, y) - 2 \rho\left(-x, 1 + \dfrac{(x+1)^2y}{x^2+1}\right) \\
    &= \int_{\gamma_\varepsilon \cup \delta_\varepsilon} -2 D(x) d \arg(y) + 2 D(-x) d\arg \left( 1 + \dfrac{(x+1)^2y}{x^2+1}\right)  \\
    &=  \left(-2 D\left(e^{\frac{i\pi }{2(1+\varepsilon)}}\right) \int_{\gamma_\varepsilon} d\arg (y) - 2 D\left(e^{-\frac{i\pi }{2(1+\varepsilon)}}\right) \int_{\delta_\varepsilon} d\arg (y) \right)\\
    &\quad + \left( 2 D\left(-e^{\frac{i\pi}{2(1+\varepsilon)}}\right) 
    \int_{\gamma_\varepsilon}  d\arg \left( 1 + \dfrac{(x+1)^2y}{x^2+1}\right)\right) + 2 D\left(-e^{-\frac{i\pi}{2(1+\varepsilon)}}\right) 
    \int_{\delta_\varepsilon}  d\arg \left( 1 + \dfrac{(x+1)^2y}{x^2+1}\right). \\
\end{aligned}
\end{equation*}
We have
 \begin{align*}
\int_{\gamma_\varepsilon} d\arg (y) = \int_{-\pi}^\pi ds = 2\pi, \quad \int_{\delta_\varepsilon} d \arg y  = \int_{\pi}^{-\pi} ds = -2 \pi. 
 \end{align*}
And 
 \begin{align*}
 \int_{\gamma_\varepsilon} d \arg \left(1+ \dfrac{(x+1)^2y}{x^2+1}\right)   = 2\pi, \quad   \int_{\delta_\varepsilon} d \arg \left(1+ \dfrac{(x+1)^2y}{x^2+1}\right)   = -2\pi,
 \end{align*}
by looking at Figure \ref{argument} and the fact that $\left|\frac{(x+1)^2}{x^2+1}\right| >1$.
\begin{figure}[h!]
    \centering
    \includegraphics[width=7.5cm]{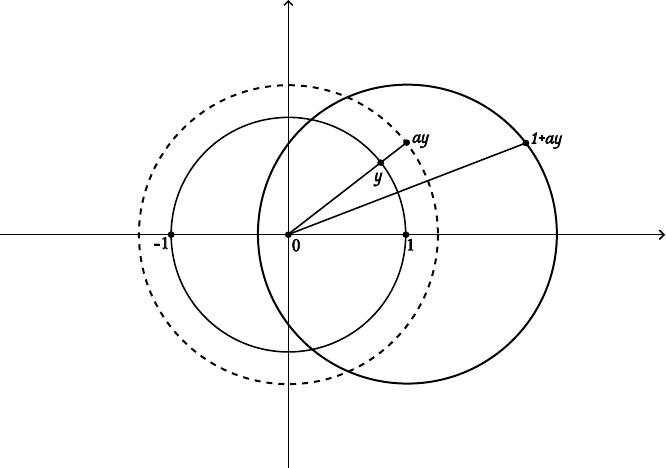}
    \caption{The argument of $1+ay$ with $|a| >1$.}
    \label{argument}
\end{figure}
Then we have  $ \lim_{\varepsilon \to 0} \int_{\partial \Gamma_\varepsilon} \rho(\xi) = -16 \pi D(e^{i\pi/2})$.
 Hence 
\begin{align*}
    \m(P) = \dfrac{4}{\pi}D(e^{i\pi/2})
     = 2L'(\chi_{-4},-1).
\end{align*}
The same arguments apply to identities (4), (5), (7), and (8) of Table \ref{table:3}.
\bigskip

c) Let us study identity (1) of Table \ref{table:3}, which involves only the $L$-function of the Dirichlet character $\chi_{-3}$
\begin{equation*}
    \m(1+ (x+1)(x^2+x+1)y +(x+1)^3 z) = 3 L'(\chi_{-3},-1).
\end{equation*}
We have that $W_P$ is given by $ (x^2 + x + 1)((x^4 + x^3)y^2 + (-2x^3 - 5x^2 - 2x)y + x + 1) = 0$, which is the union of the line $L : x^2 + x + 1 = 0$ and the curve $C : (x^4 + x^3)y^2 + (-2x^3 - 5x^2 - 2x)y + x + 1 = 0$. Figure \ref{fig:7} describes the Deninger chain in local coordinates $x= e^{it}$ and $y = e^{is}$ for $s, t \in [-\pi, \pi]$.  The boundary $\partial \Gamma = \gamma \cup \delta$, where
\begin{equation*}
    \gamma = \{t = 2\pi/3, -\pi \le s \le \pi\} \text { and } \delta = \{t = -2\pi/3, -\pi \le s \le \pi\},
\end{equation*}
which are both contained in $L$.
\begin{figure}[h!]
    \centering
    \includegraphics[width=10cm]{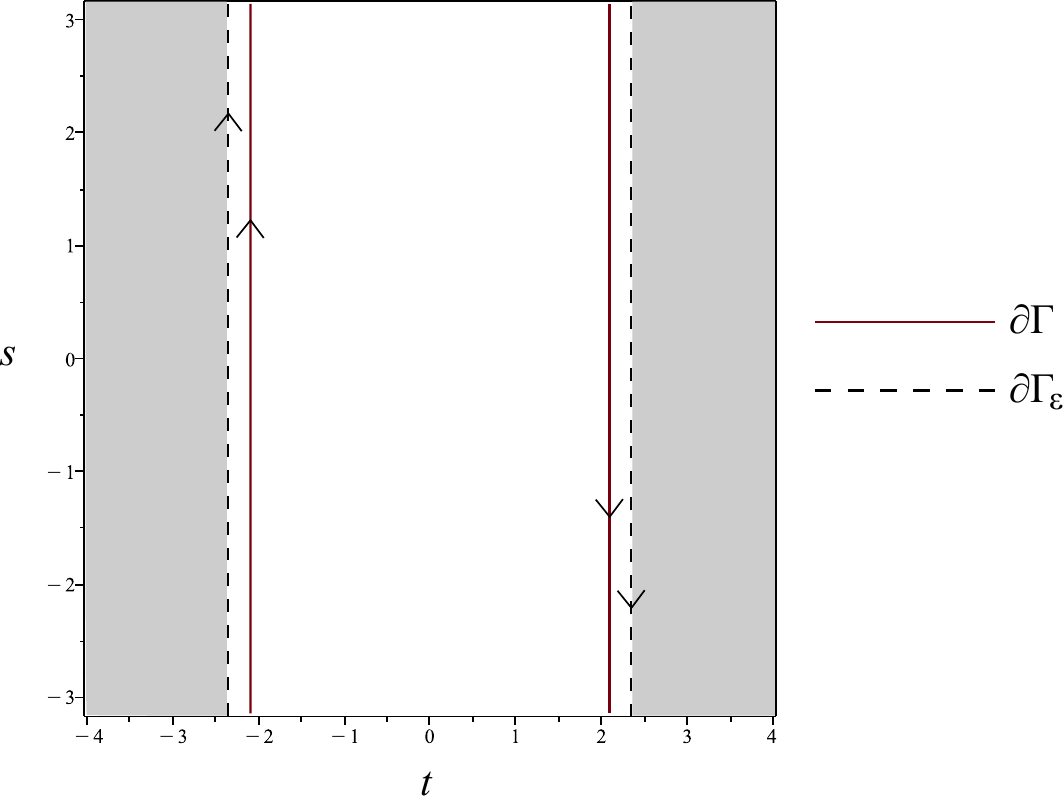}
    \caption{The integration domain.}
    \label{fig:7}
\end{figure}
The differential form $\rho(\xi)$ is again not well-defined on $\partial \Gamma$. So we adjust the Deninger chain to get $\Gamma_\varepsilon$ (see the shaded region). By similar computation as in Example (b), we have
\begin{equation*}
    \m(P) = -\dfrac{1}{4\pi^2} \lim_{\varepsilon \to 0} \int_{\partial \Gamma_\varepsilon} \rho(\xi)  = 3 L'(\chi_{-3},-1).
\end{equation*}
One can do similarly with identities (2) and (3) of Table \ref{table:3}.
\bigskip

d)  Theorem \ref{0} does not apply to identity \eqref{mix3}
\begin{equation*}
    \m(x^2 + x + 1 + (x^2 + x + 1)y + (x - 1)^2) \stackrel{?}{=} -\dfrac{1}{12} L'(E_{72},-1) + \dfrac{3}{2} L'(\chi_3, -1),
\end{equation*}
because the boundary $\partial \Gamma$ passes the singular point $(1, -1)$ of $W_P$ (see Figure \ref{fig:9}) and $\partial \Gamma$ is no longer a loop in the normalization of $W_P$.
\begin{figure}[h!]
    \centering
    \includegraphics[width=6cm]{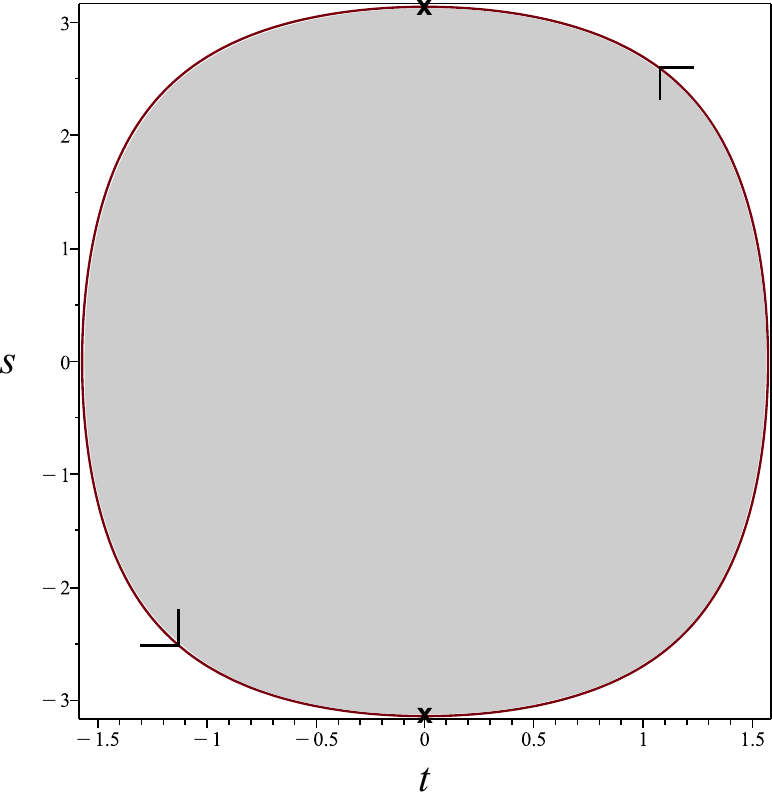}
    \caption{The Deninger chain $\Gamma$.}
    \label{fig:9}
\end{figure}

\bigskip

e) We prove the second identity of Table \ref{mix1}, under Beilinson's conjecture for genus 1 curves \ref{BeiConj}
\begin{equation}\label{395}
    \m(x^2+1 +(x+1)^2 y + (x^2 -1)z) \stackrel{?}{=}  -\dfrac{1}{10} L'(E_{48},-1) + L'(\chi_{-4}, -1),
\end{equation}
We have
\begin{align*}
    x \wedge y \wedge z &= -\dfrac{1}{2} x^2 \wedge (1+x^2) \wedge y \ + \ \dfrac{1}{2} x^2 \wedge (1-x^2) \wedge y \ +  \ \dfrac{(x+1)^2y}{x^2+1} \wedge\left(1 + \dfrac{(x+1)^2y}{x^2+1}\right) \wedge  x\\
    &\quad -2x \wedge (1+x) \wedge \left(1 + \dfrac{(x+1)^2y}{x^2+1}\right) \  + \ \dfrac{1}{2} x^2 \wedge (1+x^2) \wedge\left(1 + \dfrac{(x+1)^2y}{x^2+1}\right).
\end{align*}
Then
\begin{equation*}
\begin{aligned}
      \rho(\xi) = -\dfrac{1}{2} \rho(-x^2, y) + \dfrac{1}{2} \rho(x^2, y) + &\rho\left(-\dfrac{(x+1)^2y}{x^2+1}, x\right) \\
    &- 2 \rho \left(-x, 1+ \dfrac{(x+1)^2y}{x^2+1}\right) + \dfrac{1}{2}  \rho \left(-x^2, 1+ \dfrac{(x+1)^2y}{x^2+1}\right).
\end{aligned}
\end{equation*}
The $W_P$ is given by 
\begin{equation*}
    (x^2+1)((x+1)^2 y^2 + (3x^2 +4x +3)y + (x+1)^2) = 0,
\end{equation*}
which is the union of  $L : x^2 + 1 = 0$ and the curve $C : (x+1)^2 y^2 + (3x^2 +4x +3)y + (x+1)^2 =0,$ which is a nonsingular curve of genus 1. Figure \ref{fig:2} describes the Deninger chain $\Gamma$ and its boundary $\partial \Gamma$ in polar coordinates $x=e^{it}$ and $y = e^{is}$ for  $t, s \in [-\pi, \pi]$.
\begin{figure}[h!]
    \centering
\includegraphics[width=10cm]{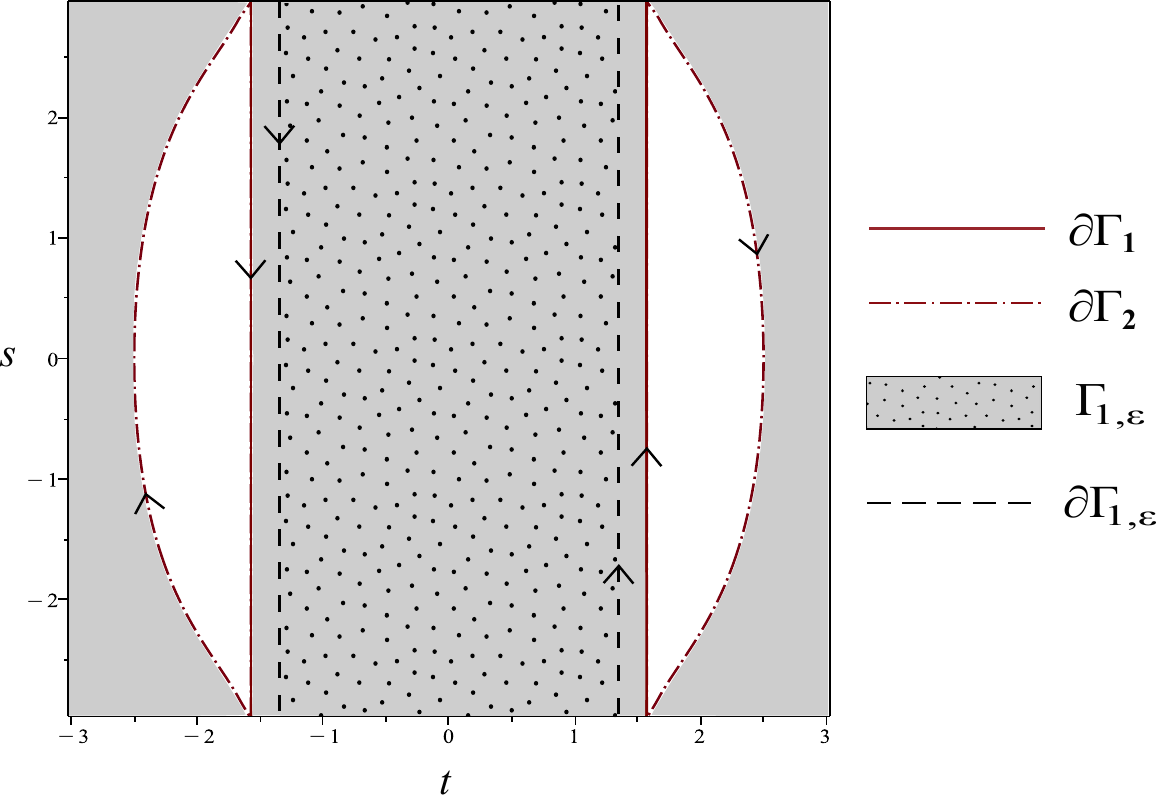}
    \caption{ The Deninger chain $\Gamma$ corresponding to \eqref{395}.}
    \label{fig:2}
\end{figure}
We have $\Gamma = \Gamma_1 \cup \Gamma_2$, where $\Gamma_1$ is the shaded region in the center with the boundary 
\begin{equation*}
    \partial \Gamma_1 = \{t = - \pi/2, -\pi\le s\le \pi\} \cup \{t = \pi/2, -\pi\le s\le \pi\}, 
\end{equation*}
and $\Gamma_2$ is the shaded region with the boundary $\partial \Gamma_2$ as in the figure. We observe that $\partial \Gamma_1$ is contained in $L$ and  $\partial \Gamma_2$ is contained in $C$. We have
\begin{align*}
    \m(P) = m_1 + m_2,
\end{align*}
where $m_1$ can be computed by the same method as the example (b). Let $\Gamma_{1, \epsilon}$ be the adjustment of $\Gamma_1$ as in Figure \ref{fig:2}. We have 
\begin{align*}
    m_1 =  -\dfrac{1}{4 \pi^2} \int_{\Gamma_1} \eta = -\dfrac{1}{4 \pi^2} \lim_{\varepsilon\to 0} \int_{\Gamma_{1,\varepsilon}} \eta = -\dfrac{1}{4 \pi^2} \lim_{\varepsilon \to 0} \int_{\partial \Gamma_{1, \varepsilon}} \rho(\xi) = L'(\chi_{-4},-1),
\end{align*}
and 
\begin{align*}
    m_2 =  - \dfrac{1}{4\pi^2} \int_{\Gamma_2} \eta = - \dfrac{1}{4\pi^2} \int_{\partial\Gamma_2} \rho(\xi).
\end{align*}
By the following change of variables
\begin{equation*}
    x= -\dfrac{2Y + X^2}{X^2 - 2X - 4}, \quad 
 y = -\dfrac{2}{X + 2},
\end{equation*}
the Jacobian of $C$ is given by 
\begin{equation*}
    E/ \Qbb : Y^2 = X^3 + X^2 - 4X - 4,
\end{equation*}
which the elliptic curve of type $48a1$. Its torsion subgroup is $\Zbb/2\Zbb \times \Zbb/2 \Zbb = \left<A\right> \times \left<B \right>$, where $A = (2, 0)$ and $B = (-1,0)$. Set $K = \Qbb(\alpha, \beta)$ where $\alpha^2 - 2 \alpha -4 = 0$ and $\beta^2 +4  = 0$. Let us write
\begin{equation*}
    P_1 = (\alpha, \alpha+2), P_2 = (\alpha, -\alpha -2), P_3 = (0, s, 1).
\end{equation*}
We have 
\begin{align*}
    &\div(x)  = -(P_1) + (P_2),
    &\div\left(\frac{1+(1+x)^2y}{x^2+1}\right) &= 2(A) + 2(A+B) -2(P_3),\\
    &\div(y) =  2(\Ocal) - 2(A+B), &\div\left(\frac{x^2y + x^2 + 2x + y + 1}{y(x^2 + 1)}\right) &= 2(\Ocal) +2(B) -2(P_3).
\end{align*}
We have $u_A = u_B = u_{A+B} = 0$ and $ u_{P_3} = -2 \{-\beta/2\}_2 -2 \{\beta/2\}_2 = 0.$  The residues $u_{P_i}$ for $i=1,2$ are trivial because they belong to $\Bcal (\Qbb(\alpha))$, the Bloch group (tensored with $\Qbb$) of the real quadratic field $\Qbb(\alpha)$.  Therefore, we have the following identity under Beilinson's conjecture \ref{BeiConj}
\begin{equation*}
    m_2 = a \cdot  L'(E_{48}, -1), \quad a \in \Qbb^\times.
\end{equation*}
In conclusion, we obtain the following identity under Beilinson's conjecture \ref{BeiConj}
\begin{equation*}
    \m(P) = L'(\chi_{-4}, -1) + m_2 = a \cdot L'(E_{48}, -1) + L'(\chi_{-4}, -1), \quad a \in \Qbb^\times.
\end{equation*}
\newpage
\fontsize{10pt}{10pt}\selectfont
\printbibliography

\end{document}